\documentclass{article}
\usepackage[utf8]{inputenc}
\usepackage[T1]{fontenc}
\usepackage[left=2.5cm,right=2.5cm,top=2.5cm,bottom=2.5cm]{geometry}

\usepackage{amsmath,amsthm,amscd,amssymb,mathrsfs,setspace}
\usepackage{amsfonts}
\usepackage{mathtools}
\usepackage{latexsym,epsf,epsfig}
\usepackage{color}
\usepackage[colorlinks]{hyperref}
\usepackage[nameinlink]{cleveref}
\usepackage{pgfplots}
\usepackage{authblk}
\usepackage[mathscr]{euscript}
\usepackage{dutchcal}
\usepackage{algorithm}
\usepackage{algorithmicx}
\usepackage{algpseudocode}
\usepackage{verbatim}
\usepackage{siunitx}
\usepackage{booktabs}
\usepackage{subcaption}
\newtheorem{theorem}{Theorem}[section]
\newtheorem{proposition}[theorem]{Proposition}
\newtheorem{lemma}[theorem]{Lemma}

\theoremstyle{definition}
\newtheorem{definition}[theorem]{Definition}
\newtheorem{assumption}[theorem]{Assumption}
\newtheorem{example}[theorem]{Example}

\theoremstyle{remark}
\newtheorem{remark}[theorem]{Remark}

\crefname{theorem}{Theorem}{Theorems}
\Crefname{theorem}{Theorem}{Theorems}
\crefname{assumption}{Assumption}{Assumptions}
\Crefname{assumption}{Assumption}{Assumptions}
\crefname{lemma}{Lemma}{Lemmas}
\Crefname{lemma}{Lemma}{Lemmas}
\crefname{definition}{Definition}{Definitions}
\Crefname{definition}{Definition}{Definitions}
\crefname{proposition}{Proposition}{Propositions}
\Crefname{proposition}{Proposition}{Propositions}
\crefname{algorithm}{Algorithm}{Algorithms}
\Crefname{algorithm}{Algorithm}{Algorithms}
\crefname{section}{Section}{Sections}
\Crefname{section}{Section}{Sections}
\crefname{appendix}{Appendix}{Appendices}
\Crefname{appendix}{Appendix}{Appendices}

\DeclareMathOperator*{\argmin}{arg\,min}

\DeclareMathOperator*{\dist}{dist}

\DeclareMathOperator*{\cvxflag}{cvx\_flag}
\newcommand{\N}{\mathbb{N}}
\newcommand{\R}{\mathbb{R}}

\newcommand{\Ha}{\mathcal{H}}

\newcommand{\st}{\operatorname{subject\ to}}
\DeclareMathOperator*{\TV}{TV}

\DeclareMathOperator*{\pred}{pred}
\DeclareMathOperator*{\ared}{ared}

\DeclareMathOperator*{\BV}{BV}

\DeclareMathOperator*{\supp}{supp}
\DeclareMathOperator*{\dvg}{div}

\DeclareMathOperator*{\diam}{diam}

\newcommand*\dd{\mathop{}\!\mathrm{d}}

\newcommand{\weakto}{\rightharpoonup}

\newcommand{\bdvg}[1]{\dvg{}_{{\hspace*{-2pt}#1\hspace*{2pt}}}}
\usepackage{tikz}
\usepackage{pgfplots}
\pgfplotsset{compat=1.7}
\usetikzlibrary{arrows.meta}
\usepackage{wasysym}
\usepackage[numbers,sort&compress]{natbib}
\usetikzlibrary{shapes,backgrounds}
\newcommand{\ds}{\, \textup{d} s }
\newcommand{\dx}{\, \textup{d} x}
\newcommand{\dxs}{\, \textup{d}x\textup{d}s}
\newcommand{\intO}{\int_{\Omega}}

\newcommand{\Ltwo}{L^2(\Omega)}
\newcommand{\Linf}{L^\infty(\Omega)}
\newcommand{\Hone}{H^1(\Omega)}
\newcommand{\Hneg}{H^{1}(\Omega)^*}

\newcommand{\ud}{u_{\textup{d}}}

\newcommand{\calM}{\mathcal{M}}
\newcommand{\calU}{\mathcal{U}}
\newcommand{\Wad}{W_{\textup{ad}}}

\def\tw{\tilde{w}}

\def\ow{\overline{w}}
\def\uw{\underline{w}}

\newcommand{\Wadzero}{W_{\textup{ad}}^0}
\newcommand{\Wadeps}{W_{\textup{ad}}^{\varepsilon}}
\newcommand{\Rgrad}{\nabla_{{\hspace*{-2pt}R\hspace*{2pt}}}}
\newcommand{\symdiff}{\mathbin{\triangle}}

\numberwithin{equation}{section}

\usepackage[normalem]{ulem}

\title{Homotopy trust-region method for phase-field
	approximations in perimeter-regularized
	binary optimal control\footnote{Paul Manns acknowledges funding by Deutsche Forschungsgemeinschaft
	through research grant MA 10080/2-1.}}

\author[1]{Paul Manns}
\author[2]{Vanja Nikoli\'{c}}
\affil[1]{Faculty of Mathematics, TU Dortmund University, 44227 Dortmund, Germany, \textit{paul.manns@tu-dortmund.de}}
\affil[2]{Department of Mathematics, Radboud University, 6525 AJ Nijmegen, The Netherlands, \textit{vanja.nikolic@ru.nl}}

\begin{document}
\maketitle
	
\begin{abstract}
We consider optimal control problems that have binary-valued control
input functions and a perimeter regularization. We develop and analyze
a trust-region algorithm that solves a sequence of subproblems in which
the regularization term and the binarity constraint are relaxed
by a non-convex energy functional. 

We show how the parameter that controls the distinctiveness of the
resulting phase field can be coupled to the trust-region radius updates and be driven to zero over the course
of the iterations in order to obtain convergence to 
points that satisfy a first-order optimality
condition of the limit problem under suitable regularity
assumptions.

Finally, we highlight and discuss the assumptions and restrictions of our
approach and provide the first computational results for a motivating 
application in the field of control of acoustic waves in dissipative
media.
\end{abstract}
\textbf{2020 Mathematics Subject Classification.} 49M05,49M20,49Q15

\section{Introduction}

Let $\Omega \subset \R^d$ with $d \in \{2,3\}$ be a bounded
Lipschitz domain. We are interested in solving
optimal control problems of the following form:
\begin{gather}\label{eq:p} \tag{$\textup{P}$}
\left\{ \begin{aligned}
&\hspace*{2cm}\min_{u, w}\ J(u) + \gamma C_0 P_{\Omega}(w^{-1}(\{1\})),\\
&\text{with} \\ 
&\hspace*{2cm}w  \in \Wadzero \coloneqq
\BV(\Omega, \{0,1\}),\quad u \in \calU,
\\
&\text{subject to }\\
&\hspace*{2cm}u = S(w).
\end{aligned}
\right.
\end{gather}
In \eqref{eq:p}, $J$ is the principal part of the objective
and depends on $u\in \calU$, where $\calU$ is the state space
of an underlying (initial) boundary-value problem for a partial differential equation (PDE). Further, $w$ is the control input of the PDE and may only
attain the values $0$ and $1$, which is reflected by
our notation $\BV(\Omega, \{0,1\})$ for the space
of functions of bounded variation on $\Omega$
that are $\{0,1\}$-valued almost everywhere (a.e.).

We assume that the PDE is uniquely solvable for any admissible control function $w$
and we denote the solution map by $S : \Wadzero \to \calU$.
The optimization problem is regularized with the
term $\gamma C_0 P_{\Omega}(w^{-1}(\{1\}))$
with positive constants $C_0 > 0$ and $\gamma > 0$
and the perimeter of the level set $w^{-1}(\{1\})$
in $\Omega$, that is, the function
$P_{\Omega}$ gives the perimeter of a set
in $\Omega$. We recall that for our binary setting,
$P_{\Omega}(w^{-1}(\{1\})) = \TV(w)$ if
$\TV : L^1(\Omega) \to [0,\infty]$ is the total
variation seminorm on $\Omega$.

Optimization problems of the form \eqref{eq:p} arise, for example, in acoustics in the context of optimal focusing of 
ultrasound waves and are particularly relevant in acoustic imaging~\cite{szabo2004diagnostic} and  applications of high-intensity ultrasound
waves~\cite{kennedy2005high}. Focusing can be achieved through the use of acoustic lenses, which represent regions with different acoustic material properties (e.g., the speed of sound).  
Thus, problems of optimal focusing can generally be treated with algorithmic approaches from  optimal control and shape and topology optimization.
Excellent overviews on such methods can be found in \cite{allaire2021shape, manzoni2021optimal}. We refer to~\cite{clason2021optimal} for a rigorous study of an optimal control problem for a linear wave equation with the control in the propagation speed. We also mention the works \cite{tran2017shape,garcke2022phase},
which use a phase-field method to solve the problem of optimal acoustic focusing by reframing it as an optimal control problem with the control in the phase-field function.

In this work, we combine a trust-region algorithm, see \cite{manns2022on,yan2021discrete,hu2022adaptive}
for recent work on trust-region algorithms in topology optimization, with ideas from phase-field models.
Instead of replacing \eqref{eq:p} with a phase-field model, we replace the perimeter functional
$C_0P(w^{-1}(\{1\}))$ by the Ginzburg--Landau energy $E_\varepsilon$ in the subproblems generated by the
trust-region algorithm and drive
$\varepsilon \searrow 0$, where $\varepsilon$ controls the distinctiveness or non-binarity
of the phase field. More precisely, for $\varepsilon > 0$, $E_\varepsilon$ is defined by
\begin{gather*}
E_\varepsilon(w) \coloneqq \left\{
\begin{aligned}
&\int_\Omega \frac{\varepsilon}{2}|\nabla w|^2 + \frac{1}{\varepsilon} \Psi(w)\dd x 
	&& \text{ if } w \in H^1(\Omega),\\
	&\infty 
	&& \text{ else,}
\end{aligned}\right.
\end{gather*}
where $\Psi$ is the superposition operator that is defined 
for all $w \in L^1(\Omega)$ and almost every (a.e{e}.) $x \in \Omega$ by
\[ \Psi(w)(x) \coloneqq \left\{
\begin{aligned}
&w(x)(1 - w(x)) &&\text{ if }w(x) \in [0,1], \\
&\infty &&\text{ else.}
\end{aligned}
\right.
\]
Consequently, the feasible set of our subproblems is relaxed to
$\Wadeps \coloneqq H^1(\Omega;[0,1])$, the space of weakly 
differentiable $L^2(\Omega)$-functions with weak derivative in 
$L^2(\Omega)$ that are $[0,1]$-valued a.e.

To the best of our knowledge, there is not much work on the question of how to drive $\varepsilon \searrow 0$
in optimization algorithms in order to obtain beneficial properties for the limit. We start from the recent
work on a trust-region algorithm \cite{leyffer2022sequential,manns2022on} that allows for non-trivial geometric
and thus topological changes of the level sets of the control functions. We study a 
modification of this approach, where the perimeter regularization is replaced by the 
Ginzburg--Landau energy and $\varepsilon$ is driven to zero over the course of the 
iterations. In this way, the smoothness allows for a relatively straightforward
consistent discretization of the regularization term. In contrast, a consistent 
numerical analysis for the subproblems proposed in \cite{manns2022on} is difficult 
because established strategies for the discretization of the perimeter regularizer, 
see, for example, \cite{caillaud2022error,chambolle2021approximating,bartels2012total}, are not
directly applicable, as they use convex duality and that the projection of controls 
to piecewise constant functions on a given grid yields a feasible function.
Both properties are not available for \eqref{eq:p}
or induced subproblems that do not relax the binarity constraint as in 
\cite{manns2022on}.

\subsection{Contributions}

We develop a homotopy trust-region algorithm that solves subproblems, in which
the main part of the objective $J \circ S$ is linearized and the regularizer is approximated with the
Ginzburg--Landau energy functional $E_\varepsilon$. We assume that the trust-region subproblems can be
solved exactly and limit points of our iterates admit a certain regularity condition, see the assumptions
of \cref{thm:gamma_to_stat}, which can be
interpreted as a constraint qualification when relating the approach to nonlinear programming algorithms. 
Under these conditions, we prove convergence of subsequences of the iterates to so-called L-stationary points of
the limit problem \eqref{eq:p}, where L-stationarity is the first-order optimality condition
from \cite{manns2022on}.

Our proof uses a compactness result that requires boundedness of the Ginzburg--Landau energies
for $\varepsilon \searrow 0$ for (a subsequence of) the iterations of our algorithm. The
boundedness is not trivial to establish because the Ginzburg--Landau energy may increase when
$\varepsilon$ is reduced. We establish this property by showing how to construct a suitable competitor
that eventually becomes feasible for the trust-region subproblems and then implies reductions of the
objective whenever the Ginzburg--Landau energy exceeds a certain threshold before the next 
reduction of $\varepsilon$ is triggered.

Since our trust-region subproblems are non-convex, we also show how to integrate convex subproblems' solves
into the algorithm in order to accelerate it, while our convergence analysis only
hinges on the non-convex solves so far. This is due to the fact that we do not know how Cauchy points,
which usually guide the convergence analysis of trust-region algorithms, could be constructed from
the purely local information that is given by the
violation of L-stationarity condition.

Consequently, our convergence analysis requires three relatively strong assumptions, which open up
ample possibilities for future research on their alleviation and handling and thus improving the
analyzed algorithm. First, we need to solve a class of non-convex subproblems to global optimality
in order to guarantee the asymptotic properties. Second, we need an additional regularity condition
(constraint qualification) on the (reduced) boundaries $\partial^* \bar{w}^{-1}(\{1\})$ of the level
sets of our limit points. Third, the regularity assumptions on $S$ that are required by the algorithm
currently have a gap to the regularity that we can assert for our motivating application.
While we focus on
the functionals $E_\varepsilon$ introduced above in this work,
we believe that our results can be carried over to other
families of energy functionals that approximate the perimeter functional,
see, for example, \cite{amstutz2012topology,amstutz2022consistent}.

\subsection{Structure of the remainder}

After providing some guidance on our notation, we formally introduce our problem and recap
the L-stationarity concept from \cite{manns2022on} in \cref{sec:abstract}. We prove the L-stationarity
in slightly more generality than is done in \cite{manns2022on} so that only the first derivatives are involved.
We then introduce and describe the analyzed algorithm as well as its subproblems in 
\cref{sec:homotopy_trust_region_algorithm}, where we also provide a $\Gamma$-convergence type result
for the trust-region subproblems when driving $\varepsilon \searrow 0$. In \cref{sec:asymptotics},
we show our main results, the asymptotic boundedness and 
L-stationarity of subsequences of our iterates.
In \cref{sec:applications}, we verify our regularity assumptions on $S$ for an elliptic PDE and show
how far we can currently get for a linear acoustic equation when the design variable enters the propagation speed coefficient.
Moreover, we provide a discretization of the acoustic equation and show some preliminary computational
results of our algorithm in this case, where we resort to convex subproblem solves.
Finally, we provide a brief conclusion in \cref{sec:conclusion}.

\subsection{Notation and assumptions}
We assume throughout that $\Omega \subset \R^d$ is a bounded Lipschitz domain. In the context of acoustic
waves, $T>0$ denotes the final propagation time, which is given and fixed. We use $H^1(\Omega)^*$ to denote
the dual of $H^1(\Omega)$. When denoting norms on Bochner spaces, we often omit the temporal domain; that is, 
$\|\cdot\|_{L^p(L^q(\Omega))}$ denotes the norm on $L^p(0,T; L^q(\Omega))$. Note that we also replace
$\|\cdot\|_{L^p(\Omega)}$ by $\|\cdot\|_{L^p}$ and $\|\cdot\|_{H^1(\Omega)}$ by $\|\cdot\|_{H^1}$ in several
of our proofs when there are no ambiguities in the interest of a clean and
compact presentation.
We use $x \lesssim y$ to denote $x \leq C y$,
where $C>0$ is a generic constant. For a set of finite 
perimeter (Caccioppoli set) $E$, we denote its reduced 
boundary by $\partial^* E$.
We recall that for a Caccioppoli set $E$
with $\{0,1\}$-valued indicator function $\chi_E$,
the reduced boundary $\partial^* E$ is given as
the set of points $x \in \supp D\chi_E$ such that
\[ \nu_E(x) \coloneqq
\lim_{r \searrow 0} \frac{D\chi_E(B_r(x))}{|D\chi_E|(B_r(x))}
\]
exists and satisfies $\|\nu_E(x)\| = 1$. Here, $D\chi_E$
denotes the weak derivative of $\chi_E$.
We refer to Section 15 in \cite{maggi2012sets} for more details.
For a set $A$, $\delta_A$ denotes its $\{0,\infty\}$-valued
indicator function. We write $\Ha^{d-1}$ for the
$d-1$-dimensional Hausdorff measure on $\R^d$.

\section{Abstract problem formulation and optimality conditions}\label{sec:abstract}
For the statement and analysis of the asymptotics of our 
algorithm, we replace $J(u)$ and $u = S(w)$ in \eqref{eq:p}
by the reduced objective $j(w) \coloneqq J(S(w))$,
which gives the following abstract and handy problem formulation
for $C_0 > 0$:
\begin{gather}\label{eq:p_abstract}
\begin{aligned}
\min_{w \in L^2(\Omega)}\ & j(w) + \gamma C_0P_\Omega(w^{-1}(\{1\})) \\
\text{s.t.}\quad & w(x) \in \{0,1\} \text{ for a.e.\ } x\in \Omega.
\end{aligned}\tag{PA}
\end{gather}
The following existence result for minimizers is well known.
\begin{proposition}
Let $j : L^2(\Omega) \to \R$ be lower semi-continuous
and bounded below. Then \eqref{eq:p_abstract} admits
a minimizer.
\end{proposition}
\begin{proof}
The claim follows from the direct method of calculus of
variations, where the necessary compactness comes from
the fact that bounded subsets of $\Wadzero$ (with respect
to the perimeter functional) are compact in $L^1(\Omega)$,
see, for instance,
\cite{leyffer2022sequential,burger2012exact}.
\end{proof}
With the help of local variations (smooth perturbations of the essential boundary)
of the level sets of minimizers of \eqref{eq:p_abstract}, one can
prove first-order optimality conditions for \eqref{eq:p_abstract}. To this end, it is necessary
to bound the remainder term of the Taylor expansion of the perimeter functional
\cite{leyffer2022sequential,manns2022on,maggi2012sets,marko2022integer}, 
which is possible, for example, by means of the following
assumption.
\begin{assumption}[Relaxation of Assumption 4.3 in \cite{manns2022on}]\label{ass:general_var}
	Let $j : L^1(\Omega) \to \R$ be continuously Fr\'{e}chet differentiable.
\end{assumption}
As in \cite{leyffer2022sequential,manns2022on}, we note that
the requirement that the differentiability is with respect to the
$L^1(\Omega)$-norm on the domain in \cref{ass:general_var} implies
a regularity enhancement inside $F$. This is a typical assumption
in the presence of binary- or discrete-valued control functions
because the linear part of Taylor's expansion of $F$ can only
be estimated with respect to the $L^1$-norm difference and
the $L^1$-norm difference is the same as the squared $L^2$-norm
for binary-valued functions.
We introduce the tailored concept of
first-order optimality below.
\begin{proposition}\label{prp:stationarity}
	Let \cref{ass:general_var} hold and $\Rgrad j(\bar{w}) \in C(\bar{\Omega})$
	for some $\bar{w}\in \Wadzero$.
	If $\bar{w}$ is a local minimizer, that is
	\[j(\bar{w}) + \gamma C_0P_\Omega(\bar{w}^{-1}(\{1\})) \le j(w) + \gamma C_0P_\Omega(w^{-1}(\{1\}))\]
	for all $w \in \Wadzero$ with $\|w - \bar{w}\|_{L^1(\Omega)} \le r$ for some $r > 0$, 
	then $\bar{w}$ and $E = \bar{w}^{-1}(\{1\})$ satisfy
\begin{gather}\label{eq:p_variational_principle}
\int_{\partial^*E \cap \Omega}
-\Rgrad j(\bar{w})(\phi\cdot n_{E}) \dd \Ha^{d-1} = 
\gamma C_0 \int_{\partial^* E \cap \Omega}\bdvg{E} \phi\dd \Ha^{d-1}
\end{gather}	
for all $\phi \in C_c^\infty(\Omega, \R^d)$, where
$\bdvg{E} \phi: \partial^*E \to \R$ is the 
\emph{boundary divergence of} $\phi$ on $E$, that is
\[\bdvg{E}\phi(x) \coloneqq \dvg \phi(x) - n_E(x) \cdot \nabla \phi(x)n_E(x)\]
for $x \in \partial^* E$, $n_E$ is the outer normal on $\partial^* E$,
and $\Rgrad j(\bar{w}) \in L^2(\Omega)$ is the Riesz representative of $j'(\bar{w})$.
\end{proposition}
\begin{proof}
	We closely follow the arguments in \cite{manns2022on}. Let
	the local variation $(f_t)_{t\in(-\varepsilon,\varepsilon)}$
	be defined by
	$f_t \coloneqq I + t \phi$ for a smooth and compactly supported vector field
	$\phi \in C_c^\infty(\Omega; \R^d)$.
	Here, $\varepsilon > 0$ is small enough such that $(f_t)_{t\in(-\varepsilon,\varepsilon)}$ is a family
	of diffeomorphisms. Moreover, $\{ x \in \R^d\,|\, f_t(x) \neq x\} \subset K$ holds for some compact set $K \subset \Omega$
	and all $t \in (-\varepsilon, \varepsilon)$, see \cite[Prop.\,3.2]{manns2022on}.
	We define the induced transformation of the $\{0,1\}$-valued function $\bar{w}$ by means of $f_t$ as
	\[f_{t}^{\#}\bar{w} \coloneqq \chi_{f_t(\bar{w}^{-1}(\{1\}))}
	= \left\{
		\begin{aligned}
		1 & \text{ if } x \in f_t(\bar{w}^{-1}(\{1\})) \\
		0 & \text{ else }
		\end{aligned}
	\right.
	 \]
	for $t \in (-\varepsilon, \varepsilon)$. This means that $f_{t}^{\#}\bar{w}$ is the $\{0,1\}$-valued
	function whose level sets are given by the images of the level sets of $\bar{w}$ under $f_t$. In other words,
	we obtain a corresponding perturbation of the level sets of $\bar{w}$, which translates
	to a perturbation of the interface between the level sets. Our condition \eqref{eq:p_variational_principle} arises below
	from local optimality because no such perturbation can be able to improve the objective.
	
	Specifically, the function $t \mapsto J(f_{t}^{\#}\bar{w}) + \gamma C_0 P_\Omega((f_{t}^{\#}\bar{w})^{-1}(\{1\}))$ is differentiable at
	$t = 0$, the first-order optimality condition
	\[  \frac{\dd}{\dd t}J(f_{t}^{\#}\bar{w}) + \gamma C_0 P_\Omega((f_{t}^{\#}\bar{w})^{-1}(\{1\}))\Big|_{t = 0} = 0 \]
	holds by virtue of Fermat's theorem, the assumed optimality of $\bar{w}$ in an $L^1$-ball of radius $r$, and the estimate
	\begin{gather}\label{eq:L1_bound}
	\|f_t^{\#}\bar{w} - \bar{w}\|_{L^1} \le \bar{C} |t| P_\Omega((f_{t}^{\#}\bar{w})^{-1}(\{1\}))
	\end{gather}
	that holds for some $\varepsilon_0 > 0$ and $\bar{C} > 0$ for all $t \in (-\varepsilon_0,\varepsilon_0)$. Estimate \eqref{eq:L1_bound} is shown in
	Lemma 3.8 in \cite{manns2022on} (use $\TV(f_{t}^{\#}\bar{w}) = P_\Omega((f_{t}^{\#}\bar{w})^{-1}(\{1\}))$).
	
	The Fr\'{e}chet differentiability of $J$ implies
	\begin{align*}
	\frac{\dd}{\dd t} J(f_t^{\#}\bar{w})\Big|_{t = 0}
	= \lim_{t \searrow 0}\frac{(\nabla_R J(\bar{w}), f_t^{\#}\bar{w} - \bar{w})_{L^2}}{t}
	+ \frac{o(\|f_t^{\#}\bar{w} - \bar{w}\|_{L^1})}{t}
	= \frac{\dd}{\dd t} (\nabla J(v), f_t^{\#}\bar{w})_{L^2}\Big|_{t = 0},
	\end{align*}
	where the second identity follows from the estimate \eqref{eq:L1_bound}. In combination
	with Theorem 17.5 in \cite{maggi2012sets},
	the function
	$t \mapsto J(f_{t}^{\#}\bar{w}) + \gamma C_0 P_\Omega((f_{t}^{\#}\bar{w})^{-1}(\{1\}))$
	is differentiable at $t = 0$.
	
	Consequently, we obtain
	\begin{gather}\label{eq:derivative_to_derivative}
	\frac{\dd}{\dd t} J(f_t^{\#}\bar{w})\Big|_{t = 0}
	= \frac{\dd}{\dd t}
	(\nabla J(v), f_t^{\#}\bar{w})_{L^2}\Big|_{t = 0}
	\end{gather}
	and
	the claim follows from Lemmas 3.3 and 3.5 in \cite{manns2022on}.
\end{proof}

\begin{definition}[Definition 4.4 in \cite{manns2022on}]\label{dfn:stationarity}
Let \cref{ass:general_var} hold and $\Rgrad j(\bar{w}) \in C(\bar{\Omega})$	for some $\bar{w}\in \Wadzero$.
Then we call $\bar{w}$ \emph{L-stationary} if it satisfies
\eqref{eq:p_variational_principle}.
\end{definition}
L-stationarity is a first-order condition for local optimality for 
\eqref{eq:p_abstract} at points, where the Riesz representative 
$\Rgrad j(w)$ is also a continuous function, which we recall in
the following proposition.
\eqref{eq:p_variational_principle} means that the negative gradient 
of the first part of the objective is the distributional mean 
curvature of the boundary of the level set $w^{-1}(\{1\})$ after
scaling by $\gamma^{-1}C_0^{-1}$, see Remark 17.7 in 
\cite{maggi2012sets}. Using the
\emph{second variation of the perimeter}, a stronger optimality
condition for \eqref{eq:p} that involves the
\emph{principal curvatures} of the boundaries of the level sets
can likely be shown under the assumption of additional boundary 
regularity at the local minimizer. In the case $\Omega \subset \R$,
which is not covered in this work, L-stationarity was
improved to a second-order type optimality condition
under the assumption that $\nabla_R j(\bar{w})$ is
in $C^1$ in \cite[\S3.2]{marko2022integer}.

\section{Homotopy trust-region algorithm}\label{sec:homotopy_trust_region_algorithm}
We develop a trust-region algorithm for the optimization of 
\eqref{eq:p}, where we follow the abstract algorithm design in 
\cite{leyffer2022sequential,manns2022on}.
We introduce trust-region subproblems for the sharp interface
problem and the phase-field approximation in \cref{sec:subproblems}, 
which are related to each other with a result of
$\Gamma$-convergence-type. Then, we introduce the algorithm that 
solves trust-region subproblems and drives the interface parameter
to zero over the course of its iterations in 
\cref{sec:algorithm_statement}.

\subsection{Trust-region subproblems}\label{sec:subproblems}
We consider three trust-region subproblems. The first arises from the sharp-interface formulation:
\begin{gather}\label{eq:tr_sharp}
\text{{\ref{eq:tr_sharp}}}(\bar{w}, g, \Delta) \coloneqq
\left\{
\begin{aligned}
\min_{w \in L^2(\Omega)}\ & (g, w - \bar{w})_{L^2(\Omega)} + \gamma C_0  P_\Omega(w^{-1}(\{1\}))
-\gamma C_0  P_\Omega(\bar{w}^{-1}(\{1\}))\\
\text{s.t.}\quad & \|w - \bar{w}\|_{L^1(\Omega)} \le \Delta\text{ and }w(x) \in \{0,1\} \text{ for a.e.\ } x \in \Omega.
\end{aligned}
\right.
\tag{TR}
\end{gather}
The inputs of \eqref{eq:tr_sharp} are a (partial) linearization point $\bar{w}$, a function $g$, typically the gradient
of the reduced objective $\Rgrad j(\bar{w})$ or an approximation, and a trust-region radius $\Delta$.
This subproblem is analyzed in \cite{leyffer2022sequential,manns2022on} (choose $V = \{0,1\}$, $\alpha = \gamma$,
and $C_0 = 1$ therein and observe that $\TV(w) = P_\Omega(w^{-1}(\{1\}))$. A distinct feature of
\eqref{eq:tr_sharp} is that the regularization
term is considered exactly and not replaced by a model function. \\
\indent After discretization, an instance of class \eqref{eq:tr_sharp} becomes a finite-dimensional integer linear
program \cite{leyffer2022sequential}. It is, however, not known if the discretizations can be
solved efficiently for $d \ge 2$. Moreover, binary-valued functions
are usually modeled as piecewise constant functions. Therefore any discretization of the domain $\Omega$
into finitely many cells restricts the geometry of the level sets of binary-valued ansatz functions. Specifically, the boundaries
of their level sets can only follow the boundaries of the discretization cells. This may lead to a gap between the value of 
the perimeter functional for a discretized approximant and the
perimeter functional for the function that is approximated that
is bounded below.
A detailed example with visualization is given in Example 2.10 and Figure 1 in \cite{schiemann2024discretization}.
Our analysis of the trust-region algorithm 
introduced below solves a sequence of
subproblems that
approximate \eqref{eq:tr_sharp}.

Specifically, the second trust-region subproblem arises by replacing the perimeter regularization $C_0 P_{\Omega}(w^{-1}(\{1\}))$ with the Ginzburg--Landau energy $E_\varepsilon(w)$, which yields
\begin{gather}\label{eq:tr_gle}
\text{{\ref{eq:tr_gle}}}(\bar{w}, g, \Delta) \coloneqq
\left\{
\begin{aligned}
\min_{w \in L^2(\Omega)}\ & (g, w - \bar{w})_{L^2(\Omega)} + \gamma E_\varepsilon(w) - \gamma E_\varepsilon(\bar{w})\\
\text{s.t.}\quad & \|w - \bar{w}\|_{L^2(\Omega)}^2 \le \Delta\text{ and }w(x) \in [0,1] \text{ for a.e.\ } x \in \Omega.
\end{aligned}
\right.
\tag{TR$_\varepsilon$}
\end{gather}
We note that we impose the trust region
in \eqref{eq:tr_gle} with the (smooth) squared $L^2$-norm 
instead of the $L^1$-norm. This does not make a difference
for binary-valued functions, however, because the difference of
two binary-valued functions is $\{-1,0,1\}$-valued
so that the squared $L^2$-norm coincides with the $L^1$-norm
of the difference.
While instances of \eqref{eq:tr_gle} are non-convex optimization problems, there are well-established discretizations
available for them that can, in principle, be solved to $\varepsilon$-global optimality
with successive approximations of the non-convexities and a branch-and-bound strategy, albeit this is clearly
computationally very expensive.

Since this trust-region subproblem is hard to solve due to the non-convexity of the term $\Psi$, we propose
to replace many of its solves by solves of the following trust-region subproblem that replaces the
non-convex term in $E_\varepsilon(w)$ by its linearization at $\bar{w}$ and is convex.
\begin{gather}\label{eq:tr_cvx}
\text{{\ref{eq:tr_cvx}}}(\bar{w}, g, \Delta) \coloneqq
\left\{
\begin{aligned}
\min_{w \in L^2(\Omega)}\ & \left(g + \frac{\gamma}{\varepsilon}(1 - 2\bar{w}), w - \bar{w}\right)_{L^2(\Omega)} \\
&\quad + \gamma \frac{\varepsilon}{2}\|\nabla w\|^2_{L^2(\Omega)} - \gamma \frac{\varepsilon}{2}\|\nabla \bar{w}\|^2_{L^2(\Omega)}
   \\
\text{s.t.}\quad & \|w - \bar{w}\|_{L^2(\Omega)}^2 \le \Delta\text{ and }w(x) \in [0,1] \text{ for a.e.\ } x \in \Omega.
\end{aligned}
\right.
\tag{TR$^c_\varepsilon$}
\end{gather}
We briefly state that the trust-region subproblems \eqref{eq:tr_sharp}, \eqref{eq:tr_gle}, and
\eqref{eq:tr_cvx} admit minimizers under mild assumptions.
\begin{proposition}\label{prp:tr_existence}
Let $\gamma > 0$, $\varepsilon > 0$, $g \in L^2(\Omega)$, and $\Delta > 0$.
\begin{enumerate}
\item If $\bar{w} \in H^1(\Omega)$ and $\bar{w}(x) \in [0,1]$ for a.e.\ $x \in \Omega$, then
$\text{{\emph{\ref{eq:tr_gle}}}}(\bar{w}, g, \Delta)$ admits a minimizer in $H^1(\Omega)$.
\item If $\bar{w} \in H^1(\Omega)$ and $\bar{w}(x) \in [0,1]$ for a.e.\ $x \in \Omega$, then
$\text{{\emph{\ref{eq:tr_cvx}}}}(\bar{w}, g, \Delta)$ admits a minimizer in $H^1(\Omega)$.
\item If $\bar{w}(x) \in \{0,1\}$ for a.e.\ $x \in \Omega$ and $P_\Omega(\bar{w}^{-1}(\{1\})) <\infty$, then 
$\text{{\emph{\ref{eq:tr_sharp}}}}(\bar{w}, g, \Delta)$ admits a minimizer
in $\Wadzero$.
\end{enumerate}
In all of these situations, the minimal objective value is less than or equal to zero.
\end{proposition}
\begin{proof}
The first and second existence claims follow by using the direct method of calculus of variations. The existence
of an infimal value in $\R$ follows, for example, from the $L^\infty$-bounds, which yields a bound on
the first term in the objective in both cases and a bound on the non-convex term in $E_\varepsilon$,
as well as the coercivity of the term
$\frac{\varepsilon}{2}\|\nabla w\|_{L^2(\Omega)}^2$ with respect to
$H^1(\Omega)$ under the boundedness in $L^2(\Omega)$.
Then, the limit of a weakly converging infimal sequence is a minimizer because of the compact embedding
$H^1(\Omega) \hookrightarrow L^2(\Omega)$, which yields continuity of the first term and the non-convex term
in $E_\varepsilon$ in the first claim. Finally, the term $\frac{\varepsilon}{2}\|\nabla w\|_{L^2(\Omega)}^2$ is lower
semi-continuous with respect to weak convergence in $H^1(\Omega)$.

The third existence claim follows with the direct method of calculus of variations when the feasible set
of $\text{{\ref{eq:tr_sharp}}}(\bar{w}, g, \Delta)$ is equipped with the weak-$^*$ topology
of functions of bounded variations, see \cite[Prop.\ 2.3]{leyffer2022sequential}.

Finally, the feasibility of $\bar{w}$ with objective value zero gives the non-positivity of the
minimal objective value for all of the three trust-region subproblems.
\end{proof}

L-stationarity in the sense of \cref{dfn:stationarity} also holds for minimizers of \eqref{eq:tr_sharp},
which is given below.

\begin{proposition}\label{prp:tr_sharp_stationarity}
Let $g \in C(\bar{\Omega})$. Let $\Delta > 0$.
Let $w \in L^1(\Omega)$ satisfy $w(x) \in \{0,1\}$ for a.e.\ $x \in \Omega$ and $P_\Omega(w^{-1}(\{1\})) <\infty$.
Let $w$ be optimal for $\text{\emph{\ref{eq:tr_sharp}}}(w,g,\Delta)$ in an $L^1$-ball of radius $r$ around $w$
for some $r > 0$, then $w$ is L-stationary.
\end{proposition}
\begin{proof}
This follows from \cref{prp:stationarity} with the choice $j(w) \coloneqq (g,w)_{L^2}$, $w \in L^2(\Omega)$,
after observing that the optimality of $w$ in an $L^1$-ball of radius $r$ also implies
optimality in an $L^1$-ball of radius $r_0$ for some $r_0 \le \Delta$, see also the proof of
Proposition 5.5 in \cite{manns2022on}.
\end{proof}

In \cite{manns2022on}, the convergence analysis of the 
trust-region algorithm therein uses that, for a fixed trust-region radius,
the minimizers of the trust-region subproblems for a converging (sub)sequence
of iterates converge to a minimizer to the limit problem. This then leads
to a sufficient decrease for iterates that are close enough to the limit
if the limit is not stationary. This work will use a similar argument
and we thus re-establish this convergence of minimizers for our situation,
where the regularization term changes with $\varepsilon^n \searrow 0$
over the course of the iterations $n \to \infty$. The intended result
that minimizers of \eqref{eq:tr_gle} converge to minimizers \eqref{eq:tr_sharp},
follows from the following $\Gamma$-convergence
considerations. To this end, we consider a sequence
$(v^n)_n \subset \Wadzero$ and $v \in L^1(\Omega)$,
a sequence $(g^n)_n \subset L^2(\Omega)$ and $g\in L^2(\Omega)$,
$\varepsilon^n \subset (0,\infty)$, and $\Delta > 0$.

For $n \in \N$, we define the functionals $T^n:\{ w \in L^1(\Omega)\,|\,\|w\|_{L^\infty(\Omega)}\le 1\}\to\R$ by
\[ T^n(w) \coloneqq (g^n, w - v^n)_{L^2(\Omega)}
+ \gamma E_{\varepsilon^n}(w) 
- \gamma E_{\varepsilon^n}(v^n) 
+ \delta_{[0,\Delta]}(\|w - v^n\|^2_{L^2(\Omega)})
\]
for $w \in L^1(\Omega)$ and the limit functional $T:\{  w \in L^1(\Omega)\,|\,\|w\|_{L^\infty(\Omega)}\le 1\} \to \R$ by
\[
T(w) \coloneqq (g, w - v)_{L^2} + \gamma C_0P_\Omega(w^{-1}(\{1\})) - \gamma C_0P_\Omega(v^{-1}(\{1\}))
+ \delta_{\{0\}}(|w^{-1}((0,1))|)
+ \delta_{[0,\Delta]}(\|w - v\|_{L^1})
\]
for $w \in L^1(\Omega)$.
\begin{theorem}\label{thm:tr_gamma_convergence}
Let $(v^n)_n$ and $v$ satisfy
$\sup_{n\in\N}\|v^n\|_{L^\infty(\Omega)}\le 1$,
$v^n \to v$ in $L^1(\Omega)$, and
$E_{\varepsilon^n}(v^n) \to C_0P_\Omega(v^{-1}(\{1\}))$
for $n \to \infty$ with $v(x) \in \{0,1\}$
for a.e.\ $x\in \Omega$.
Let $(g^n)_n$ and $g$ satisfy $g^n \weakto g$ in $L^2(\Omega)$.
Let $\varepsilon^n \searrow 0$ and $\Delta > 0$.
Then the functionals $T^n$ $\Gamma$-converge to $T$
with respect to convergence in $L^1(\Omega)$.
\end{theorem}
\begin{proof}
We extend the arguments from the proof of Theorem 5.2 in \cite{manns2022on}, where
the perimeter regularization is present in the functionals $T^n$ instead of
the Ginzburg--Landau energy. We prove the lower bound inequality first and the upper bound
inequality second.

In order to prove the lower bound inequality, let $w^n \to w$ in $L^1(\Omega)$ and
$\sup_{n\in\N} \|w^n\|_{L^\infty} \le 1$.
Then we obtain $w^n \to w$ in $L^2(\Omega)$ and thus $(g^n, w^n - v^n)_{L^2} \to (g, w - v)_{L^2}$.
We also have $\gamma E_{\varepsilon^n}(v^n) \to \gamma C_0 P_\Omega(v^{-1}(\{1\}))$
by assumption. If $\|w^{n_k} - v^{n_k}\|_{L^2}^2 \le \Delta$ holds for a subsequence (indexed by $k$),
we obtain $\|w - v\|_{L^1} \le \Delta$ by virtue of the triangle inequality
and the fact that $\|f\|_{L^1} = \|f\|_{L^2}^2$ if $f$ is $\{0,1\}$-valued. Combining these considerations
with the analysis of $E_{\varepsilon^n}$ with respect to convergence in $L^1(\Omega)$
in Proposition 1 in \cite{modica1987gradient} yields the lower bound inequality
\[ T(w) \le \liminf_{n\to\infty} T^n(w^n). \]
Note that Proposition 1 in \cite{modica1987gradient} also asserts $w(x) \in \{0,1\}$ a.e.\ if
$\liminf_{n\to\infty} E_\varepsilon(w^n) < \infty$.

The upper bound inequality is immediate if $T(w) = \infty$ holds with the choice $w^n = w$ for all $n\in\N$.
Let $w \in L^1(\Omega)$ with $T(w) < \infty$ be given. We distinguish the cases
$\|v - w\|_{L^1} < \Delta$ and $\|v - w\|_{L^1} = \Delta$.
In the case $\|v - w\|_{L^1} \eqqcolon \tilde{\Delta} < \Delta$, we apply Lemma 1 in
\cite{modica1987gradient}, which provides open and bounded subsets of $\R^d$
with smooth boundary $F_k$ so that $P_\Omega(F_k) \to P_\Omega(w^{-1}(\{1\}))$
and $|(F_k \cap \Omega) \symdiff w^{-1}(\{1\})| \to 0$.
For each such set $F_k$, Proposition 2 in \cite{modica1987gradient}
provides functions $w^n_k \to \chi_{F_k \cap \Omega}$ in $L^1(\Omega)$ with $\|w^n_k\|_{L^\infty} \le 1$
for all $n \in \N$ such that
\[ \limsup_{n\to\infty} \gamma E_{\varepsilon^n}(w^n_k) = \gamma C_0 P_\Omega(F_k). \]
By choosing a suitable diagonal sequence, which we denote by  $w^n$, we obtain
\[ \limsup_{n\to\infty} \gamma E_{\varepsilon^n}(w^n) = \gamma C_0 P_\Omega(w^{-1}(\{1\})). \]

Moreover, we obtain
\begin{gather}\label{eq:wn_vn_strictly_less}
\|w^n - v^n\|_{L^2}^2
\le \|w^n - v^n\|_{L^1}
\le \|v - v^n\|_{L^1} + \|w - w^n\|_{L^1} + \tilde{\Delta},
\end{gather}   
where the first inequality follows from $|w^n(x) - v^n(x)| \le 1$,
thereby implying $|w^n(x) - v^n(x)|^2 \le |w^n(x) - v^n(x)|$,
for a.e.\ $x \in \Omega$.
Because of $v^n \to v$ and $w^n \to w$ in $L^1(\Omega)$ and $\tilde{\Delta} < \Delta$,
the right hand side in \eqref{eq:wn_vn_strictly_less} is bounded by $\Delta$ for all large enough $n$.

We finish the proof by considering the case $\|v - w\|_{L^1} = \Delta$. 
This implies that there is a set $D = \{ x \in \Omega \,|\, w(x) = 1 - v(x) \}$ that has strictly positive $d$-dimensional Lebesgue measure.
We assume without loss of generality that $v(x) = 1$ a.e.\ in $D$ and obtain
\[ P_\Omega(D) 
= P_\Omega\left(v^{-1}(\{1\}) \cap w^{-1}(\{0\})
\right).
\]
Because $D$ has strictly positive $d$-dimensional Lebesgue measure,
$D$ has a point $y \in D$ of density $1$, that is
\[
\lim_{r\searrow 0} \frac{|D \cap B_r(y)|}{|B_r(y)|} = 1.
\]
We deduce that there exists $r_0 > 0$ such that
\[ 
\overline{B_{r_0}(y)} \subset \Omega
\quad
\text{and}
\quad
|D \cap B_r(y)| \ge \frac{1}{2}|B_r(y)| \text{ for all } 0 < r < r_0.
\]
In particular, $B_{r_0}(y)$ has strictly positive distance to the boundary $\partial \Omega$.

Let $0 < \delta < r_0$. We define for $x \in \Omega$:
\[ w^\delta(x) \coloneqq \left\{ \begin{aligned}
w(x) & \text{ if } x \in \Omega \setminus B_{\delta}(y), \\
v(x) & \text{ else}
\end{aligned}\right.
\]
so that
\begin{gather}\label{eq:wdelta_v}
\|w^\delta - v\|_{L^1} 
\le \|w - v\|_{L^1} - \frac{1}{2}\delta^d|B_1|
= \Delta - \frac{1}{2}\delta^d |B_1|.
\end{gather}
We approximate $(w^\delta)^{-1}(\{1\})$ with an open and bounded set
$F^\delta$ that has a smooth boundary and define
$\tilde{w}^\delta \coloneqq \chi_{F^\delta \cap \Omega}$ such that
\begin{align}
\|w^\delta - \tilde{w}^\delta\|_{L^1} &\leq \frac{1}{4}\delta^d|B_1| \text{ and} \label{eq:wdelta_tildewdelta}\\
|P_\Omega(F^\delta) - P_\Omega((w^\delta)^{-1}(\{1\}))| &< 
\delta \frac{1}{2C_0},
\end{align}
which is possible because of Lemma 2 in \cite{modica1987gradient}.
Next, we consider an approximation as is asserted in Proposition 2 in
\cite{modica1987gradient} for $\tilde{w}^\delta$.
In particular, there exists a large enough $n_\delta \in \N$
such that for all $n \ge n_\delta$ there exists a
$[0,1]$-valued function $\tilde{w}^\delta_n \in H^1(\Omega)$ such that
\begin{align}
|E_{\varepsilon^{n}}(\tilde{w}^\delta_n) - C_0P_\Omega((w^\delta)^{-1}(\{1\}))| &\le |E_{\varepsilon^{n}}(\tilde{w}^\delta_n) - C_0P_\Omega((\tilde{w}^\delta)^{-1}(\{1\}))|
\label{eq:per_tildewdeltan_wdelta}\\
&\hphantom{\le}
+ |C_0 P_\Omega((\tilde{w}^\delta)^{-1}(\{1\})) - C_0P_\Omega((w^\delta)^{-1}(\{1\}))|\nonumber \\
&< \delta,\nonumber\\
\|w^\delta - \tilde{w}^\delta_n\|_{L^1} &\le
\|w^\delta - \tilde{w}_\delta\|_{L^1} + \|\tilde{w}^\delta - \tilde{w}^\delta_n\|_{L^1} 
\le \frac{\delta^d}{4}|B_1|,
\label{eq:wdelta_wtildedeltan}\text{ and}\\
\|v^{n} - v\|_{L^1} &\le \frac{1}{4}\delta^d|B_1|.
\label{eq:vn_v}
\end{align}
We combine \eqref{eq:wdelta_v}, \eqref{eq:wdelta_wtildedeltan},
and \eqref{eq:vn_v} in order to obtain 
\begin{align}
\|v^{n} - \tilde{w}^\delta_n\|_{L^2}^2
\le \|v^{n} - v\|_{L^1} 
+ \|v - w^\delta\|_{L^1} 
+ \|w^\delta - \tilde{w}^\delta_n\|_{L^1} \nonumber\\
\le \|v^{n} - v\|_{L^1} + \Delta - \frac{\delta^d}{4}|B_1|
\le \Delta.\label{eq:L1_reverse_approximation}
\end{align}
Next, we consider the asymptotics of $E_{\varepsilon^{n}}(\tilde{w}^\delta_n)$.
First, we obtain by the construction of $w^\delta$,
the $\sigma$-additivity of $\Ha^{d-1}$, the finiteness of all involved terms,
and the triangle inequality that
\begin{align}
\hspace{1em}&\hspace{-1em}
\left|C_0 P_\Omega\big((w^\delta)^{-1}(\{1\})\big) - C_0 P_\Omega\big(w^{-1}(\{1\})\big)\right|
\nonumber \\
&= C_0 \Big|\Ha^{d-1}\big( (\Omega\setminus \overline{B_\delta(y)}) \cap
\partial^* w^{-1}(\{1\})\big)
+ \Ha^{d-1}\big(\partial B_\delta(x) \cap \partial^*(w^\delta)^{-1}(\{1\})\big) \nonumber \\
&\hphantom{=} + \Ha^{d-1}\big(B_\delta(y) \cap \partial^*v^{-1}(\{1\})\big) 
- \Ha^{d-1}\big(\Omega \cap \partial^* w^{-1}(\{1\})\big)\Big|
\nonumber \\
&= C_0\Big|\Ha^{d-1}\big(\partial B_\delta(y) \cap \partial^*(w^\delta)^{-1}(\{1\})\big)
+ \Ha^{d-1}\big(B_\delta(y) \cap \partial^*v^{-1}(\{1\})\big) 
\nonumber \\
&\hphantom{=} 
- \Ha^{d-1}\big(\partial B_\delta(y) \cap \partial^* w^{-1}(\{1\})\big)
- \Ha^{d-1}\big(B_\delta(y) \cap \partial^* w^{-1}(\{1\})\big)
\Big|
\nonumber\\
&\le C_0\left(2 \Ha^{d-1}\big(\partial B_\delta(y)\big)
+ \Ha^{d-1}\big(B_\delta(y) \cap \partial^*v^{-1}(\{1\})\big) 
+ \Ha^{d-1}\big(B_\delta(y) \cap \partial^* w^{-1}(\{1\})\big)\right)
\nonumber \\
&\le \delta^{d - 1}c_1
+ C_0 \Ha^{d-1}\big(B_\delta(y) \cap \partial^*v^{-1}(\{1\})\big) 
+ C_0 \Ha^{d-1}\big(B_\delta(y) \cap \partial^* w^{-1}(\{1\})\big)
\end{align}
for some $c_1 > 0$. In combination with
\eqref{eq:per_tildewdeltan_wdelta}, the triangle inequality gives
\begin{multline}\label{eq:per_reverse_approximation}
|C_0 P_\Omega(w^{-1}(\{1\})) - E_{\varepsilon^{n}}(\tilde{w}^\delta_n)| \\
\le \delta + \delta^{d-1}c_1 
+ C_0 \Ha^{d-1}\big(B_\delta(y) \cap \partial^*v^{-1}(\{1\})\big) 
+ C_0 \Ha^{d-1}(B_\delta(y) \cap \partial^* w^{-1}(\{1\})).
\end{multline}
We observe that the restricted measures $\Ha^{d-1}(\ \cdot\ \cap \partial^* w^{-1}(\{1\}))$ and
$\Ha^{d-1}(\ \cdot\ \cap \partial^* v^{-1}(\{1\}))$
are finite Radon measures so that the last two terms tend to zero for $\delta \searrow 0$.
Thus, to conclude the argument, we consider a sequence $\delta^k \searrow 0$
and define $n_k \coloneqq n_{\delta^k}$. We define for $k \in \N$
\[ w^{n_k} \coloneqq \tilde{w}^{\delta^{k}}_{n_k}
\]
and obtain by virtue of \eqref{eq:L1_reverse_approximation} and
\eqref{eq:per_reverse_approximation} that
\[ T(w) = \lim_{k\to\infty}T^{n_k}(w^{n_k}), \]
which implies the desired upper-bound inequality. 
\end{proof}

\begin{remark}
The assumption that the regularization terms converge, specifically $E_{\varepsilon^n}(v^n) \to C_0P_\Omega(v^{-1}(\{1\}))$,
may seem restrictive at first. It corresponds to assuming strict (or intermediate)
convergence of $v^n$ in the space of functions of bounded variation if we were
considering convergence in $\BV$ instead of a phase-field approach.
In \cite{manns2022on}, it is shown that
this property can be ensured for iterates produced by a trust-region algorithm with a reset strategy for the
trust-region radius. Our algorithm will provide this feature as well.
\end{remark}
Unfortunately, the result above does not hold for the convex
problem \eqref{eq:tr_cvx}, which can be seen by means of the 
example below. Because of this, our convergence analysis
hinges on solutions to the subproblems \eqref{eq:tr_gle} and 
the subproblems \eqref{eq:tr_cvx} only provide a means to accelerate the
algorithm.
\begin{example}\label{ex:no_gamma_convergence_for_cvx_tr}
Let $w = 1$ and $\bar{w} = 1$ and $\Delta = |\Omega|$, implying that we can neglect
the trust-region constraint in \eqref{eq:tr_cvx}. Let $g^n = g \in L^2(\Omega)$, $v^n = v$ for all $n \in \N$.
Let $\varepsilon^n \searrow 0$. We choose $w^n \to w$ in $L^2(\Omega)$ with $\limsup_{n\to\infty} \gamma E_{\varepsilon^n}(w^n) < \infty$ such that $\varepsilon^n \in o(\|w - w^n\|_{L^1})$.
Then
\begin{align*}
\hspace*{2em}&\hspace*{-2em} (g, w - \bar{w})_{L^2} + \gamma C_0 P_\Omega(w^{-1}(\{1\})) - \gamma C_0 P_\Omega(\bar{w}^{-1}(\{1\})) \\
&= 0 >- \infty = \liminf_{n\to\infty} 
\underbrace{\frac{\gamma}{\varepsilon^n}\int_{\Omega} -(w^n - w) \dd x}_{\to - \infty}\\
&= \liminf_{n\to\infty} \left(g + \frac{\gamma}{\varepsilon}(1 - 2\bar{w}), w^n - \bar{w}\right)_{L^2(\Omega)}
+ \gamma E_{\varepsilon^n}(w^n) - \gamma E_{\varepsilon^n}(\bar{w}).
\end{align*}
This implies that the lower bound inequality cannot hold for the objective functional of \eqref{eq:tr_cvx}.
\end{example}

\subsection{Algorithm statement}\label{sec:algorithm_statement}
The homotopy trust-region algorithm is given in \cref{alg:trm} and proceeds as follows.
For a given $\varepsilon$, it repeatedly solves convex trust-region subproblems of
the form \eqref{eq:tr_cvx} and accepts candidate iterates $\tilde{w}^n$ as new iterates
$w^{n+1}$ if the acceptance criterion in ln.\ \ref{ln:suff_dec} is satisfied. The trust-region
radius $\Delta^n$ is doubled on acceptance of the computed candidate step and halved on
rejection. The trust-region radius is not increased over the value $\Delta^0$, which
stays constant over the course of the algorithm.

If the trust-region radius falls below $\underline{\Delta}^n$, then the trust-region radius is
reset to $\Delta^0$ and \cref{alg:trm} invokes solves of the non-convex trust-region
subproblems \eqref{eq:tr_gle} for successively halved trust-region radii until a
computed iterate can be accepted or the trust-region radius falls below $\underline{\Delta}^n$
again. In the former case, the trust-region radius is reset, and \cref{alg:trm} returns
to solving trust-region subproblems of the form \eqref{eq:tr_cvx}. In the latter case,
the trust-region is reset as well, $\underline{\Delta}^n$
is divided by two, and the regularization parameter $\varepsilon^n$
of the Ginzburg--Landau energy is divided by $r > 4$.
Then \cref{alg:trm} also returns to solving
instances of type \eqref{eq:tr_cvx}.

In this way, \cref{alg:trm} ensures that a closer approximation of the perimeter
regularizer $C_0P_\Omega$ by means of a tighter Ginzburg--Landau energy is only triggered
if \cref{alg:trm} cannot improve the objective further with steps larger than
$\underline{\Delta}^n$ and that are computed by means of the non-convex trust-region
subproblem \eqref{eq:tr_gle}. We will see that this procedure allows us to relate a
first-order sufficient decrease with respect to \eqref{eq:p_abstract} / \eqref{eq:tr_sharp}
to an eventual acceptance of a solution of \eqref{eq:tr_gle} in the analysis of the
asymptotics of \cref{alg:trm} without having the problem that the algorithm does
not progress towards sharp interface limits because it produces infinitely many iterates for
one $\varepsilon^n$.

This, in turn, allows us to prove the convergence of iterates produced by \cref{alg:trm} to
L-stationary points of \eqref{eq:p_abstract}. We have failed to prove the convergence of
iterates produced by \cref{alg:trm} to L-stationary points of \eqref{eq:p_abstract} when relying on
subproblem solves of type \eqref{eq:tr_cvx}, which are much less expensive after
discretization due to their convexity. However, we have added the preferred solution of instances
of \eqref{eq:tr_cvx} in order to show how to accelerate the solution process
of many of the trust-region subproblems and, thereby, the overall algorithm if better results cannot be had
in the future.

\begin{algorithm}[ht]
\caption{Homotopy Trust-region Algorithm}\label{alg:trm}
\begin{flushleft}
\textbf{Input:} $j$ sufficiently regular, $\rho \in (0,1)$,\\
\textbf{Input:} $\Delta^0 > 0$,  $\varepsilon^0 > 0$, $r > 4$, $\rho > \kappa^0 > 0$,
$\underline{\Delta}^0 \in (0,\Delta^0)$,\\
\textbf{Input:} $w^0 \in H^1(\Omega) \cap \{ w \in L^\infty(\Omega) \,|\, w(x) \in \{0,1\}\ \text{a.e.} \}$.
\end{flushleft}
\begin{algorithmic}[1]
	\State $\cvxflag \gets 1$
	\For{$n = 0,\ldots$}
	\If{$\cvxflag = 1$}
	\State\label{ln:tr_step_cvx} $\tilde{w}^n \gets$ minimizer of \ref{eq:tr_cvx}$(w^n, \Rgrad j(w^{n}), \Delta^n)$
	\Else
	\State\label{ln:tr_step_gle} $\tilde{w}^n \gets$ minimizer of \ref{eq:tr_gle}$(w^n, \Rgrad j(w^{n}), \Delta^n)$
	\EndIf	
	\If{sufficient decrease \eqref{eq:suffdec} and $\ared(w^n,\tilde{w}^n,\varepsilon^n) > \kappa^n \underline{\Delta}^n$}\label{ln:suff_dec}
	\State $w^{n+1} \gets \tilde{w}^n$
	\State $\Delta^{n+1} \gets \min\{\Delta^0,2 \Delta^n\}$
	\State $\cvxflag \gets 1$
	\Else
	\State $w^{n+1} \gets w^n$
	\State $\Delta^{n+1} \gets \Delta^n / 2$
	\EndIf
	\If{$\Delta^{n+1} < \underline{\Delta}^n$ \text{ and } $\cvxflag = 1$}
	\State $\cvxflag \gets 0$
	\State $\Delta^{n+1} \gets \Delta^0$
	\State $\underline{\Delta}^{n+1} \gets \underline{\Delta}^n$
	\State $\varepsilon^{n+1} \gets \varepsilon^n$
	\State $\kappa^{n+1} \gets \kappa^n$
	\ElsIf{$\Delta^{n+1} < \underline{\Delta}^n$}\label{ln:reduce_epsilon}
	\State $\cvxflag \gets 1$
	\State $\Delta^{n+1} \gets \Delta^0$
	\State $\underline{\Delta}^{n+1} \gets \underline{\Delta}^n / 2$
	\State $\varepsilon^{n+1} \gets \varepsilon^n / r$
	\State $\kappa^{n+1} \gets \kappa^n / 2$	    
	\Else
	\State $\underline{\Delta}^{n+1} \gets \underline{\Delta}^n$
	\State $\varepsilon^{n+1} \gets \varepsilon^n$
	\State $\kappa^{n+1} \gets \kappa^n$	    
	\EndIf
	\EndFor
\end{algorithmic}
\end{algorithm}

For the sufficient-decrease condition in \cref{alg:trm} ln.\ \ref{ln:suff_dec}, we consider
$\rho \in (0,1)$ and choose the acceptance criterion
\begin{gather}\label{eq:suffdec}
\ared(w^n,\tilde{w}^n,\varepsilon^n) \ge \rho \pred(w^n, \Delta^n, \varepsilon^n)
\end{gather} 
with the \emph{actual reduction} $\ared$ of the objective that is defined by
\[ \ared(w, v, \varepsilon) \coloneqq j(w) + \gamma E_{\varepsilon}(w) 
- j(v) - \gamma E_{\varepsilon}(v)
\]
and the \emph{predicted reduction} $\pred$ that is defined by the solution of the trust-region subproblem as
\[ \pred(w, \Delta, \varepsilon) \coloneqq -\left\{
\begin{aligned}
\text{minimal objective value of }\text{{\ref{eq:tr_cvx}}}(w, \Rgrad j(w), \Delta) & \text{ if } \cvxflag = 1,\\
\text{minimal objective value of }\text{{\ref{eq:tr_gle}}}(w, \Rgrad j(w), \Delta) & \text{ if } \cvxflag = 0.
\end{aligned}
\right.
\]

\section{Asymptotics of \cref{alg:trm}}\label{sec:asymptotics}

In order to analyze the limits of iterates generated
by \cref{alg:trm}, we need to ensure their
existence and feasibility (that is, their
binarity) for \eqref{eq:p_abstract}, which is
the purpose of \cref{sec:compactness}.
We continue by proving that the limits produced
by \cref{alg:trm} are L-stationary for 
\eqref{eq:p_abstract} in the sense of
\cref{dfn:stationarity} under certain
regularity conditions in
\cref{sec:asymptotic_stationarity}.

\subsection{Boundedness of iterates and compactness}\label{sec:compactness}

A standard way of proving boundedness of the iterates
produced by optimization algorithms and in turn weak-$^*$ convergence
of a subsequence is to combine the boundedness
of the objective from below with descent properties of the produced
sequence of iterates. In our situation, descent of the objective
values is only guaranteed for a fixed value of $\varepsilon^n$ so
that the objective value may increase when $\varepsilon^n$ is reduced
and thus the objective term itself changes from one iteration to the next.
We recover convergence of subsequences of the iterates in $L^1(\Omega)$
by combining a compactness result with competitor constructions that
use sufficient descent properties of the algorithm during the iterations
when $\varepsilon^n$ is not changed to ensure that the objective values
have to remain bounded despite the changes of the objective term itself.
We briefly state the compactness result due to~\cite{fonseca1989gradient},
on sequences $w^\varepsilon$ that are bounded with respect to
the functionals $j + \gamma E_\varepsilon$ for $\varepsilon \searrow 0$, 
which guarantees the convergence of a subsequence of iterates produced by
\cref{alg:trm} to a binary-valued limit under boundedness
of $(E_\varepsilon(w^\varepsilon))_\varepsilon$.
\begin{proposition}\label{prp:compactness}
Let $\varepsilon \searrow 0$,
$(w^\varepsilon)_\varepsilon \subset H^1(\Omega)$
with $0 \le w^\varepsilon(x) \le 1$ for a.e.\ $x \in \Omega$.
Let there exist $C > 0$ such that $j(w^\varepsilon) + \gamma E_\varepsilon(w^\varepsilon) \le C$
for all $\varepsilon$. Then $w^\varepsilon$ has a subsequence that converges
in $L^p(\Omega)$ for all $p \in [1,\infty)$ and the $L^p(\Omega)$-limit
is $\{0,1\}$-valued and satisfies
$P_\Omega(w^{-1}(\{1\})) \le \liminf_{\varepsilon\to 0} E_\varepsilon(w^\varepsilon)$.
\end{proposition}
\begin{proof}
The existence of an $L^1(\Omega)$-limit follows from the compactness result
\cite[Theorem\,4.1]{fonseca1989gradient} and the uniform boundedness of
$F$ on $\{ w \in L^\infty(\Omega)\,|\, 0 \le w(x) \le 1 \text{ for a.e.\ } x \in \Omega\}$.
The $L^\infty$-bounds on the $w^\varepsilon$ also yield convergence in $L^p(\Omega)$
for all $p \in [1,\infty)$ and the lower semi-continuity is due to
\cite[Proposition 1]{modica1987gradient}.
\end{proof}
In order to leverage \cref{prp:compactness} in the convergence analysis of \cref{alg:trm}, we need the existence of a bounded
subsequence. We pose this as an assumption below and then proceed by asserting it under appropriate assumptions.
\begin{assumption}\label{ass:algorithm}
For $k \in \N$, let $n(k) \in \N$ be the last iteration
such that $\varepsilon^{n(k)} = \varepsilon^0 r^{-k}$.
We assume that the subsequence $(w^{n(k)})_k$ of
the iterates $(w^n)_n$ produced by \cref{alg:trm} is bounded.
\end{assumption}
\Cref{ass:algorithm} is well defined if $n(k)$
exists for all $k$, which we prove in
\cref{lem:finitely_many_iterations_per_epsilon}
below.
\begin{lemma}\label{lem:finitely_many_iterations_per_epsilon}
Let $n_0 \in \N$ be the first iteration
of \cref{alg:trm} with $\varepsilon^n = \varepsilon^{n_0}$. Then there are only
finitely many iterations $n \ge n_0$ such that
$\varepsilon^{n} = \varepsilon^{n_0}$.
\end{lemma}
\begin{proof}
We note that the iterations $n \ge n_0$
with $\varepsilon^n = \varepsilon^{n_0}$
also satisfy $\kappa^n = \kappa^{n_0}$.
Because the objective $j + \gamma E_{\varepsilon^n}$
is bounded below, there can only be finitely many iterates
that are accepted in \cref{alg:trm} ln.\ \ref{ln:suff_dec}
and thus only finitely many iterations, in which $\cvxflag$ is set
to one. The trust-region radius is halved if an iterate is not accepted, so
that $\cvxflag$ is set to zero eventually and not set to one anymore afterward.
Since there is no accepted iterate anymore, the trust-region
radius is halved with $\cvxflag = 0$
until the condition in
\cref{alg:trm} ln.\ \ref{ln:reduce_epsilon}
is met and $\varepsilon^n$ is halved. This happens after finitely many
iterations because $\underline{\Delta}^n > 0$
holds inductively.
\end{proof}
The boundedness of the iterates that allows to
apply \cref{prp:compactness} can be ensured under boundedness 
assumptions on $j$ and its gradient, which we prove below.

After the auxiliary result \cref{lem:nonincreasing_for_const_eps}, \cref{prp:boundedness_simple}
shows the boundedness of the sequence of iterates for the case that the reset trust-region
radius $\Delta^0$ is so large that all of $\Omega$ is contained in the trust region
after the reset occurs. In particular, the zero function is feasible and satisfies
$\gamma E_\varepsilon(0) = 0$. Consequently, if the penalty parameter becomes
small, the regularization dominates the objective in the subproblem in the event of an
unbounded sequence so that the function $0$ is a competitor that reduces the subproblem
solution to lead to a contradiction.
Then \cref{prp:boundedness_involved} shows the boundedness
for smaller values of $\Delta^0$. The proof is based on the same general idea as
\cref{prp:boundedness_simple} but the function $0$ may not be feasible and
therefore, we provide a more involved construction and estimates for a 
feasible point of the trust-region subproblems for the initial
trust-region radius $\Delta^0$ that satisfies our needs.
\begin{lemma}\label{lem:nonincreasing_for_const_eps}
\Cref{alg:trm} produces a sequence of iterates with non-increasing objective values
over subsequent iterations during which $\varepsilon^n$ is not altered.
\end{lemma}
\begin{proof}
The iterate $w^{n+1}$ coincides with the previous iterate $w^n$ if the computed update $\tilde{w}^n$
is rejected. Thus, under the condition $\varepsilon^{n+1} = \varepsilon^n$, the objective value only changes
if a step is accepted. Specifically, it changes by $\ared(w^n, w^{n+1}, \varepsilon^n)$.
Because \cref{prp:tr_existence} asserts that $\pred(w^n, \Delta^n, \varepsilon^n)$ is non-negative,
the sufficient decrease condition \eqref{eq:suffdec} and \cref{alg:trm} ln.\ \ref{ln:suff_dec}
yield the claim.
\end{proof}
\begin{proposition}\label{prp:boundedness_simple}
Let $j$ and $\nabla_R j$ be bounded on the unit ball in $L^\infty(\Omega)$.
Let $(w^n)_n \subset \{ w \in H^1(\Omega)\,|\, 0 \le w(x) \le 1 \text{ for a.e.\ } x \in \Omega\}$
be the sequence of iterates produced by \cref{alg:trm}. Let $\Delta^0 \ge |\Omega|$.
Then there exist $C > 0$ and
a subsequence  $(w^{n_k})_k \subset (w^n)_n$ such that
\[ \sup_{k \in \N} j(w^{n_k}) + \gamma E_{\varepsilon^{n_k}}(w^{n_k}) < C < \infty. \]
\end{proposition}
\begin{proof}
We define constants for uniform bounds on the value of $j$ and the first term of the objective of the trust-region subproblem:
\begin{align*}
C_j    &\coloneqq 
\sup\{ |j(w)| \,|\, w \in L^2(\Omega) \text{ and } \|w\|_{L^\infty} \le 1\} \text{ and}\\
C_{g} &\coloneqq 
\sup\{ (\Rgrad j(w),d)_{L^2}\,|\, w, d \in L^2(\Omega) 
\text{ and } \|w\|_{L^\infty} \le 1, \|d\|_{L^\infty} \le 1 \}.
\end{align*}
Because of the assumption $\Delta^0 \ge |\Omega|$, the trust-region constraint is not present in
(or does not restrict) the trust-region subproblems of the form
$\text{{\ref{eq:tr_gle}}}(w^n, \nabla_R j(w^n), \Delta^0)$ and
$\text{{\ref{eq:tr_cvx}}}(w^n, \nabla_R j(w^n), \Delta^0)$,
which implies that all $w \in H^1(\Omega)$ with $0 \le w(x) \le 1$ for a.e.\ $x \in \Omega$ are feasible
with finite objective value for the trust-region subproblems. We obtain
\[ \gamma E_{\varepsilon^n}(w^n) + C_{g} \ge \pred(w^n,\Delta^0,\varepsilon^n)
\ge \gamma E_{\varepsilon^n}(w^n) - C_{g}
\]
because $0$ is feasible and $E_{\varepsilon^n}(0) = 0$ and for any solution 
$\bar{w}^n \in \argmin \text{{\ref{eq:tr_gle}}}(w^n, \nabla_R j(w^n), \Delta^0)$
and $\cvxflag = 1$ that
\[ \ared(w^n,\bar{w}^n,\varepsilon^n) \ge 
\gamma E_{\varepsilon^n}(w^n) - \gamma E_{\varepsilon^n}(\bar{w}^n) - 2 C_j
\ge \gamma E_{\varepsilon^n}(w^n) - 2 C_g - 2 C_j,
\]
where the second inequality follows from the optimality of $\bar{w}^n$,
which gives
\[ (\Rgrad j(w^n), \bar{w}^n - w^n)_{L^2} + \gamma E_{\varepsilon^n}(\bar{w}^n)
\le -(\Rgrad j(w^n), w^n)_{L^2}
\]   
and thus
\[ -2C_g \le -\gamma E_{\varepsilon^n}(\bar{w}^n). \]
Consequently, if $\gamma E_{\varepsilon^n}(w^n) \ge C_E \coloneqq (1 - \rho)^{-1}\left((2 + \rho) C_g + 2 C_j\right)$,
we have
\begin{gather}\label{eq:accept_for_gE_large_enough}
\frac{\ared(w^n,\bar{w}^n,\varepsilon^n)}{\pred(w^n,\Delta^0,\varepsilon^n)}
\ge \frac{\gamma E_{\varepsilon^n}(w^n) - 2 C_g - 2 C_j}{\gamma E_{\varepsilon^n}(w^n) + C_{g}}
\ge \rho.
\end{gather}
\Cref{lem:nonincreasing_for_const_eps} and the design of \cref{alg:trm}
give that the objective can only increase if $\varepsilon^n = \varepsilon^{n - 1}/r$. 
In this case, the trust-region radius was reduced below $\underline{\Delta}^{n-1}$, implying that
$w^n = w^{n-1}$ by design of \cref{alg:trm}.

This means: if there is some iteration such that
\[ j(w^n) + \gamma E_{\varepsilon^n}(w^n) \ge j(w^n) + C_E > j(w^{n}) + \gamma E_{\varepsilon^{n-1}}(w^n) \]
and thus also $\gamma E_{\varepsilon^n}(w^n) \ge C_E$ for some $n \in \N$, then
\cref{lem:nonincreasing_for_const_eps,lem:finitely_many_iterations_per_epsilon}
give that there is some $k \in \N$ such that $\varepsilon^{n} = \ldots = \varepsilon^{n + k}$
\[ j(w^n) + \gamma E_{\varepsilon^n}(w^n)
\ge \ldots \ge 
j(w^{n+k}) + \gamma E_{\varepsilon^n}(w^{n+k}) \]
with $\cvxflag = 1$ and in the next iteration $n + k + 1$ we have
\[ \varepsilon^{n+k+1} = \varepsilon^{n+k} \text{ and } \Delta^{n+k+1} = \Delta^0
\text{ and } \cvxflag = 1.
\]
Then \eqref{eq:accept_for_gE_large_enough} implies that $\bar{w}^{n+k+1}$ is accepted
and the objective is reduced by at least $C_E - 2 C_g - 2 C_j$.
Such iterations occur consecutively
at least until the value of $\gamma E_{\varepsilon^{n+k+1}}(w^{n+k+1})$ falls below
$C_E$ because the algorithm reduces the objective until it sets $\cvxflag$ to $1$ and the trust-region radius
to $\Delta^0$ after further finitely many iterations,
in which case \eqref{eq:accept_for_gE_large_enough} implies a further acceptance and subsequent reduction of
the objective by at least $C_E - 2 C_g - 2 C_j$.

Because $j(w^n) \le C_j$ holds for all $n \in \N$, the considerations above yield that there is a subsequence,
indexed by $k$, such that $j(w^{n_k}) + \gamma E_{\varepsilon^{n_k}}(w^{n_k}) \le C_j + C_E$.
\end{proof}
\begin{proposition}\label{prp:boundedness_involved}
Let $\Omega$ be convex.
Let $j$ and $\nabla_R j$ be bounded on the unit ball in $L^\infty(\Omega)$.
Let $(w^n)_n \subset \{ w \in H^1(\Omega)\,|\, 0 \le w(x) \le 1 \text{ for a.e.\ } x \in \Omega\}$
be the sequence of iterates produced by \cref{alg:trm}.
Then there exist $C > 0$ and
a subsequence  $(w^{n_k})_k \subset (w^n)_n$ such that
\[ \sup_{k \in \N} j(w^{n_k}) + \gamma E_{\varepsilon^{n_k}}(w^{n_k}) < C < \infty. \]
\end{proposition}
\begin{proof}
We consider the case $\Delta^0 < |\Omega|$ and define the ratio $R \coloneqq \left\lceil \frac{|\Omega|}{\Delta^0} \right\rceil + 1$.
Without loss of generality, that is, after possibly applying a suitable translation to $\Omega$, we
assume that $\inf \{x_1\,|\,x \in \Omega\} = 0$. Moreover, we define
$\ell \coloneqq \sup\{x_1 \,|\, x \in \Omega\}$. Then, by virtue of Fubini's theorem,
we obtain
\[ |\Omega| = \int_{0}^{\ell} \lambda_{\R^{d-1}}(\Omega_t) \dd t
\text{ with } \Omega_t = \{ (a + t e_1)_{2,\ldots,d} \,|\, a \in \{e_1\}^\perp \text{ and }
a + t e_1 \in \Omega \}
\text{ for } t \in [0,\ell]
\]
and where $\lambda_{\R^{d-1}}$ denotes the Lebesgue measure on $\R^{d-1}$. The function
$s\mapsto \int_{0}^{s} \lambda_{\R^{d-1}}(\Omega_t) \dd t$ is continuous and we thus obtain
$0 = s_0 \le s_a < s_b \le s_L = \ell$ so that
\[ |S| = \int_{s_{a}}^{s_{b}} \lambda_{\R^{d-1}}(\Omega_t) \dd t = \frac{|\Omega|}{R}
\]
for a slice $S \coloneqq \bigcup_{t=s_{a}}^{s_b} \Omega_t$. All such slices $S$ have a finite perimeter because we may estimate
their perimeter from above by $P_{\R^d}(\Omega) + 2 (\diam \Omega)^{d-1}$, which is finite.
All such slices $S$ are also convex because they are intersections of two half-spaces with the convex set $\Omega$.
We will approximate $\chi_S$ smoothly along $e_1$ following the ideas in~\cite{modica1987gradient,modica5esempio}. Then, we will use this approximation
in order to construct a competitor that is feasible for the trust-region subproblem
and reduces the objective value substantially enough to contradict that
the sequence of iterates produced by \cref{alg:trm} is unbounded with respect
to the objective values.

To this end, we first analyze $S$ independently of its position along
$e_1$. Our required choices for $S$ will be defined later. We construct a set $\hat{S}$ from $S$ by extending it in the directions in $\{e_1\}^\perp$ in order to obtain a cylinder with positive distance of its edges to $\Omega$ and then smoothing the edges. By making this
extension large enough, these properties can be established independently of the specific choice of $S$. This set satisfies $\hat{S} \cap \Omega = S$, $\Ha^{d-1}(\partial \hat{S} \cap \partial \Omega) = 0$, and
\begin{gather}\label{eq:perimeter_equality_S_Shat}
P_\Omega(\hat{S}) = \Ha^{d-1}(\partial \hat{S} \cap \Omega) = \Ha^{d-1}(\partial S \cap \Omega) = P_\Omega(S).
\end{gather}

The situation is sketched in \cref{fig:sketch_set}.
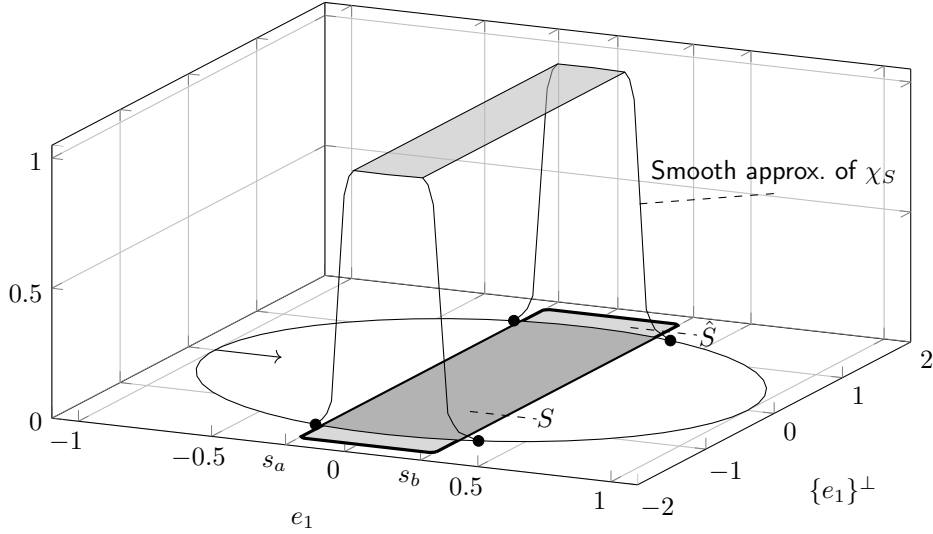
\begin{figure}[h]
\begin{center}
\begin{tikzpicture}
\begin{axis}[
height=8cm,
width=13cm,
grid,
xmin=-1.1,
xmax=1.1,
ymin=-2,
ymax=2,
zmin=0.0,
zmax=1.05,
extra x ticks={-0.2225209,0.2845276},
extra x tick labels={$s_a$,$s_b$},	 
xlabel={$e_1$},
ylabel={$\{e_1\}^\perp$} 
]

\addplot3[draw=black,mark=*]
coordinates {
(0.3500000,1.4051245,0)
(0.3500000,-1.4051245,0)
};

\addplot3[draw=black,mark=*]
coordinates {
(-0.2500000,1.4523688,0)
(-0.2500000,-1.4523688,0)
};

\addplot3[draw=black,fill=white] coordinates
{
(1.0000000,0.0000000,0)
(0.9917900,0.1918157,0)
(0.9672949,0.3804819,0)
(0.9269168,0.5629005,0)
(0.8713187,0.7360763,0)
(0.8014136,0.8971658,0)
(0.7183494,1.0435238,0)
(0.6234898,1.1727472,0)
(0.5183926,1.2827141,0)
(0.4047833,1.3716189,0)
(0.2845276,1.4380018,0)
(0.1595999,1.4807727,0)
(0.0320516,1.4992293,0)
(-0.0960230,1.4930687,0)
(-0.2225209,1.4623919,0)
(-0.3453651,1.4077026,0)
(-0.4625383,1.3298990,0)
(-0.5721167,1.2302584,0)
(-0.6723009,1.1104170,0)
(-0.7614460,0.9723426,0)
(-0.8380881,0.8183024,0)
(-0.9009689,0.6508256,0)
(-0.9490557,0.4726623,0)
(-0.9815592,0.2867379,0)
(-0.9979454,0.0961053,0)
(-0.9979454,-0.0961053,0)
(-0.9815592,-0.2867379,0)
(-0.9490557,-0.4726623,0)
(-0.9009689,-0.6508256,0)
(-0.8380881,-0.8183024,0)
(-0.7614460,-0.9723426,0)
(-0.6723009,-1.1104170,0)
(-0.5721167,-1.2302584,0)
(-0.4625383,-1.3298990,0)
(-0.3453651,-1.4077026,0)
(-0.2225209,-1.4623919,0)
(-0.0960230,-1.4930687,0)
(0.0320516,-1.4992293,0)
(0.1595999,-1.4807727,0)
(0.2845276,-1.4380018,0)
(0.4047833,-1.3716189,0)
(0.5183926,-1.2827141,0)
(0.6234898,-1.1727472,0)
(0.7183494,-1.0435238,0)
(0.8014136,-0.8971658,0)
(0.8713187,-0.7360763,0)
(0.9269168,-0.5629005,0)
(0.9672949,-0.3804819,0)
(0.9917900,-0.1918157,0)
(1.0000000,-0.0000000,0)
};

\addplot3[draw=black, very thick, fill=black!15!white] coordinates 
{
	(0.2845276,-1.770018,0)
	(0.2845276,1.770018,0)
	(0.2795276,1.782518,0)	
	(0.2695276,1.7975018,0)
	(0.2645276,1.800018,0)	
	(-0.2025209,1.800018,0)
	(-0.2075209,1.7975018,0)
	(-0.2175209,1.782518,0)			
	(-0.2225209,1.770018,0)
	(-0.2225209,-1.770018,0)	
	(-0.2175209,-1.782518,0)
	(-0.2075209,-1.7975018,0)	
	(-0.2025209,-1.800018,0)	
	(0.2645276,-1.800018,0)		
	(0.2695276,-1.7975018,0)	
	(0.2795276,-1.782518,0)					
	(0.2845276,-1.770018,0)	
};

\addplot3[draw=black, fill=black!30!white] coordinates 
{
	(0.2845276,1.4380018,0)
	(0.1595999,1.4807727,0)
	(0.0320516,1.4992293,0)
	(-0.0960230,1.4930687,0)
	(-0.2225209,1.4623919,0)
	(-0.2225209,-1.4623919,0)
	(-0.0960230,-1.4930687,0)
	(0.0320516,-1.4992293,0)
	(0.1595999,-1.4807727,0)
	(0.2845276,-1.4380018,0)
	(0.2845276,1.4380018,0)
};

\addplot3[draw=black] coordinates 
{
	(0.1595999,1.480772673917228,1)
	(0.178339055,1.4759536436788387, 0.975)
	(0.19707820999999998,1.4705816546768802,0.9)
	(0.22206375,1.4625482230018467,.5)
	(0.24704928999999998,1.4535043717507752,0.1)
	(0.265788445,1.446047070684452,0.06)
	(0.2845276,1.4380018,0.03)
	(0.3500000,1.4051245,0)
};

\addplot3[draw=black] coordinates 
{
	(-0.2500000,1.4523688,0)
	(-0.2225209,1.4623919,0.025)
	(-0.203546215,1.4685980087512516,0.05)
	(-0.18457152999999998,1.4742286248086771,0.1)
	(-0.15927195,1.4808521206967948,0.5)
	(-0.13397237,1.4864776013019207,0.9)
	(-0.114997685,1.4900486394747126,0.975)
	(-0.0960230,1.4930687,1)
};

\addplot3[draw=black] coordinates 
{
	(0.1595999,-1.480772673917228,1)
	(0.178339055,-1.4759536436788387, 0.975)
	(0.19707820999999998,-1.4705816546768802,0.9)
	(0.22206375,-1.4625482230018467,.5)
	(0.24704928999999998,-1.4535043717507752,0.1)
	(0.265788445,-1.446047070684452,0.06)
	(0.2845276,-1.4380018,0.03)
	(0.3500000,-1.4051245,0)
};

\addplot3[draw=black] coordinates 
{ 
	(-0.2500000,-1.4523688,0)
	(-0.2225209,-1.4623919,0.025)
	(-0.203546215,-1.4685980087512516,0.05)
	(-0.18457152999999998,-1.4742286248086771,0.1)
	(-0.15927195,-1.4808521206967948,0.5)
	(-0.13397237,-1.4864776013019207,0.9)
	(-0.114997685,-1.4900486394747126,0.975)
	(-0.0960230,-1.4930687,1)
};

\addplot3[draw=black, fill=black!30!white,fill opacity=0.5] coordinates 
{ (0.1595999,1.4807727,1)
	(0.0320516,1.4992293,1)
	(-0.0960230,1.4930687,1)
	(-0.0960230,-1.4930687,1)
	(0.0320516,-1.4992293,1)
	(0.1595999,-1.4807727,1)
	(0.1595999,1.4807727,1)
};

\addplot3[->,draw=black] coordinates
{
	(-1,0,0)
	(-0.75,0,0) 
};

\addplot3[black,point meta=explicit symbolic,nodes near coords] coordinates {(0.53,-1.125,0)[$S$]};

\addplot3[draw=black,dashed] coordinates 
{ 
	(0.15,-0.75,0)
	(0.4,-0.75,0) 
};

\addplot3[black,point meta=explicit symbolic,nodes near coords] coordinates {(0.55,1.17,0)[$\hat{S}$]};

\addplot3[draw=black,dashed] coordinates 
{ 
	(0.15,1.6,0)
	(0.4,1.6,0) 
};

\addplot3[black,point meta=explicit symbolic,nodes near coords] coordinates {(0.8,1.17,0.65)[Smooth approx.\ of $\chi_{S}$]};
\addplot3[draw=black,dashed] coordinates 
{ 
	(0.215,1.4777303000000002,0.5)
	(0.8,1.17,0.65) 
};

\end{axis}
\end{tikzpicture}
\end{center}
\caption{Positioning of $S$ (dark grey) with $s_a$
and $s_b$ inside $\Omega$ (white ellipse)
along $e_1$, extended set $\hat{S}$ (light gray, extending to the exterior of $\Omega$),
and smooth approximation of $\chi_S$.}\label{fig:sketch_set}
\end{figure}
To this end, we define the signed distance function as 
\[ h(x) \coloneqq \left\{
\begin{aligned}
\dist(x, \partial \hat{S} \cap \Omega) & \text{ if } x \in \hat{S} \\
- \dist(x, \partial \hat{S} \cap \Omega) & \text{ else}
\end{aligned}
\right. \]
We approximate the jump function on $\R$, $\chi_{(-\infty,0)}$, smoothly with a function $q$ that solves
the ordinary differential equation
\[ -q''(s) + W'(q(s)) = 0 \text{ for } s \in \R \]
with 
\[ q'(s) = \sqrt{2 W(q(s))} > 0, \]
where $W(q) = q(1 - q)$. To make the solution unique, we set $q(0) = \frac{1}{2}$.
Let $\tau > 0$ be fixed but arbitrary and let $a = \frac{q'(\tau)^2}{4q(\tau)}$ and $b = 2 \frac{q(\tau)}{q'(\tau)}$.
We define 
\[ q_\tau(s) \coloneqq \left\{
\begin{aligned}
0    & \text{ if } s \in (-\infty,-\tau - b], \\
\tilde{q}_\tau(s) & \text{ if } s \in (-\tau - b,-\tau), \\
q(s) & \text{ if } s \in [-\tau,\tau], \\
1 - \tilde{q}_\tau(-s) & \text{ if } s \in (\tau,\tau + b), \\
1  & \text{ if } s \in [\tau + b,\infty), \\
\end{aligned}
\right. \]
where $\tilde{q}_\tau(s) = a(s - (-\tau - b))^2$ is a parabola with vertex at $(-\tau - b)$ such that
$q_\tau \in C^2(\R)$ and $\supp q_\tau' = [-\tau - b,\tau + b]$.
We define $v_{\delta,\tau} \coloneqq q_\tau \circ (h / \delta)$ and obtain
$\frac{\delta}{2}|\nabla v_{\delta,\tau}(x)|^2 = \frac{1}{2\delta} (q_\tau'(h(x)/\delta))^2$
because the chain rule gives $\nabla v_{\delta,\tau}(x) = \frac{1}{\delta} q_{\tau}'(h(x)/\delta) \nabla h(x)$
and $\nabla h(x)$ exists and satisfies $\|\nabla h(x)\| = 1$ a.e. Moreover, we can prove the following assertion.

\textbf{Claim:}
\begin{gather}\label{eq:claim1} 
\begin{multlined}I_{\delta,\tau} \coloneqq \int_S \frac{\delta}{2} |\nabla v_{\delta,\tau}(x)|^2 + \frac{1}{\delta}W(v_{\delta,\tau}(x)) \dd x\\
\to \left(\int_{\R} \frac{1}{2} q_\tau'(s)^2 + W(q_\tau(s))\dd s\right) P_\Omega(S)
\text{ for } \delta \searrow 0.
\end{multlined}
\end{gather}
and 
\begin{gather}\label{eq:claim1a} 
\begin{multlined}I^c_{\delta,\tau} \coloneqq \int_{\Omega \setminus S} \frac{\delta}{2} |\nabla v_{\delta,\tau}(x)|^2 + \frac{1}{\delta}W(v_{\delta,\tau}(x)) \dd x\\
\to \left(\int_{\R} \frac{1}{2} q_\tau'(s)^2 + W(q_\tau(s))\dd s\right) P_\Omega(S)
\text{ for } \delta \searrow 0.
\end{multlined}
\end{gather}
\textbf{Proof of Claim:}
Because the length of the support interval of $s \mapsto q_{\tau}(s/\delta)$ shrinks linearly with $\delta$
we can assume that $\delta$ is small enough such that the length of the support is strictly smaller than
the halved distance between the two half-spaces that cut $S$ and $\hat{S}$ into two pieces of width $(s_b - s_a) / 2$.
We combine this insight with Fubini's theorem and obtain
\begin{align*}
I_{\delta,\tau} &=
\int_{ \{e_1\}^\perp} \int_{-\infty}^{\frac{s_a + s_b}{2}} \chi_S(y + t e_1) 
\left(\frac{1}{2\delta}q_\tau'(h(y + t e_1)/\delta)^2 + \frac{1}{\delta}W(q_{\tau}(h(y + t e_1)/\delta))\right)
\dd t \dd \Ha^{d-1}(y) \\
&\enskip + \int_{ \{e_1\}^\perp} \int_{\frac{s_a + s_b}{2}}^\infty \chi_S(y + t e_1) 
\left(\frac{1}{2\delta}q_\tau'(h(y + t e_1)/\delta)^2 + \frac{1}{\delta}W(q_{\tau}(h(y + t e_1)/\delta))\right)
\dd t \dd \Ha^{d-1}(y) \\
&= \int_{ \{e_1\}^\perp} \int_{0}^{\frac{s_b - s_a}{2}} \chi_S(y + s_a + t e_1) 
\left(\frac{1}{2\delta}q_\tau'(t/\delta)^2 + \frac{1}{\delta}W(q_{\tau}(t/\delta))\right)
\dd t \dd \Ha^{d-1}(y)\\
&\enskip + 
\int_{ \{e_1\}^\perp} \int_{0}^{\frac{s_b - s_a}{2}} \chi_S(y + s_b -  t e_1) 
\left(\frac{1}{2\delta}q_\tau'(t/\delta)^2 + \frac{1}{\delta}W(q_{\tau}(t/\delta))\right)
\dd t \dd \Ha^{d-1}(y)\\
&= \int_{ \{e_1\}^\perp} \int_{0}^{\frac{s_b - s_a}{2\delta}} \chi_S(y + s_a + t\delta e_1) 
\left(\frac{1}{2}q_\tau'(t)^2 + W(q_{\tau}(t))\right)
\dd t \dd \Ha^{d-1}(y)\\
&\enskip + 
\int_{ \{e_1\}^\perp} \int_{0}^{\frac{s_b - s_a}{2\delta}} \chi_S(y + s_b - t\delta e_1) 
\left(\frac{1}{2}q_\tau'(t)^2 + W(q_{\tau}(t))\right)
\dd t \dd \Ha^{d-1}(y)\\
&=  \int_{-\infty}^{\infty} \int_{\{e_1\}^\perp} \chi_S(y + s_a + t\delta e_1)\dd \Ha^{d-1}(y)
\left(\frac{1}{2}q_\tau'(t)^2 + W(q_{\tau}(t))\right) \dd t\\
&\enskip + 
\int_{-\infty}^{\infty} \int_{\{e_1\}^\perp} \chi_S(y + s_b - t\delta e_1) \dd \Ha^{d-1}(y)
\left(\frac{1}{2}q_\tau'(t)^2 + W(q_{\tau}(t))\right)
\dd t
\end{align*}
for all $\delta$ small enough such that $\frac{s_b - s_a}{2\delta} > \tau + b$, which implies $q_{\tau}'(t) = 0$
and $W(q_\tau(t)) = 0$ for all $t > \frac{s_b - s_a}{2\delta}$. Note that we have used that $h$ is defined with respect
to $\hat{S}$ in order to obtain the second equality. For $t \in [0,\tau + b]$, we consider the term
\[ f_\delta(t) \coloneqq \int_{\{e_1\}^\perp} \chi_S(y + s_b - t\delta e_1) + \chi_S(y + s_b - t\delta e_1)  \dd \Ha^{d-1}(y). \]
Again, because $h$ is defined with respect to $\hat{S}$ and not $S$, we observe for all $\delta$ small enough that
\[ f_\delta(t) = \Ha^{d-1}(\{x \in S\,|\, h(x) = \delta t\}) = \Ha^{d-1}(\Omega \cap \{x \in \hat{S}\,|\, h(x) = \delta t\}). \]
Moreover, $\hat{S}$ satisfies the assumptions of Lemma 3 in \cite{modica1987gradient}, from which we thus obtain
\[  f_\delta(t) \to \Ha^{d-1}(\partial \hat{S} \cap \Omega) \underset{\eqref{eq:perimeter_equality_S_Shat}}= P_\Omega(S),
\]
for $\delta \searrow 0$.
Moreover, $\Ha^{d-1}(\{x \in \Omega\,|\, h(x) = \delta t\}) \le \diam(\Omega)^{d-1}$ holds for all $t$ due to the axis-parallel construction
of $S$ and $\hat{S}$. Thus we may employ Lebesgue's dominated convergence theorem in order to obtain \eqref{eq:claim1}. A very similar chain of arguments 
and the equality $P_\Omega(S) = P_\Omega(\Omega\setminus S)$ yield 
\eqref{eq:claim1a}. Thus, the claim consisting of \eqref{eq:claim1} and 
\eqref{eq:claim1a} is proven.

Next, we define
\begin{align*}
C_j    &\coloneqq 
\sup\{ |j(w)| \,|\, w \in L^2(\Omega) \text{ and } \|w\|_{L^\infty} \le 1\} \text{ and}\\
C_{g} &\coloneqq 
\sup\{ (\Rgrad j(w),d)_{L^2}\,|\, w, d \in L^2(\Omega) 
\text{ and } \|w\|_{L^\infty} \le 1, \|d\|_{L^\infty} \le 1 \}
\end{align*}
and proceed with a contradictory argument towards the overall claim that there is a subsequence of iterates, which is
bounded with respect to the objective. First, \cref{lem:finitely_many_iterations_per_epsilon} and the
definition of $R$ give for any arbitrary but fixed $\tau > 0$ that there exists $n_0 \in \N$ such that for all $n \ge n_0$
it holds that
\begin{gather}\label{eq:trfeasibility_prep}
2\diam(\Omega)^{d-1}(\tau + b)\varepsilon^n + \frac{|\Omega|}{R} \le \Delta^0.
\end{gather}
If there is no bounded subsequence, we obtain for all $M > 0$ that there is $n_M \ge n_0$ so that
for all $n \ge n_M$
\[ j(w^n) + \gamma E_{\varepsilon^{n}}(w^{n}) \ge M. \] 
We can assume wlog that $n_M$ is the first such iteration, implying $\Delta^{n_M} = \Delta^0$,
because the objective can only increase in the event of a reduction of $\varepsilon$. The inequality
gives 
\[ \gamma E_{\varepsilon^{n}}(w^{n}) \ge M - C_j \]
for all $n \ge n_M$. Now, we consider finitely many sets $S_1$, $\ldots$, $S_{2R -1}$ 
as starting points for possible competitors that are constructed as follows.
We define
\[ S_i = \bigcup_{t = s_a^i}^{s_b^i} \Omega_t \]
for all $i \in \{1,\ldots,2R - 1\}$. For $i \in \{1, 3, \ldots, 2R - 1\}$, we choose
the scalars $0 \le s_a^i < s_b^i \le \ell$ such that $s_a^{i + 2} = s_b^i$
and $|S_i| = |\Omega|/R$ as above so that they decompose $\Omega$ into $R$
pieces of equal Lebesgue measure. Moreover, we construct for the remaining $R - 1$ sets $S_i$,
$i \in \{2, 4, \ldots, 2R - 2\}$, we choose scalars $0 \le s_a^i < s_b^i \le \ell$ such
that $S_i$ is \emph{centered} around $s_a^{i+1}$, that is $(s_a^{i} + s_b^i)/2 = s_a^{i+1}$
and again $|S_i| = |\Omega|/R$.

For any such set $S = S_i$, $i \in \{1,\ldots,2R - 1\}$,
we make the following considerations.
We define
$\tilde{S}^n \coloneqq \supp v_{\varepsilon^n,\tau} \cap \Omega$.
Then
\[ \|\tilde{w}^n - w\|_{L^2}^2
\le \|\tilde{w}^n - w\|_{L^1}
= \|v_{\varepsilon^n,\tau}\|_{L^1} 
\le |\tilde{S}^n|
\le |S| + 2\diam(\Omega)^{d-1}(\tau + b)\varepsilon^n + \frac{|\Omega|}{R}
\underset{\eqref{eq:trfeasibility_prep}}\le \Delta^0
\]
and in turn we obtain that the function
$\tilde{w}^n \coloneqq (1 - v_{\varepsilon^n,\tau}) w^n$ is feasible
for $\text{{\ref{eq:tr_gle}}}(w^n, \nabla_R j(w^n), \Delta^0)$
for all $n \ge n_M$,
where the third inequality follows due to the axis-parallel 
construction of $S$, which implies
$\Ha^{d-1}(\{x \in \Omega\,|\, h(x) = t\}) \le \diam (\Omega)^{d-1}$ for all $t$.
Moreover, we define
$T^n \coloneqq \{ x \in \Omega\,|\,v_{\varepsilon^n,\tau}(x) = 1\}$.
Then $|\tilde{S}^n\setminus T^n| \le \varepsilon^n 2 (\tau + b) \diam(\Omega)^{d-1}$ and
the function $\tilde{w}^n$ coincides with $w^n$ outside $\tilde{S}^n$, is zero inside $T^n$.
It transitions smoothly from $w^n$ to zero in $\tilde{S}^n\setminus T^n$.
Because a reduction of $\tau > 0$ leaves \eqref{eq:trfeasibility_prep} intact for $n \ge n_M$,
for any of the finitely many choices of $S$, we can choose some $\tau > 0$ small enough
so that $|T^n| \ge |\Omega| / (R + 1)$ holds for the corresponding $T^n$. By choosing the
minimal of these $\tau$, this inequality holds uniformly over the possible choices for
$S$. By possibly decreasing $\tau$ further, we also obtain that, independently of
the choice of $S$, the corresponding sets $T^n$ cover $\Omega$.

For all $n \ge n_M$, we can do the following. Because $|T^n| \ge |\Omega| / (R + 1)$
does not depend on the specific choice of $S$, we can choose the set $S = S_i$,
$i \in \{1,\ldots,2R - 1\}$, such that
\begin{gather}\label{eq:large_enough_Tn}
\frac{\varepsilon^n}{2}\|\nabla w\|_{L^2(T^n)}^2
+ \frac{1}{\varepsilon^n}
\int_{T^n} w^n(1 - w^n) \dd x \ge \frac{1}{R + 1}
E_{\varepsilon^n}(w^n)
\end{gather}
holds for all $n \ge n_M$. Such a choice can be made because the induced options
for $T^n$ (by the choice of $S$) cover $\Omega$. Then
\begin{multline*}
E_{\varepsilon^{n}}(\tilde{w}^{n})
= \frac{\varepsilon^n}{2}
\|\nabla w^n\|_{L^2(\Omega\setminus \tilde{S}^n)}^2
+ \frac{1}{\varepsilon^n}
\int_{\Omega\setminus \tilde{S}^n} w^n(1 - w^n) \dd x\\
+ \underbrace{\frac{\varepsilon^n}{2}
\|\nabla \tilde{w}^n\|_{L^2(\tilde{S}^n \setminus T^n)}^2
+ \frac{1}{\varepsilon^n} \int_{\tilde{S}^n \setminus T^n}
\tilde{w}^n(1 - \tilde{w}^n)
\dd x}_{\eqqcolon A^n}.
\end{multline*}

\textbf{Claim:} For all $n$ large enough, it holds that
\begin{gather}\label{eq:claim2}
E_{\varepsilon^{n}}(w^{n})
- E_{\varepsilon^{n}}(\tilde{w}^{n})
\ge 
\frac{1}{2(R + 1)}
E_{\varepsilon^{n}}(w^n).
\end{gather}
\textbf{Proof of Claim:} The product rule yields
\begin{multline}\label{eq:An_upper_bound}
A^n \le \frac{\varepsilon^n}{2}
\|\nabla w^n\|_{L^2(\tilde{S}^n \setminus T^n)}^2
+ \frac{1}{\varepsilon^n} \int_{\tilde{S}^n \setminus T^n}w^n(1 - w^n) \dd x\\
+ \underbrace{\frac{\varepsilon^n}{2}\|1 - v_{\tau,\varepsilon^n}\|_{L^2(\tilde{S}^n)}^2
+ \frac{1}{\varepsilon^n}\int_{\tilde{S}^n} (1 - v_{\tau,\varepsilon^n})v_{\tau,\varepsilon^n} \dd x}_{\eqqcolon B^n},
\end{multline}
where the third and fourth term satisfy
\begin{align*}
B^n
&= 
\frac{\varepsilon^n}{2}\|\nabla v_{\tau,\varepsilon^n}\|_{L^2(\tilde{S}^n)}^2
+ \frac{1}{\varepsilon^n}\int_{\tilde{S}^n} v_{\tau,\varepsilon^n}(1 - v_{\tau,\varepsilon^n}) \dd x\\
&= I_{\varepsilon^n,\tau} + I_{\varepsilon^n,\tau}^c 
\le 3 \left(\int_{\R} \frac{1}{2} q_\tau'(s)^2 + W(q_\tau(s))\dd s\right)
\left(\diam(\Omega)^{d-1} + P_{\R^d}(\Omega)\right) \eqqcolon C_S
\end{align*}
for all large $n \ge \tilde{n}$ for some $\tilde{n} \in \N$
by means of \eqref{eq:claim1} and \eqref{eq:claim1a}. Again, the uniform $\tilde{n} \in \N$
exists because there are only finitely many possible realizations of $S$ considered and in
turn only finitely many sequences $(I_{\varepsilon^n,\tau})_n$ and $(I_{\varepsilon^n,\tau}^c)_n$
need to be taken into account.

By possibly choosing $\tilde{n}$ larger, we obtain for all $n \ge \max\{n_M,\tilde{n}\}$ that
\begin{gather}\label{eq:1mvtaueps_uniform_estimate}
B^n \le C_S \le 
\frac{1}{2(R + 1)} \left(
\frac{\varepsilon^n}{2}
\|\nabla w^n\|^2_{L^2(T^n)}
+ \frac{1}{\varepsilon^n}
\int_{T^n} w^n(1 - w^n) \dd x\right).
\end{gather}
We combine our insights so that we compute and estimate by means of \eqref{eq:large_enough_Tn},
\eqref{eq:An_upper_bound}, and \eqref{eq:1mvtaueps_uniform_estimate}:
\begin{align*}
E_{\varepsilon^{n}}(w^{n})
- E_{\varepsilon^{n}}(\tilde{w}^{n})
&= \frac{\varepsilon^n}{2}\|\nabla w^n\|_{L^2(\tilde{S}^n)}^2 + \frac{1}{\varepsilon^n}\int_{\tilde{S}^n} w^n(1 - w^n)\dd x - A^n\\
&\ge \frac{\varepsilon^n}{2}\|\nabla w\|_{L^2(T^n)}^2 + \frac{1}{\varepsilon^n}\int_{T^n} w^n(1 - w^n) \dd x
- B^n\\
&\ge \frac{1}{2(R + 1)} E_{\varepsilon^n}(w^n),
\end{align*}
which proves the claim \eqref{eq:claim2}.

We continue to prove the overall assertion. To this end, let $\bar{w}^n$ solve
$\text{{\ref{eq:tr_gle}}}(w^n, \nabla_R j(w^n), \Delta^0)$. As in the proof of
\cref{prp:boundedness_simple}, we obtain
\begin{align*}
\gamma E_{\varepsilon^n}(w^n) -
\gamma E_{\varepsilon^n}(\bar{w}^n)
+ C_{g} 
&\ge \pred(w^n,\Delta^0,\varepsilon^n)
\text{ and} \\
\ared(w^n,\bar{w}^n,\varepsilon^n) &\ge 
\gamma E_{\varepsilon^n}(w^n) - \gamma E_{\varepsilon^n}(\bar{w}^n) - 2 C_j. 
\end{align*}
Moreover, the optimality of $\bar{w}^n$ and the feasibility
of $\tilde{w}^n$ for $\text{{\ref{eq:tr_gle}}}(w^n, \nabla_R j(w^n), \Delta^0)$ give
\[ d^n \coloneqq 
\gamma E_{\varepsilon^n}(w^n) 
- \gamma E_{\varepsilon^n}(\bar{w}^n)
\ge
\gamma E_{\varepsilon^n}(w^n) 
- \gamma E_{\varepsilon^n}(\tilde{w}^n)
- 2 C_g
\ge
\frac{1}{2(R + 1)} \gamma E_{\varepsilon^n}(w^n)
- 2 C_g. \]
We obtain that
\[
\frac{\ared(w^n,\bar{w}^n,\varepsilon^n)}{\pred(w^n,\Delta^0,\varepsilon^n)}
\ge 
\frac{d^n - 2 C_j}{d^n + C_g}
\ge \rho
\]
if $d^n \ge (1 - \rho)^{-1}(\rho C_g + 2 C_j)$, which holds
in particular if
$\gamma E_{\varepsilon^n}(w^n) \ge 2(R + 1)((1 - \rho)^{-1}(\rho C_g + 2 C_j) + 2 C_g) \eqqcolon C_E$. Moreover, if $n$ is large enough, this also implies
that 
\[
\ared(w^n, \bar{w}^n, \varepsilon^n) \ge 
\frac{1}{2(R + 1)}\gamma E_{\varepsilon^n}(w^n) - 2 C_g - 2 C_j 
\ge \frac{\rho}{1 - \rho}C_g > 
\kappa^n \underline{\Delta}^n. \]

Now, we apply the same argument as in 
\cref{prp:boundedness_simple}.
Thus there exists $n_0 \in \N$ such that whenever
for $n \ge n_0$ the situation
$\gamma E_{\varepsilon^n}(w^n) \ge C_E$ 
occurs
in the first iteration after $\varepsilon^n$
was decreased and the trust-region radius reset to
$\Delta^0$ for $n \ge n_0$,
the next step is accepted and $\varepsilon^n$ will not
be reduced further before $\gamma E_{\varepsilon^n}(w^n) < C_E$ holds.
Consequently, there exists a subsequence, indexed by $k$, which
abides the bound $\gamma E_{\varepsilon^{n_k}}(w^{n_k}) < C_E$ for all $k \in \N$,
which implies the claim.
\end{proof}

\begin{remark}
The constructions in
\cref{prp:boundedness_simple,prp:boundedness_involved}
together with the fact that the objective decreases
monotonically during subsequent iterations during which
$\varepsilon^n$ is not altered implies that the boundedness
assumption of the iterates $\varepsilon^{n(k)}$ in
\cref{ass:algorithm} are satisfied under the
respective prerequisites of
\cref{prp:boundedness_simple,prp:boundedness_involved}.
\end{remark}

\subsection{Asymptotic L-stationarity}
\label{sec:asymptotic_stationarity}
We are ready to prove the asymptotics of \cref{alg:trm}
under \cref{ass:general_var,ass:algorithm}.
We first prove that the approximating iterates
of limit points also have subsequences so that
the values of the corresponding regularizers converge
to the value of the regularizer in the limit,
that is a multiple of the perimeter of the level set
to the value $1$ in $\Omega$. Second we prove an auxiliary
result that asserts that we can find a function that satisfies the sufficient decrease condition \eqref{eq:suffdec}
using the actual reduction $\ared$ of the unregularized objective values and the predicted reduction $\pred$ given
by the unregularized trust-region subproblem \eqref{eq:tr_sharp}.
Finally, we prove the L-stationarity of the limit points.

\begin{lemma}\label{lem:asymptotics_of_regularizer}
Let \cref{ass:general_var,ass:algorithm} hold. Let
$\bar{w}$ be a limit point in the sense of 
\cref{prp:compactness} of the subsequence $(w^{n(k)})_k$
with $n(k)$ defined as in \cref{ass:algorithm} for all $k \in \N$.
Let $\bar{w}^{-1}(\{1\})$ admit an extension $A$ outside
of $\Omega$, that is to $\R^d$, such that $\partial^* A$ is a compact, smooth hypersurface
and satisfies $\Ha^{d-1}(\partial^* A \cap \partial^*\Omega) = 0$.
Then $\liminf_{k} E_{\varepsilon^{n(k)}}(w^{n(k)}) = C_0 P_\Omega(\bar{w}^{-1}(\{1\}))$.
\end{lemma}
\begin{proof}
	\Cref{prp:compactness} implies that there is a subsequence of $(w^{n(k)})_k$ that converges
in $L^1(\Omega)$ and because of the $L^\infty(\Omega)$-bounds also in $L^2(\Omega)$.
In order to prove the claim, we restrict to this subsequence,
denote its index by $n(k)$ as well in the interest of a cleaner presentation, and prove the claim for
this subsequence. This is sufficient because \cref{prp:compactness} already gives
$\liminf_{k} E_{\varepsilon^{n(k)}}(w^{n(k)}) \ge C_0 P_\Omega(\bar{w}^{-1}(\{1\}))$.

Because $n(k)$ is the last iteration such that $\varepsilon^{n(k)} = \varepsilon^0 r^{-k}$,
it is clear that $\cvxflag = 0$ and the minimizers $\tilde{w}^{n(k)-j}$ of
\ref{eq:tr_gle}$(w^{n(k)}, \nabla_R j(w^{n(k)}), \Delta^0 2^{-j})$ were rejected in
\cref{alg:trm} ln.\ \ref{ln:suff_dec} for all $j \in \{0,\ldots,k\}$, which implies that
\begin{align}
\ared(w^{n(k)},\tilde{w}^{n(k) - j},\varepsilon^{n(k)})
&\le \kappa^0 \underline{\Delta}^0 2^{-2k} \label{eq:contr_bb} \\
\text{or } \ared(w^{n(k)},\tilde{w}^{n(k) - j},\Delta^0 2^{-j})
&< \rho \pred(w^{n(k)},\Delta^0 2^{-j},\varepsilon^{n(k)}) \label{eq:contra_sd}
\end{align}
We proceed with a proof by contradiction and assume that 
$\liminf_{k} \gamma E_{\varepsilon^{n(k)}}(w^{n(k)}) - P_\Omega(\bar{w}^{-1}(\{1\})) > \delta$ 
holds for some $\delta > 0$. We show that neither of the inequalities above
can hold for sufficiently large values of $k$ so that a step must have been accepted and thus a subsequence
with these properties cannot exist. 
We start by showing that \eqref{eq:contra_sd} cannot hold for large enough $j$ and $k$ (depending on $j$)
and then infer that \eqref{eq:contr_bb} cannot hold as well.

For all $\eta > 0$ there exists $j_0 \in \N$ such that for all $k > j_0$,
all $j \in \{j_0,\ldots,k\}$, and all $w$
that are feasible for \ref{eq:tr_gle}$(w^{n(k)}, \nabla_R j(w^{n(k)}), \Delta^0 2^{-j})$, we obtain the estimates
\begin{align*}
\ared(w^{n(k)},w,\varepsilon^{n(k)}) &= j(w^{n(k)}) - j(w)
+ \gamma E_{\varepsilon^{n(k)}}(w^{n(k)})
- \gamma E_{\varepsilon^{n(k)}}(w)\\
&\ge \gamma E_{\varepsilon^{n(k)}}(w^{n(k)}) - \gamma E_{\varepsilon^{n(k)}}(w) - \eta
\end{align*}
and
\begin{align*}
\left|\pred(w^{n(k)},\Delta^0 2^{-j},\varepsilon^{n(k)})
- \gamma E_{\varepsilon^{n(k)}}(w^{n(k)})
+ \gamma E_{\varepsilon^{n(k)}}(w)\right|
&= |(\Rgrad j(w^{n(k)}), w^{n(k)} - w)_{L^2}|
\le \eta,
\end{align*}
which give
\[ 
\frac{\ared(w^{n(k)},\tilde{w}^{n(k)},\varepsilon^{n(k)})}{\pred(w^{n(k)},\Delta^0 2^{-j_0},\varepsilon^{n(k)})}
\ge 
\frac{\pred(w^{n(k)},\Delta^0 2^{-j_0},\varepsilon^{n(k)}) - 2 \eta}{\pred(w^{n(k)},\Delta^0 2^{-j_0},\varepsilon^{n(k)})}.
\]
In order to continue the estimate, we consider an appropriate feasible approximation of $\bar{w}$ and distinguish two
cases, specifically whether $\partial^* \bar{w}^{-1}(\{1\}) \cap \Omega = \emptyset$ or not.

If $\partial^* \bar{w}^{-1}(\{1\})  \cap \Omega = \emptyset$, then $\bar{w}(x) = 0$ a.e.\ or $\bar{w}(x) = 1$ a.e.,
which gives that $\bar{w}$ is feasible for the subproblems
\ref{eq:tr_gle}$(w^{n(k)}, \nabla_R j(w^{n(k)}), \Delta^0 2^{-j_0})$
for all $k$ large enough. We choose $\eta = \delta\cdot\frac{1 - \rho}{3}$, in particular
$3\eta < \delta$, and obtain for all $k > j_0$ that 
\[ \frac{\ared(w^{n(k)},\tilde{w}^{n(k)},\varepsilon^{n(k)})}{\pred(w^{n(k)},\Delta^0 2^{-j_0},\varepsilon^{n(k)})} \ge \frac{\delta - 3\eta}{\delta} = \rho
\]
because $\bar{w}$ is feasible, implying $\pred(w^{n(k)},\Delta^0 2^{-j_0},\varepsilon^{n(k)})
\ge \delta - \eta$, and the function $x \mapsto \frac{x - 2 \eta}{x}$ is monotonically increasing for
$x \ge 2 \eta$. 
This implies that the step is accepted latemost for the trust-region radius $\Delta^0 2^{-j_0}$
(depending on $\eta$ but not on $k$ for $k > j_0$),
implying that \eqref{eq:contra_sd} cannot hold for $k > j_0$.
For the accepted step $\tilde{w}$ for the trust-region radius  $\Delta^0 2^{-\tilde{j}}$
with $\tilde{j} \le j_0$, we know that
\[ \ared(w^{n(k)},\tilde{w},\varepsilon^{n(k)})
\ge \rho \pred(w^{n(k)},\Delta^0 2^{-\tilde{j}},\varepsilon^{n(k)})
\ge \rho\left(\delta - \eta\right) 
\ge \frac{2\rho}{3}\delta, \]
where the second inequality follows from the optimality of $\tilde{w}$ for the
trust-region subproblem.
Consequently, \eqref{eq:contr_bb} cannot hold for large enough $k$.
This implies that the sequence $(n(k))_k$ cannot exist under the assumption that
$\liminf_{k} \gamma E_{\varepsilon^{n(k)}}(w^{n(k)})
- P_\Omega(\bar{w}^{-1}(\{1\})) > \delta$ for some $\delta > 0$.

If $\partial^* \bar{w}^{-1}(\{1\}) \cap \Omega \neq \emptyset$, then $\bar{w} \notin H^1(\Omega)$ so that
$E_{\varepsilon^{n(k)}}(\bar{w}) = \infty$ and we need an additional approximation
in order to obtain a feasible approximation. The properties of the assumed extension $A$
of $\bar{w}^{-1}(\{1\})$ imply that we can apply the reverse approximation argument from 
Proposition 2 in \cite{modica1987gradient}, which gives that there are functions 
$w_{\varepsilon^{n(k)}}$, $k \in \N$,
such that $E_{\varepsilon^{n(k)}}(\bar{w}_{\varepsilon^{n(k)}}) \to C_0P_\Omega(\bar{w}^{-1}(\{1\}))$.
As part of its proof (see p.\ 135 in \cite{modica1987gradient}), it is shown that the functions satisfy
\begin{gather}\label{eq:L1_estimate_for_smooth_sets}
\|\bar{w}_{\varepsilon^{n(k)}} - \bar{w}\|_{L^1(\Omega)}
\le (\varepsilon^{n(k)})^{\frac{1}{2}} C(\bar{w}),
\end{gather}   
where $C(\bar{w})$ is a constant that depends on $\bar{w}$. Because $r > 4$
we have that $\varepsilon^{n(k)} \in o((\underline{\Delta}^{n(k)})^2)$
by means of the update rules for $\varepsilon^n$ and $\underline{\Delta}^{n}$.
Consequently, by choosing $\eta = \delta\cdot\frac{1 - \rho}{3}$,
we obtain again a corresponding $j_0$ and for all $k$ large enough we obtain
that $\bar{w}_{\varepsilon^{n(k)}}$ is feasible 
for the trust-region subproblem \ref{eq:tr_gle}$(w^{n(k)}, \nabla_R j(w^{n(k)}), \Delta^0 2^{-j_0})$
(for example by choosing $j_0 - 1$ for the $j_0$ that could be chosen for $\bar{w}$).
Now, we obtain
\begin{gather*}
\begin{multlined} E_{\varepsilon^{n(k)}}(w^{n(k)}) 
- E_{\varepsilon^{n(k)}}(\bar{w}_{\varepsilon^{n(k)}})\\
\ge E_{\varepsilon^{n(k)}}(w^{n(k)}) - C_0P_\Omega (w^{-1}(\{1\}))
- |E_{\varepsilon^{n(k)}}(\bar{w}_{\varepsilon^{n(k)}}) - C_0P_\Omega (w^{-1}(\{1\}))|
\ge \frac{\delta}{2}
\end{multlined}
\end{gather*}
for all $k$ large enough. Then, using $\frac{\delta}{2}$ instead of $\delta$ for the adjustment
of $\eta$, the same reasoning as for $\partial^* \bar{w}^{-1}(\{1\}) \cap \Omega = \emptyset$
applies in order to refute the existence of the claimed sequence $(n(k))_k$.
\end{proof}

\begin{remark}
If the boundary regularity assumption of the extension in \cref{lem:asymptotics_of_regularizer} 
is dropped, a further approximation of $\bar{w}^{-1}(\{1\})$
by sets with smooth boundary is necessary
in order to apply Proposition 2 from \cite{modica1987gradient}.
Unfortunately, the constant in the estimate \eqref{eq:L1_estimate_for_smooth_sets} then depends 
on these approximants and might increase for every approximant, thereby
preventing us from controlling the convergence rate of the left hand side
of \eqref{eq:L1_estimate_for_smooth_sets}.
Consequently, we have not been able to prove the claim in this case.
\end{remark}

\begin{remark}
In \cref{alg:trm} \ref{ln:reduce_epsilon}, $\underline{\Delta}^n$ is divided by two while
$\varepsilon^n$ is divided by $r > 4$ when $\Delta^{n+1}$ falls below $\underline{\Delta}^n$.
This sufficient to guarantee $\varepsilon^{n(k)} \in o((\underline{\Delta}^{n(k)})^2)$
in the proof of \cref{lem:asymptotics_of_regularizer}. While we have used these update rules
in the interest of a clean presentation, other update rules that
guarantee $\varepsilon^{n(k)} \in o((\underline{\Delta}^{n(k)})^2)$ would be feasible as well.
\end{remark}

The following auxiliary lemma is a variant of Lemma 6.2 in \cite{manns2022on}, where 
we assume the relaxed condition \cref{ass:general_var} instead of twice continuous 
differentiability.

\begin{lemma}\label{lem:step_accept}
	Let \cref{ass:general_var} hold.
	Let $\rho \in (0,1)$. Let $w \in L^1(\Omega)$ satisfy
	$w(x) \in \{0,1\}$ for a.e.\ $x \in \Omega$ and $P_\Omega(w^{-1}(\{1\})) <\infty$.
	Let $\nabla_R j(w) \in C(\bar{\Omega})$. Let $\Delta^k \searrow 0$.
	Let $\tilde{w}^k$ minimize $\text{\emph{\ref{eq:tr_sharp}}}(w,\nabla_R j(w),\Delta^{k})$ for $k \in \N$.
	Then at least one of the following statements holds true.
	\begin{enumerate}
		\item\label{itm:step_accept_w_stationary} The function $w$ satisfies the L-stationarity
		conditions \eqref{eq:p_variational_principle}.
		\item\label{itm:step_accept_w_optimal} There exists $k_0 \in \N$ such that
		the objective value of $\text{\emph{\ref{eq:tr_sharp}}}(w,\Rgrad j(w), \Delta^{k_0})$
		at $w^{k_0}$ is zero.
		\item\label{itm:step_accept_w_sufficient_decrease} There exists $k_0 \in \N$ such that 
		\begin{multline}\label{eq:suffdec_sharp}
		\delta^{k_0} \coloneqq j(w) + \gamma C_0P(w^{-1}(\{1\})) - j(\tilde{w}^{k_0}) - \gamma C_0P((\tilde{w}^{k_0})^{-1}(\{1\})) \\
		\ge \rho (\Rgrad j(w),w - \tilde{w}^{k_0}) + \gamma C_0P(w^{-1}(\{1\})) - \gamma C_0P((\tilde{w}^{k_0})^{-1}(\{1\})) \eqqcolon p^{k_0}.
		\end{multline}
	\end{enumerate}
\end{lemma}
\begin{proof}
	We closely follow the lines of the proof of Lemma 6.2
	in \cite{manns2022on}.
	Let $k \in \N$. The objective of $\text{{\ref{eq:tr_sharp}}}(w,\Rgrad j(w), \Delta^{k})$, the negative of the 
	right hand side in \eqref{eq:suffdec_sharp}, is zero when inserting $w$ itself. Consequently, if the optimal solution $\tilde{w}^k$ has objective
	value zero (Outcome \ref{itm:step_accept_w_optimal}), then $w$ minimizes $\text{{\ref{eq:tr_sharp}}}(w, \Rgrad j(w), \Delta^{k})$.
	\Cref{prp:tr_sharp_stationarity} implies that $w$ is L-stationary (Outcome \ref{itm:step_accept_w_stationary}).
	Thus it remains to show that Outcome \ref{itm:step_accept_w_sufficient_decrease} when
	Outcome \ref{itm:step_accept_w_stationary} does not, that is we show \eqref{eq:suffdec_sharp} for some $k_0$.

	Because $w$ is not L-stationary, there exists $\phi \in C_c^\infty(\Omega;\R^d)$ with
	$\supp \phi \neq \emptyset$ and $\eta > 0$ such that
	\[
	\int_{\partial^*E \cap \Omega}
	-\Rgrad j(w)(\phi\cdot n_{E}) \dd \Ha^{d-1} - 
	\gamma C_0 \int_{\partial^* E \cap \Omega}\bdvg{E} \phi\dd \Ha^{d-1}	
	> \eta. 
	\]
	Let $(f_t)_{t\in (-\tau_0,\tau_0)}$ be defined by $f_t \coloneqq I + t \phi$ for $\phi \in C_c^\infty(\Omega, \R^d)$ and $t \in (-\tau_0,\tau_0)$
	for some $\tau_0 > 0$, see also Definition 3.1 and Proposition 3.2 in
	\cite{manns2022on}.
	We define $f_{t}^{\#}w \coloneqq \chi_{f_t(w^{-1}(\{1\}))}$.
	Then as in the proof of Lemma 6.2 in \cite{manns2022on}, we obtain for all $k \in \N$ sufficiently large that
	there is $\tau^k > 0$ such that $f_t^{\#}w$ is feasible for	$\text{{\ref{eq:tr_sharp}}}(w,\Rgrad j(w),\Delta^k)$
	for all $t \in [0,\tau^k]$ and
	$\|f^{\#}_{\tau^k}w - w\|_{L^1} = \Delta^k$.
		
	\Cref{ass:general_var} yields
	\[ \delta^k
	= p^k + o(\|\tilde{w}^k - w\|_{L^1}).
	\]
	Consequently, we obtain for all $k$ large enough and
	all $t \in (0,\tau^k]$:
	\begin{align*}
	\delta^k
	&\ge \rho p^k -(1 - \rho)\left(
	(\Rgrad j(w), f_{\tau^k}^{\#}w - w)_{L^2} + \alpha \TV(f_{\tau^k}^{\#}w) - \alpha \TV(w)\right)
	+ o(\|\tilde{w}^k - w\|_{L^1})\\
	&\ge \rho p^k + (1 - \rho)\left({\tau^k} \eta + o(\tau^k)\right) 
	+ o(\|\tilde{w}^k - w\|_{L^1})\\
	&\ge \rho p^k + (1 - \rho)\left(\tau^k \eta + o(\tau^k)\right) + o(\tau^k).
	\end{align*}
	Here, the first inequality holds because $\tilde{w}^k$
	minimizes $\text{{\ref{eq:tr_sharp}}}(w,\Rgrad j(w), \Delta^{k})$,
	the second inequality holds by virtue of Lemmas 3.3 and 3.5
	in \cite{manns2022on}, and the third holds because
	\[  \|\tilde{w}^k - w\|_{L^1} \le \Delta^k
	= \|f_{\tau^k}^{\#}w - w\|_{L^1}
	\le \kappa \tau^k
	\]
	where the existence of $\kappa > 0$, which is independent of $k$,
	is asserted by Lemma 3.8 in \cite{manns2022on}. The term
	$(1 - \rho) \tau^k \eta$, which is positive,
	is larger than the $o(\tau^k)$ terms for all $k$
	large enough so that \eqref{eq:suffdec_sharp} holds for some
	$k_0 \in \N$.
\end{proof}

We are ready to prove the asymptotics of our algorithm.

\begin{theorem}\label{thm:gamma_to_stat}
Let \cref{ass:general_var,ass:algorithm} hold. Let
$\bar{w}$ be a limit point in the sense of 
\cref{prp:compactness} of the subsequence $(w^{n(k)})_k$
with $n(k)$ defined as in \cref{ass:algorithm}.
Let $\Rgrad j(\bar{w}) \in C(\bar{\Omega})$ and
$\bar{w}^{-1}(\{1\})$ admit an extension $A$ outside
of $\Omega$, that is to $\R^d$, such that $\partial^* A$ is a compact, smooth hypersurface
and satisfies $\Ha^{d-1}(\partial^* A \cap \partial^*\Omega) = 0$.
Then $\bar{w}$ is L-stationary.
\end{theorem}
\begin{proof}
We consider a subsequence of the iterates $(n(k))_k$, which we denote by the same symbol for ease of notation,
as is asserted by \cref{ass:algorithm} such that
\begin{gather}\label{eq:strict_in_alg}
E_{\varepsilon^{n(k)}}(w^{n(k)}) \to C_0 P_\Omega(\bar{w}^{-1}(\{1\})).
\end{gather}
By virtue of the update rules for $\underline{\Delta}^n$ in \cref{alg:trm} and the well-definedness
of the sequence $(n(k))_k$ it follows that for all $j \in \N$ there is $k_j$ so that for all $k \ge k_j$
it holds that $\Delta^0 2^{-j} > \underline{\Delta}^{n(k)}$ and the solution  of
\ref{eq:tr_gle}$(w^{n(k)}, \nabla_R j(w^{n(k)}), \Delta^0 2^{-j})$ is not accepted
in \cref{alg:trm} ln.\ \ref{ln:suff_dec}.

For any such $j$, we employ \cref{thm:tr_gamma_convergence} to deduce that the minimizers $w^{n(k),*}$
returned by the subproblem solves of \ref{eq:tr_gle}$(w^{n(k)}, \Rgrad j(w^{n(k)}), \Delta^0 2^{-j})$ 
in \cref{alg:trm} ln.\ \ref{ln:tr_step_gle} converge
(after possibly restricting to subsequences and after invoking the compactness
result from \cite{fonseca1989gradient}, see \cref{prp:compactness}) to minimizers $w^*$
of \ref{eq:tr_sharp}$(\bar{w}, \Rgrad j(\bar{w}), \Delta^0 2^{-j})$ in $L^1(\Omega)$
for $k \to \infty$:
\[ T(w^*) \le
\liminf_{k\to\infty} T^{n(k)}(w^{n(k),*}) \le
\limsup_{k\to\infty} T^{n(k)}(w^{n(k),r})
\to T(w^*),\]
where the $w^{n(k),r}$ denote the elements of the reverse approximating sequence that exist by 
virtue of the upper bound inequality of the $\Gamma$-convergence result in \cref{thm:tr_gamma_convergence}.

Consequently, the predicted reductions for the returned minimizers of
\ref{eq:tr_gle}$(w^{n(k)}, \nabla_R j(w^{n(k)}), \Delta^0 2^{-j})$
converge to the predicted reduction of the limit problem 
\ref{eq:tr_sharp}$(\bar{w}, \nabla_R j(\bar{w}), \Delta^0 2^{-j})$,
which in combination with the continuity properties of $\nabla j$
yields
\[ \gamma E_{\varepsilon^{n(k)}}(w^{n(k),*}) \to C_0P_\Omega((w^*)^{-1}(\{1\})). \]
In turn, we imply convergence of the induced actual reductions as well.

This implies that if there is $s \in (\rho,1)$ such that 
\begin{gather}\label{eq:suff_dec_unreg}
\frac{j(\bar{w}) + \gamma C_0P(\bar{w}^{-1}(\{1\})) - j(w^*) - \gamma C_0P((w^*)^{-1}(\{1\}))}{(\Rgrad j(\bar{w}),\bar{w} - w^*) + \gamma C_0P(\bar{w}^{-1}(\{1\})) - \gamma C_0P((w^*)^{-1}(\{1\}))}
\ge s,
\end{gather}
then there is $\bar{k}_j \ge k_j$ such that for all $k \ge \bar{k}_j$ it follows that
\[ \frac{\ared(w^{n(k)}, w^{n(k),*}, \varepsilon^{n(k)})}{\pred(w^{n(k)}, \Delta^0 2^{-j}, \varepsilon^{n(k)})} \ge \rho, \]
which would contradict the definition of $n(k)$ and in turn $w^{n(k)}$.

If $\bar{w}$ is not L-stationary, then \cref{lem:step_accept} \eqref{itm:step_accept_w_sufficient_decrease}
implies that \eqref{eq:suff_dec_unreg} holds for some large enough $j$, which leads to said
contradiction and thus we conclude that $\bar{w}$ is
L-stationary.
\end{proof}
\begin{remark}
We do not see how \cref{alg:trm} could possibly ensure that its limit points $\bar{w}$
satisfy $\Rgrad j(\bar{w}) \in C(\bar{\Omega})$ or the regularity of the extension of $\bar{w}^{-1}$ by choice of its parameters,
which we thus keep as true regularity assumptions. 
While $\Rgrad j(\bar{w}) \in C(\bar{\Omega})$ can be verifiable a priori by means of the regularity properties of the
underlying PDE, we note that we do not know at the moment if and how the assumption in \cref{thm:gamma_to_stat} that the extension $A$ of $\bar{w}^{-1}$ has
a reduced boundary $\partial^* A$ that is a compact, smooth hypersurface and satisfies $\Ha^{d-1}(\partial^* A \cap \partial^*\Omega) = 0$
can be checked in practice. Therefore, it has a similar role as so-called constraint qualifications in nonlinear programming,
where some regularity on the constraints (in our case on the integrality restriction) needs to be satisfied in order to obtain stationarity
or convergence to stationary points.

We believe that the assumption may turn out to be realistic in many settings, as it is common in optimization and in particular optimal control that the optimality
conditions allow to improve regularity of stationary points.
\end{remark}

\section{Applications of the algorithmic framework}\label{sec:applications}

We discuss two classes of problems. In \cref{sec:application_elliptic}, we consider
the case of elliptic equations and briefly argue that non-standard regularity theory
allows to certify \cref{ass:general_var,ass:algorithm}.
In \cref{sec:application_elliptic_control}
we provide a numerical setup and computational results for the elliptic equation.
In \cref{sec:application_wave}, we consider
a  more challenging class of coefficient-control problems for acoustic wave equations. Here, we point out
the gaps between our regularity assumptions.
In \cref{sec:numerical_experiments}, we give a numerical setup for the acoustic
wave equations and provide some computational results in order to illustrate
the algorithm albeit outside of the scope of its assumptions.

\subsection{Source control of elliptic equations}\label{sec:application_elliptic}

We first discuss the developed optimization framework for the control of a
source term in an elliptic equation. Let $\Omega = (0,2)^2$ and $\nu > 0$. Consider
\begin{gather}\label{eq:elliptic} 
\left \{
\begin{aligned}
-\nu \Delta u + u =&\, B w + f\quad &&\text{in } \Omega, \\
u =&\, 0\quad &&\text{on }\partial \Omega
\end{aligned} \right.
\end{gather}
for some bounded linear operator $B : L^1(\Omega) \to L^1(\Omega)$, $f \in L^1(\Omega)$,
and $w \in L^1(\Omega)$ is the given control. We further assume that the adjoint
$B^* : L^\infty(\Omega) \to L^\infty(\Omega)$
is also bounded as a mapping $B^* : C(\bar{\Omega}) \to C(\bar{\Omega})$.

By means of
\cite[Theorems 1\,\&\,3]{groger1989aw} and \cite[Theorem 4.6]{simader1972dirichlet},
there exists $q>2$ such that \[\|u\|_{W^{1,q}(\Omega)} \le c_1 \|w\|_{W^{-1,q}(\Omega)} + c_2\|f\|_{W^{-1,q}(\Omega)}\]
for some constants $c_1$, $c_2 > 0$. Analyzing the adjoint of the solution operator to \eqref{eq:elliptic}, we obtain
\begin{gather}\label{eq:elliptic_w1qstar_estimate}
\|u\|_{W^{1,q^*}(\Omega)} \le c_3 \|w\|_{W^{-1,q^*}(\Omega)} + c_4\|f\|_{W^{-1,q^*}(\Omega)}
\end{gather}
with $q^*$ being the H\"older conjugate of $q$ and $c_3$, $c_4 > 0$.
Because of the dense embeddings
$W^{1,q}(\Omega) \hookrightarrow C(\bar{\Omega})$ for $q > 2$, we obtain
$L^1(\Omega) \hookrightarrow \calM(\Omega) \hookrightarrow W^{-1,q^*}(\Omega)$, where
$\calM(\Omega)$ denotes the space of Radon measures by means of the Riesz representation theorem.
Combining \eqref{eq:elliptic_w1qstar_estimate} and the embeddings stated above with $W^{1,q^*}(\Omega) \hookrightarrow L^2(\Omega)$,
there is a continuous solution operator $S : L^1(\Omega) \to L^2(\Omega)$ to \eqref{eq:elliptic}.

Let $j$ in \eqref{eq:p_abstract} be $j = J \circ S$ and $J : L^2(\Omega) \to \R$ be continuously differentiable.
The corresponding boundary-value problem for the adjoint equation reads
\begin{gather}\label{eq:elliptic_adjoint} 
\left \{
\begin{aligned}
-\nu \Delta p + p &= -J'(S(w))\quad &&\text{in }\Omega,\\
p &= 0\quad &&\text{on }\partial \Omega
\end{aligned} \right.
\end{gather}
and inherits the regularity properties of \eqref{eq:elliptic} and we obtain
$\Rgrad j(w) = B^* p \in C(\bar{\Omega})$. Let $B_{L^\infty(\Omega)}$ be the unit ball
of $L^\infty(\Omega)$. Then $B_{L^\infty(\Omega)}$ with respect to all $L^s$-norms, $s \ge 1$,
and in turn $S(B_{L^\infty(\Omega)})$ is bounded in $W^{1,q}(\Omega)$ for the $q > 2$
from above. Consequently, $S(B_{L^\infty(\Omega)})$ is also bounded in $H^1(\Omega)$ so that
$S(B_{L^\infty(\Omega)})$ is a compact subset of $L^2(\Omega)$. Thus because $J : L^2(\Omega) \to \R$ is
continuous, $J \circ S$ is bounded on $B_{L^\infty(\Omega)}$. Because $J: L^2(\Omega)\to \R$ is even
continuously differentiable, we obtain that the right hand side of \eqref{eq:elliptic_adjoint}
is bounded in $L^2(\Omega)$ and in turn $p$ is bounded in $W^{1,q}(\Omega)$ and $C(\bar{\Omega})$
too. Together with the assumptions on $B^*$, we obtain that $\Rgrad j$ is bounded
in $L^\infty(\Omega)$ and $C(\bar{\Omega})$.

Combining these considerations, we obtain the following proposition, which proves
that \cref{ass:general_var} is satisfied
as well as the prerequisites
of \cref{prp:boundedness_simple} and \cref{prp:boundedness_involved}, implying that \cref{ass:algorithm}
is satisfied for the execution of \cref{alg:trm}.
\begin{proposition}
Let $S : L^1(\Omega) \to L^2(\Omega)$ be the
solution operator of \eqref{eq:elliptic} and $J : L^2(\Omega) \to \R$ be continuously differentiable.
Let $j$ in \eqref{eq:p_abstract} be $j = J \circ S$. 
Then $j = J \circ S$ is continuously differentiable as a mapping 
$j: L^1(\Omega) \to \R$ and $j$ and $\Rgrad j$ are bounded on the 
unit ball of $L^\infty(\Omega)$.
Moreover, $\Rgrad j(w) \in C(\bar{\Omega})$
for all $w \in B_{L^\infty(\Omega)}$.
\end{proposition}

Consequently, we satisfy all assumptions a priori except the constraint qualification-type assumption that any
limit point $\bar{w}$ produced by the algorithm satisfies that $\bar{w}^{-1}(\{1\})$ admits an extension $A$
outside of $\Omega$, that is to $\R^d$, such that $\partial^* A$ is a compact, smooth hypersurface
and $\Ha^{d-1}(\partial^* A \cap \partial^*\Omega) = 0$.

\subsection{Numerical experiments for the source control problem}\label{sec:application_elliptic_control}
We provide an example in \cref{sec:source_exact_solution}, where a 
global minimizer of the problem \eqref{eq:p} is known exactly.
We briefly describe the discretization of the state and adjoint
PDE in \cref{sec:source_discretizations}. Our algorithmic setup
and the experiments are described in \cref{sec:source_setup}.
We describe the results in \cref{sec:source_results}.

\subsubsection{Exact solution}\label{sec:source_exact_solution}
We construct a globally optimal solution $(u^*, w^*)$ to the relaxed problem
\begin{gather}\label{eq:source_relaxed}
\begin{aligned}
\min_{u,w}\ & \frac{1}{2}\|u - u_d\|_{\Ltwo}^2 + \gamma C_0 P_{\Omega}(w^{-1}(\{1\})) \\
\st\ &-\nu \Delta u + u = w + f \text{ in } \Omega,\\
& u = 0 \text{ on } \partial\Omega, \\
&w(x) \in \R \text{ for a.e.\ } x \in \Omega
\end{aligned}
\end{gather}
such that $w^*$ is $\{0,1\}$-valued and $\partial^* {w^*}^{-1}(\{1\})$ is a compact, smooth hypersurface
that is completely contained in $\Omega$. This implies that $u^*$ is also a solution to \eqref{eq:p}
with the choices $J(u) \coloneqq \frac{1}{2}\|u - u_d\|_{\Ltwo}^2$ and $S(w) \coloneqq (-\nu \Delta + I)^{-1}(w + f)$
and satisfies the additional constraint qualification-type assumption of \cref{thm:gamma_to_stat}.
Regarding the constants, we choose $\nu = 10^{-2}$ and $\gamma = 10^{-2}C_0^{-1}$ and 
note that $C_0 = \int_0^1 \sqrt{2 s(1 - s)}\dd s = \tfrac{\pi}{4\sqrt{2}}$
for our choice $E_\varepsilon$, which can be deduced by following the analysis in \cite[p.\,132]{modica1987gradient}
and observing that choosing $W(x) = 2 x (1 - x)$ therein corresponds to our setting (after a transformation
of $\tfrac{\varepsilon}{2}$ to $\sqrt{\varepsilon}$). We follow the ideas of \cite{schiemann2024regularization}
and construct $u_d$, $f$, $u^*$, and $p^*$ such that $w^* = \chi_{B_{0.5}((1,1)^T)}$, $u^*$, and $p^*$ 
satisfy the following optimality system of \eqref{eq:source_relaxed},
see also Theorem 3 in \cite{casas2019analysis} and Theorem
2.6 in \cite{bredies2024extremal} for the derivation of optimality conditions,
\begin{subequations}
\begin{align}
-\nu \Delta u + u &= w + f \text{ in } \Omega,\ \ u = 0 \text{ on } \partial \Omega, \label{eq:u_match} \\
-\nu \Delta p + p &= u - u_d \text{ in } \Omega,\ \ p = 0 \text{ on } \partial \Omega, \label{eq:p_match} \\
\int_\Omega p \dd x &= 0  \label{eq:grad_equality} \\
(-p, w)_{L^2} &= \gamma C_0 \TV(w), \label{eq:tv_match} \\
(-p, v)_{L^2} &\le \gamma C_0 \TV(v) \text{ for all } v \in \BV(\Omega). \label{eq:tv_inequality}
\end{align}
\end{subequations}
Following \cite{schiemann2024regularization}, we first construct a function $\psi \in C^3_c((0,1))$
that satisfies $\psi(0.5) = 1$ and $0 \le \psi(s) \le 1$ for all $s \in (0,1)$. This then allows to
define
\[ \phi(x) = \left\{
\begin{aligned}
- \psi(\|x - (1,1)^T)\|)\tfrac{x - (1,1)^T}{\|x - (1,1)^T\|} &\ \text{ if } \|x - (1,1)^T\| \in \supp \psi, \\
0 &\ \text{ else.}
\end{aligned}
\right.
\]
In order to meet all conditions, we solve a linear system for the coefficients of a degree 8 polynomial
ansatz for $\psi$. Then we choose $p^* = \gamma C_0 \dvg \phi$, which allows to verify \eqref{eq:tv_match} and
\eqref{eq:tv_inequality} for the choice $w = w^*$. Next, we choose $u^*\in C_0^2(\Omega)$ as
$u^*(x) = -2 x_1^2(2 - x_1)^2 x_2^2(2 - x_2)^2$ and in turn
$f = -\nu \Delta u^* + u^* - w^*$ to satisfy \eqref{eq:u_match}
and $u_d = u^* - (-\nu \Delta p^* + p^*)$ to satisfy \eqref{eq:p_match}.
Finally, \eqref{eq:grad_equality} holds by means of Gau\ss--Green formula and the fact that $\phi$
has compact support.

\subsubsection{Discretization of the state and adjoint problems}\label{sec:source_discretizations}
We discretize $\Omega$ into $64 \times 64$ squares, which are subdivided into two triangles each.
We choose CG1 ansatz functions for $u$, $p$, and $w$. We use the \texttt{FEniCSx} library 
\cite{ScroggsEtal2022,BasixJoss,AlnaesEtal2015} to assemble and solve the PDE \eqref{eq:u_match}
and its adjoint \eqref{eq:p_match}. In order to provide (approximations of) $u_d$ and $f$, we implement
the exact solution from \cref{sec:source_exact_solution} using \texttt{sympy} \cite{sympy} and interpolate
it into CG1 functions.

\subsubsection{Setup}\label{sec:source_setup}
Regarding the implementation of \cref{alg:trm}, we only solve 
instances of the convex subproblems \eqref{eq:tr_cvx} in order to 
keep the implementation and computational effort
manageable, where we note that solving quadratic programs with
one negative eigenvalue is already
NP-hard in general \cite{pardalos1991quadratic}. We execute our algorithm on a laptop computer
with an Intel(R) Core(TM) i7-11850H CPU (2.50\,GHz) and 64\,GB RAM. We solve the discretized
subproblems with Gurobi 10.0.0 \cite{Gurobi}.
We initialize the algorithm with the choices $\Delta^0 = 1$,
$\rho = 10^{-4}$, $\kappa^0 = 10^{-8}$, $\underline{\Delta}^0 = \num{2.44e-4}$, and $r = 5$.
We initialize the homotopy with the different values
$\varepsilon^0 \in \{25, 5, 1, \num{2e-1}, \num{4e-2}, \num{8e-3}\}$
and for each of these six values for $\varepsilon^0$ with $21$ different initial controls $w^0$.
Specifically, we choose eleven constant initial controls
$w^0(x) = i \times 10^{-1}$ for $i \in \{0,\ldots,10\}$ and
$w^0(x) = \sin(0.25 i \pi x_1)^2\sin(0.25 i \pi  x_2)^2$ for $i \in \{1,\ldots,10\}$.
This gives a total of 126 runs of \cref{alg:trm}.

Our fixed CG1 ansatz limits the gradient of $w$ and, consequently, $E_\varepsilon$
cannot be expected to be reliable for very small values of $\varepsilon^n$.
In particular, $\|\nabla w^n\|_{L^2}^2$ cannot become larger and $\int_\Omega \Psi(w^n)$ cannot
become smaller if a phase transition from $0$ to $1$ already happens inside one grid cell.
We have observed such effects when $\varepsilon^n$ was reduced below $\num{8e-3}$ and
therefore always consider the last accepted step for $\varepsilon^n = \num{8e-3}$ as the
last iterate of the homotopy but note that in all cases, the algorithm did accept at most one further
iteration for the next smaller value $\varepsilon^n = \num{1.6e-3}$ with the $L^2$-norm
of the step always being below $10^{-10}$.

\subsubsection{Results}\label{sec:source_results}
All 126 runs produced sequences of iterates so that the
nodal values in the CG1 coefficient vector were almost binary.
The largest non-binarity of the coefficients over all choices of $w^0$
and for $\varepsilon^0 \in \{25, 5, 1, \num{2e-1}, \num{4e-2} \}$
was $\num{3.37e-10}$ while the largest final non-binarity for
$\varepsilon^0 =\num{8e-2}$ was $\num{1.08e-7}$.

For the choices $\varepsilon^0 \in \{25, 5, 1, \num{2e-1}, \num{4e-2} \}$,
all	sequences of produced iterates for the corresponding 105 runs have 
approximately the same final iterate $w_f$ and corresponding state and
adjoint $u_f$ and $p_f$, which are close to the CG1 interpolations of the
exact solution $\mathcal{I}w^*$ and corresponding
$\mathcal{I}u^*$ and $\mathcal{I}p^*$. Specifically,
we obtain for all cases
\[ \|w_f - \mathcal{I} w^*\|_{\Ltwo} = \num{3.12e-2},\quad
\|u_f - \mathcal{I} u^*\|_{\Ltwo} = \num{3.68e-3},\quad
\|p_f - \mathcal{I} p^*\|_{\Ltwo} = \num{2.21e-3}
\]
after rounding to three significant digits. In line with the fact that
$\mathcal{I} w^*$ is not optimal for the discretized problems, we observe 
that the objective value $J(S(w_f)) + \gamma E_{8 \cdot 10^{-3}}(w_f)$ is always slightly
smaller than $J(S(\mathcal{I} w^*)) + \gamma E_{8 \cdot 10^{-3}}(\mathcal{I} w^*)$.
We have validated for two runs that these errors decrease
when choosing a finer grid.
To give a visual impression of this global convergence behavior, we pick the choice
$w^0(x) = \sin(2.5 \pi x_1)^2\sin(2.5 \pi  x_2)^2$ and plot the $L^2$-distance
$\|w^n - \mathcal{I}w^*\|_{L^2}$ over the course of the iterations $n$ of \cref{alg:trm}
for the choices $\varepsilon^0 \in \{25, 5, 1, \num{2e-1}, \num{4e-2}\}$ in \Cref{fig:global_convergence}.
\begin{figure}
	\centering
	\includegraphics[width=.5\linewidth]{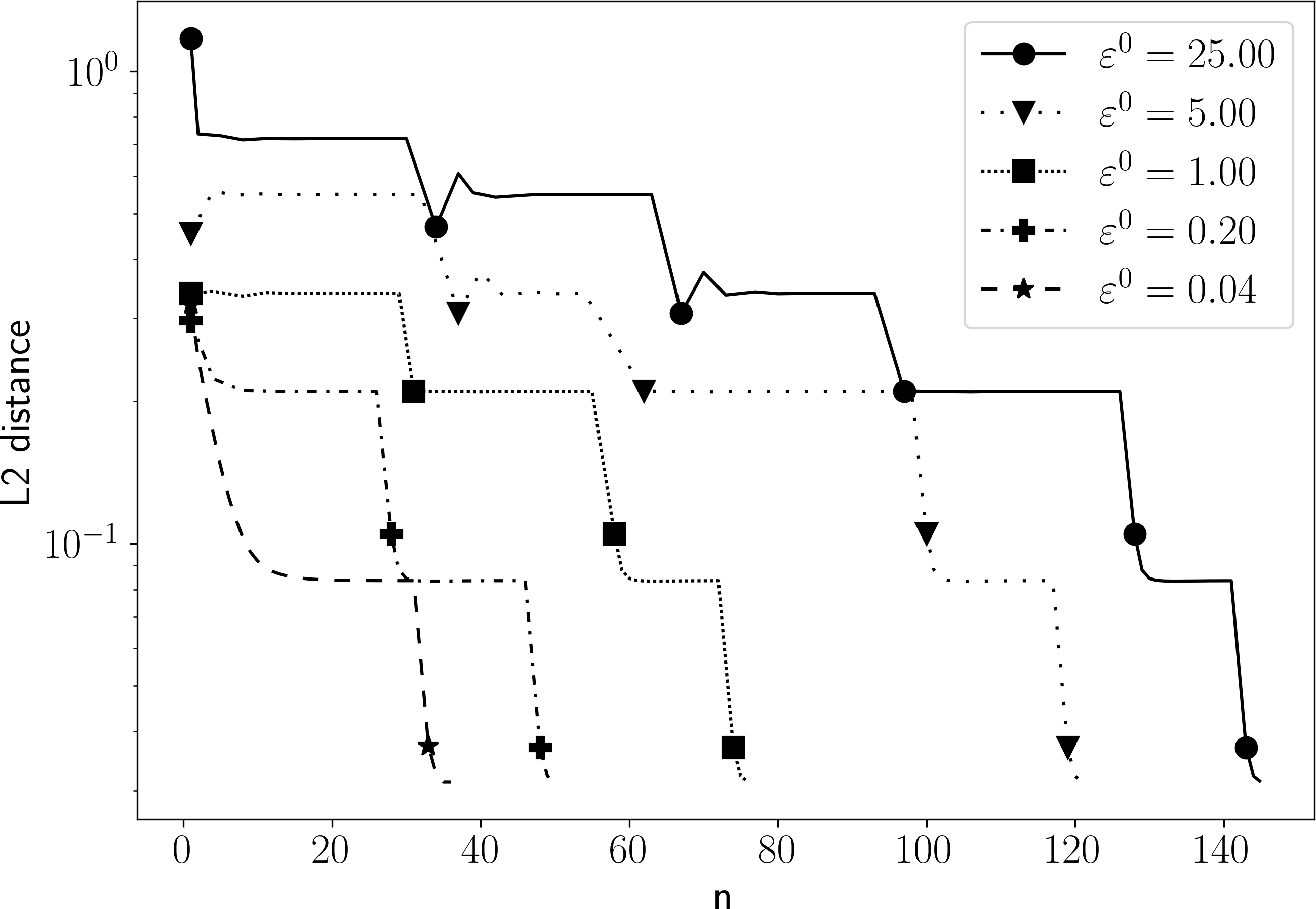}
	\caption{Global convergence of the iterates produced by \cref{alg:trm}
		indicated by the distance $\|w^n - \mathcal{I}w^*\|_{L^2(\Omega)}$ for the different
		values of $\varepsilon^0$. The locations of the marker symbols indicate the
		first iteration of an accepted iterate for a new (reduced) value of $\varepsilon^n$.}
	\label{fig:global_convergence}
\end{figure}

The 21 runs for the choice $\varepsilon^0 = \num{8.0e-3}$ produced
only very few iterations (between 1 and 8) and did not yield final
iterates close to $\mathcal{I}w^*$ but generally a binary-valued function
that was close to the initialization. While it is difficult to
come to a definite conclusion about the reasons here since discretization
and numerical precision effects may also be involved here, we believe that
this can mainly be attributed the fact that we only use the convex trust-region 
subproblems, which imply a high penalty proportionally to $\tfrac{1}{\varepsilon}$ on the change of a nodal value
that is already close to $1$ to $0$ or vice versa, see \eqref{eq:tr_cvx} 
although the difference may be very small or even zero in the non-convex 
formulation.
\begin{table}[h]
	\centering	
	\caption{Average and median number of iterations and average run time of the executions of \cref{alg:trm} over the different initializations
		$w^0$ for the different values of $\varepsilon^0$.}
	\label{tbl:avg_run_times}	
		\begin{tabular}{c|ccccc}
			\toprule
			$\varepsilon^0$ & $25$ & $5$ & $1$ & $\num{2e-1}$ & $\num{4e-2}$ \\
			\midrule
			Average iteration number
			& \num{146.76} & \num{109.81} & \num{76.81} & \num{52} & \num{37.24} \\
			Median iteration number
			& \num{147} & \num{109} & \num{76} & \num{52} & \num{33} \\			
			Average run time [s] & \num{593} & \num{458} & \num{325} & \num{213} & \num{147} \\			
			\bottomrule
	\end{tabular}
\end{table}
~\\
\indent For the choices $\varepsilon^0 \in \{25, 5, 1, \num{2e-1}, \num{4e-2}\}$,
where the final iterates were close to the globally optimal solution
of the limit problem and who are thus comparable, the number of iterations
and run times for the execution of \cref{alg:trm} decreased when starting
with a smaller value of $\varepsilon^0$. We have tabulated the mean and
average iteration numbers as well as the
average run times in \Cref{tbl:avg_run_times} for
$\varepsilon^0 \in \{25, 5, 1, \num{2e-1}, \num{4e-2}\}$.

\subsection{Control of the propagation speed in a linear wave equation}\label{sec:application_wave}

Let now $\Omega \subset \R^d$ with $d \in \{1,2,3\}$ be a bounded Lipschitz-regular domain. We next consider the control of the propagation speed in a damped wave equation. This problem falls outside of the developed theoretical framework; however, as shown later, the algorithm nevertheless performs well in numerical tests. More precisely, we consider the following problem:
\begin{equation}  \label{wave_eq} 
\left \{\begin{aligned}
&u_{tt}-\textup{div}(a(w)\nabla u)- b \Delta u_t= f \quad \text{on } \ &&\Omega \times (0,T), \\
&\frac{\partial u}{\partial n} =0 \quad \text{on } \ &&\partial \Omega \times (0,T),\\
&(u, u_t)_{\vert t=0}=(u_0, u_1). &&
\end{aligned} \right.
\end{equation}
In the context of acoustics, the wave equation given above describes propagation of sound waves through thermoviscous media. Here $u=u(x,t)$ represents the acoustic pressure, $a(w)$ corresponds to the speed of sound squared (controlled by $w$), and $b>0$ is the so-called sound diffusivity.
We assume that
\[
a(w)= c^2(1+w(x)),
\]
where $c>0$ is the reference value for the speed of sound in the medium and
\begin{equation} \label{def_Wad}
\begin{aligned}
w \in  \Wad= \{w \in \Linf: \ \uw \leq w \leq \ow \quad \textup{a.e.}  \text{ in } \Omega\}
\end{aligned}
\end{equation}
for some $\ow$, $\uw$, such that $1+ \uw > 0$. This setting is similar to the one of~\cite{clason2021optimal}, which studied coefficient-control of the undamped wave equation (i.e., with $b=0$ in \eqref{wave_eq}). The objective of the optimization process is to reach the desired pressure distribution in a focal area of interest $D \subset \Omega$:
\begin{equation*}
J(u)= \dfrac12 \int_0^T \int_D (u-u_\textup{d})^2 \, \textup{d}x \textup{d}s
\end{equation*}
for a given $u_\textup{d}\in L^2(0,T; \Ltwo)$. The admissible space for $u$ is
\begin{equation} \label{def_U}
\begin{aligned}
\calU = \{u \in L^\infty(0,T; H^1(\Omega)):& \ u_t \in  L^\infty(0,T; L^2(\Omega)) \cap L^2(0,T; H^1(\Omega)), \\  &\ u_{tt} \in L^2(0,T; \Hneg)\}.
\end{aligned}
\end{equation}
Control of a strongly damped (nonlinear) wave equation has been considered in~\cite{garcke2022phase} using a phase-field approach. We can can adapt the theoretical results from~\cite{garcke2022phase} to the present simpler setting in a straightforward manner.
\begin{proposition}[Well-posedness of the state problem]\label{Prop:WellPState} 
	Given $T>0$ and $b >0$, let $f \in L^2(0,T; \Hneg)$,  and 
	$(u_0, u_1) \in H^1(\Omega) \times L^2(\Omega)$.
	Then for every $w \in \Wad$,  there is a unique $u \in \mathcal{U}$, which solves
	\begin{equation}\label{weakformorprob}
	\begin{aligned}
	\begin{multlined}[t]
	\intO	 u_{tt}(t) v \dx+\intO (a(w) \nabla u(t)+b \nabla u_t(t)) \cdot \nabla v \dx 
	= \intO fv \dx
	\end{multlined}
	\end{aligned}
	\end{equation}
	a.e.\ in time  for all $v \in H^1(\Omega)$, with $(u, u_t)\vert_{t=0}=(u_0, u_1)$.	Furthermore, this solution satisfies the following energy estimate:
	\begin{equation} \label{energy_est_linear}
	\begin{aligned}
	\hspace{2em}&\hspace{-2em} \|u_t\|^2_{L^\infty(\Ltwo)} +\|\nabla u\|^2_{L^\infty(\Ltwo)}+ \|\nabla u_t\|^2_{L^2(\Ltwo)} \\
	&\leq \,C_1 \exp(C_2T)(\| u_0\|_{\Hone}^2+\|u_1\|^2_{\Ltwo}+\|f\|^2_{L^2(\Hneg)}),
	\end{aligned}
	\end{equation}
	where $C_{1,2}>0$ do not depend on $w$.
\end{proposition}
\begin{proof}
Unique solvability follows, for example, by~\cite[Theorem 1, Ch.\ 5]{dautray1992evolution}. Estimate \eqref{energy_est_linear} follows by testing the problem with $v=u_t(t) \in H^1(\Omega)$ and a straightforward modification of the energy arguments in, e.g.,~\cite[Proposition 3.1]{garcke2022phase}. Note that from \eqref{energy_est_linear} and the PDE, we also have the following bound:
\begin{align*}
\hspace{-2em} \|u_{tt}\|_{L^2(\Hneg)}
&\lesssim \|\nabla u\|_{L^2(\Ltwo)} + \|\nabla u_t\|_{L^2(\Ltwo)}+\|f\|_{L^2(\Hneg)} \\
&\le C_1 \exp(C_2T)\left(\| u_0\|_{\Hone}+\|u_1\|_{\Ltwo}+\|f\|_{L^2(\Hneg)}\right);
\end{align*}
see, e.g.~\cite[Chapter 7]{evans2010partial} for similar arguments.
\end{proof}
Thanks to the above well-posedness result, the operator $S: \Wad \rightarrow \calU$ is well-defined.
The next result established the differentiability of the control-to-state mapping.
\begin{proposition}\label{Prop:Cont_u_w} Under the assumptions of \cref{Prop:WellPState}, let $\tilde{u}$ and $u$ be the solutions of the state problem corresponding to the controls $\tilde{w}$ and $w$, respectively. Then
\begin{equation} \label{continuity_u_w}
\begin{aligned}
\|\tilde{u} -u\|_{\calU} \leq C(u_0, u_1, T) \|\tw -w \|_{\Linf}. 
\end{aligned}
\end{equation}
\end{proposition}
\begin{proof}
As the difference $\bar{u}=\tilde{u}-u$ solves
	\begin{equation*}
	\begin{aligned}
	\intO	 \bar{u}_{tt}(t) v \dx+\intO (a(w) \nabla \bar{u}(t)+b \nabla \bar{u}_t(t)) \cdot \nabla v \dx =\, - \intO (a(\tw)-a(w)) \nabla \tilde{u} \cdot \nabla v \ds,
	\end{aligned}
	\end{equation*}
for all $v \in H^1(\Omega)$ a.e.\ in time, with zero initial data, the statement follows by testing the above weak form with $v= \bar{u}_t(t) \in H^1(\Omega)$, integrating over time, and using the fact that $\tilde{u} \in \calU$.
\end{proof}

\begin{proposition} \label{Prop:FDifferentiability}
Under the assumptions of \cref{Prop:WellPState},  the control-to-state operator $S: \Wad \rightarrow \mathcal{U}$ is well-defined. Furthermore, it is Fr\'echet differentiable and its directional derivative at $w$ in the direction of $\xi \in L^\infty(\Omega)$ is given by $S'(w) \xi=u^{*, \xi}$, where $u^{*,\xi}$ is the unique solution of
\begin{equation}\label{diff_state_}
\begin{aligned}
\hspace{2em}&\hspace{-2em}
\intO	 u^{*,\xi}_{tt}(t) v \dx+\intO (a(w) \nabla u^{*,\xi}(t)+b \nabla u^{*,\xi}_t(t)) \cdot \nabla v \dx 
= -c^2\intO \xi \nabla S(w)(t) \cdot \nabla v \dx  
\end{aligned}
\end{equation}
a.e.\ in time for all $v \in H^1(\Omega)$, with $(u^{*,\xi}, u^{*,\xi}_t)\vert_{t=0}=(0, 0)$. Furthermore, it holds
\begin{equation} \label{est_u_*_u}
\begin{aligned}
\hspace{2em}&\hspace{-2em}
\|u^{*,\xi}\|_{\calU}  \leq C\left(T \right) \|\xi\|_{\Linf}(\|u_0\|_{\Hone}+\|u_1\|_{\Ltwo}+\|f\|_{L^2(\Hneg)}).
\end{aligned}
\end{equation}
\end{proposition}
\begin{proof}
The proof of the statement follows similarly to \cite[Proposition 4.2 and Theorem 4.1]{garcke2022phase}; the main difference is that we are in a lower-order setting compared to \cite{garcke2022phase} in terms of the regularity of data (on the other hand, we are working with a simpler equation). We provide some details here as they are relevant for discussing this problem in the context of the theoretical framework of the paper. Given $\xi \in L^\infty(\Omega)$, let \[r=S(w+\xi)-S(w)-u^{*, \xi}.\]
Note that the well-posedness of \eqref{diff_state_} follows by \cref{Prop:WellPState} with the right-hand side set to $f= \textup{div}(c^2 \xi \nabla u) \in L^2(0,T; \Hneg)$.  We can see $r$ as the solution of the following problem:	
\begin{equation}\label{rv}
\begin{aligned}
& \intO	 r_{tt} v \dx+\intO (a(w) \nabla r+b \nabla r_t) \cdot \nabla v \dx 
=\, - \intO c^2 \xi \left(S(w+\xi)-S(w)\right)\cdot \nabla v \dx
\end{aligned}
\end{equation}
for all test functions $v \in H^{1}(\Omega)$, a.e.\ in time, supplemented by zero initial conditions. To show that
\[\|r\|_X =o(\|\xi\|_{\Linf}) \quad \text{ as } \|\xi\|_{\Linf} \rightarrow 0.\]
we choose $v=r_t(t) \in \Hone$ in \eqref{rv} and estimate the resulting right-hand side as follows:
\begin{equation} \label{est_rhs_diff}
\begin{aligned}
&-\intO c^2 \xi \nabla \left(S(w+\xi)-S(w)\right)\cdot \nabla r_t \dx \\
&\lesssim \|\xi\|_{\Linf} \|\nabla( S(w+\xi)-S(w))\|_{\Ltwo} \|\nabla r_t\|_{\Ltwo}  
\end{aligned}
\end{equation}
a.e.\ in time. The statement then follows by the continuity of the mapping S established in Proposition~\ref{Prop:Cont_u_w}.
\end{proof}

Let $w \in \Wad$ and let $u=S(w) \in \mathcal{U}$ be the corresponding state. On account of
\cref{Prop:WellPState}, we can introduce the reduced functional $j(w)=J(S(w))$.
\begin{proposition}
Under the assumptions of of \cref{Prop:WellPState}, the reduced functional $j:  \Wad \rightarrow \R$ is Fr\'echet differentiable.
\end{proposition}
\begin{proof}
The proof follows similarly to \cite[Proposition 5.2]{garcke2022phase} and we thus again omit the details.
\end{proof}
Note that we have
\begin{equation}
\begin{aligned}
j'(w) \xi =&\, \int_0^T \int_\Omega (u-u_{\textup{d}}) u^{*, \xi} \dxs
\end{aligned}
\end{equation}
with $u^{*, \xi}$ given by \eqref{diff_state_}. The adjoint problem is  given by
\begin{equation} \label{adjoint}
\begin{aligned}
\begin{multlined}[t]
\intO p_{tt}{v} \dx +\intO (a(w) \nabla p(t)-b \nabla p_t) \cdot \nabla v \dx = \int_\Omega (u-u_\textup{d})v \dxs ,
\end{multlined}
\end{aligned}
\end{equation}
for all $v \in H^1(\Omega)$ a.e.\ in time, with $(p, p_t)\vert_{t=T}=(0,0)$. After time reversal
$t \mapsto p(T - t)$, its unique solvability follows by \cref{Prop:WellPState} together with the bound
\begin{equation*}
\begin{aligned}
\|p\|_{\mathcal{U}} \leq C(T) \|u-\ud\|_{L^2(L^2(\Omega))}.  
\end{aligned}
\end{equation*}
Using the adjoint problem, we then find
\begin{equation} \label{reduced_der}
\begin{aligned}
j'(w) \xi =  \begin{multlined}[t]-\int_0^T \intO a'(w)\xi \nabla u \cdot \nabla p \dxs.  \end{multlined}
\end{aligned}
\end{equation}

Looking at \eqref{est_rhs_diff}, it does not seem feasible in this setting to have an estimate involving only $\xi \in L^q(\Omega)$ for some $q < \infty$ so as to come closer to satisfying \cref{ass:general_var}. Thus this acoustic problem does not fit into our set of assumptions for the algorithmic framework so far. As a possible
alleviation, one may want to introduce a regularization of the input by means of a
mollification as is proposed in \cite{bociu2022input}. However, the continuity
with respect to $L^\infty$ is not enough to pass to the limit in the mollification.
We intend to use non-standard elliptic regularity theory to improve on the properties
of $j'$ in the future.
Nevertheless, we demonstrate in the next section that the proposed algorithm still 
shows a sensible output when applied to this setting.

\subsection{Numerical experiments for the acoustic field}\label{sec:numerical_experiments}

We next wish to provide a qualitative assessment of the algorithm in practice.
To this end, we first describe the used discretization in \cref{sec:state_equation} as well as the setup and our implementation of \cref{alg:trm}
in \cref{sec:trm_setup}. The numerical results are provided in~\cref{sec:results}.

\subsubsection{Discretization of the state and adjoint problems}\label{sec:state_equation}

We follow the numerical approach of~\cite{clason2021optimal} and adapt the implementation used there\footnote{\url{https://github.com/clason/tvwavecontrol}} to incorporate strong damping in the wave equation. 
More precisely, we employ continuous piecewise linear finite elements in space and time to discretize the state and adjoint problems. We briefly discuss the extension of the numerical discretization in~\cite{clason2021optimal} to the strongly damped acoustic setting (that is, having $b>0$ in \eqref{wave_eq}).  \\
\indent Let $\mathcal{T}_h=\{T\}$ be a mesh consisting of triangles or tetrahedra $T$ with the mesh size $h$ and let $D_h \subset H^1(\Omega) \cap C(\overline{\Omega})$ be the corresponding finite element space consisting of continuous piecewise linear functions. We further define the finite element space $D_\tau \subset H^1(0,T) \subset C[0,T]$ of continuous piecewise linear functions associated with the uniform discretization $0=t_0 < t_1 < \ldots <t_{N_{\tau}}=T$, where $\tau= t_{i+1}-t_i$, $i \in \{0, \ldots, N_{\tau}-1\}$. Let $\{e_i\}$ be the basis (hat) functions of $D_\tau$.\\
\indent Let $\nu := (h, \tau)$ and $D_{\nu}:=D_h \otimes D_{\tau}$. We look for the approximate solution $u_{\nu} \in D_{\nu}$, which satisfies
\begin{equation} \label{discrete_spacetime_problem}
\begin{aligned}
\begin{multlined}[t]      \int_0^T \intO \Big\{- \partial_t u_{\nu} \partial_t v -(\sigma-\frac16) \tau^2 a(w_h) \nabla \partial_t u_{\nu}\cdot \nabla \partial_t v 
+ a(w_h) \nabla u_{\nu} \cdot \nabla v \\ \hspace*{4cm}+ b \nabla \partial_t u_{\nu} \cdot \nabla v \Big\}\dxs 
= \intO u_1 v(0) \dx + \int_0^T \int_{\Omega} f v \dxs
\end{multlined}
\end{aligned}
\end{equation}
for all $v \in D_{\nu}$ with $v(T)=0$, and $u_\nu(0)= P_0 u_0$, where $P_0$ is the $\Ltwo$ projection operator defined by
\begin{equation}
(P_0 u_0, \varphi)_{\Ltwo} = (u_0, \varphi)_{\Ltwo} \quad \text{for all }\ \varphi \in D_h.
\end{equation}
In \eqref{discrete_spacetime_problem}, $\sigma>0$ denotes the stabilization parameter; cf.~\cite[Def.\ 5.1]{clason2021optimal}. Problem \eqref{discrete_spacetime_problem} can be restated as a time-stepping scheme using the fact that
\begin{equation*}
\begin{aligned}
u_{\nu}(t)= \begin{cases}
\dfrac{t_i-t}{t_{i}-t_{i-1}}u_h(t_{i-1}) +\dfrac{t-t_{i-1}}{t_{i}-t_{i-1}}u_h(t_{i}), \quad  \ t\in [t_{i-1}, t_{i}],\\[2mm]
\dfrac{t_{i+1}-t}{t_{i+1}-t_{i}}u_h(t_{i}) +\dfrac{t-t_{i}}{t_{i+1}-t_{i}}u_h(t_{i+1}), \quad t \in [t_{i}, t_{i+1}], 
\end{cases}  
\end{aligned}
\end{equation*}
and expressing the test function as $v = \displaystyle \sum_{i=0}^{N_{\tau}-1} e_i(t) \phi_i(x)$. Indeed, let $u_h^0= P_0 u_0$. We first compute $u_h^1$ using the fact that it solves
\begin{equation}
\begin{aligned}
\begin{multlined}[t]    \left (\frac{u_h^1-u_h^0}{\tau}, \phi\right)_{\Ltwo}+ \tau (a(w_h) \nabla (\sigma u_h^1+(\frac12-\sigma) u_h^0), \nabla \phi)_{\Ltwo} \\
\hspace*{2.7cm} + \frac{b}{2} (\nabla (u_h^1-u_h^0), \nabla \phi)_{\Ltwo} 
= (u_1, \phi)_{\Ltwo}+ \left(\int_0^{t_1} f e_0, \phi\right)_{L^2(\Omega)}
\end{multlined}  
\end{aligned}    
\end{equation}
for all $\phi \in D_h$. Then given $(u_h^{i-1},u_h^i)$, we can compute $u_h^{i+1}$ since it solves
\begin{equation} \label{time_stepping}
\begin{aligned}
\begin{multlined}[t]    \left(\frac{u^{i+1}_h-2u_h^i+u_h^{i-1}}{\tau}, \phi\right)_{\Ltwo} + \tau \left(a(w_h) \nabla (\sigma u_h^{i+1}+ (1-2\sigma) u_h^i+ \sigma u_h^{i-1}, \nabla \phi \right)_{\Ltwo} \\
+ \frac{b}{2} \left(\nabla (u_{h}^{i+1} - u_h^{i-1}), \nabla \phi\right)_{\Ltwo} = \left(\int_{t_{i-1}}^{t_{i+1}} f e_i\ds, \phi\right)_{L^2(\Omega)}
\end{multlined} 
\end{aligned}
\end{equation}
for $1 \leq i \leq N_{\tau}-1$ and all $\phi \in D_h$. \\
\indent Concerning the adjoint problem, discrete adjoint state $p_{\nu}= \displaystyle \sum_{i=1}^{N_\tau}p_h^i(x) e_i(t) \in D_{\nu}$  satisfies the following problem:
\begin{equation} \label{discrete_spacetime_adjoint_problem}
\begin{aligned}
\begin{multlined}[t]      \int_0^T \intO \Big\{- \partial_t v \partial_t p_{\nu} -(\sigma-\frac16) \tau^2 a(w_h) \nabla \partial_t v \cdot \nabla \partial_t p_{\nu}
+ a(w_h) \nabla v \cdot \nabla p_{\nu} \\ \hspace*{5cm}+ b \nabla \partial_t v \cdot \nabla p_{\nu} \Big\}\dxs 
=  \int_0^T \intO (u_{\nu}-u_{\textup{d}}) v \dxs
\end{multlined}
\end{aligned}
\end{equation}
for all $v \in D_{\nu}$ with $v(0)=0$, and $p_{\nu}(T)=0$. This problem can  be reformulated as a time-stepping scheme analogously to the state equation. Indeed, starting from $p_h^{N_\tau}=0$, we compute $p_h^{N_{\tau-1}}$ using the fact that it solves
\begin{equation}
\begin{aligned}
    \begin{multlined}[t]    -\left (\frac{p_h^{N_\tau}-p_h^{N_\tau-1}}{\tau}, \phi\right)_{\Ltwo}+ \tau (a(w_h) \nabla (\sigma p_h^{N_\tau-1}+(\frac12-\sigma) p_h^{N_\tau}), \nabla \phi)_{\Ltwo} \\
    + \frac{b}{2} (\nabla (p_h^{N_\tau}+p_h^{N_\tau-1}), \nabla \phi)_{\Ltwo} 
    =  \int_{N_\tau-1}^{N_\tau} \intO (u_{\nu}-u_{\textup{d}}) e_{N_\tau} \dxs.
    \end{multlined}  
\end{aligned}    
\end{equation}
Given $(p_h^{i+1}, p_h^i)$, we can compute $p_h^{i-1}$ using the fact that
\begin{equation} \label{time_stepping_adjoint}
    \begin{aligned}
   \begin{multlined}[t]    \left(\frac{p^{i+1}_h-2p_h^i+p_h^{i-1}}{\tau}, \phi\right)_{\Ltwo} + \tau \left(a(w_h) \nabla (\sigma p_h^{i+1}+ (1-2\sigma) p_h^i+ \sigma p_h^{i-1}, \nabla \phi \right)_{\Ltwo} \\
  - \frac{b}{2} \left(\nabla (p_{h}^{i+1} - p_h^{i-1}), \nabla \phi\right)_{\Ltwo} = \left(\int_{t_{i-1}}^{t_{i+1}} (u_\nu-u_{\textup{d}}) e_i\ds, \phi\right)_{L^2(\Omega)}
    \end{multlined} 
    \end{aligned}
\end{equation}
for $1 \leq i \leq N_{\tau}-1$ and all $\phi \in D_h$.
The considerations concerning the properties of the discrete control-to-state operator follow analogously to Section 5.1 in \cite{clason2021optimal}, so we omit them here.

\subsubsection{Setup}\label{sec:trm_setup}

We consider a two-dimensional rectangular spatial domain $\Omega = (-1,1) \times (-1,2)$ that are discretized into $96 \times 96$ squares, which are subdivided
into two triangles each and a time horizon from $0$ to $T = 5$
that is discretized into $256$ intervals. We choose $c^2 = 20$
and $b = 1.25 \cdot 10^{-2}$. 
We choose CG1 elements as ansatz functions for the state vector
as well as the control vector.
We use the \texttt{FEniCSx} library \cite{ScroggsEtal2022,BasixJoss,AlnaesEtal2015}
and the space-time discretization from \cite{clason2021optimal} described in the section above
in order to solve the PDE and its adjoint.
We use a tracking-type functional and choose $\gamma = 7.5 \cdot 10^{-6}$ in
the objective. In order to approximate the $L^1$-norm in the trust-region 
subproblems, we compute the $L^1$-norm of the function that is obtained by 
reflecting the negative node values (our ansatz has a nodal basis) at the
origin. We choose the acoustic source term $f$ as the Ricker wavelet as
in \cite{clason2021optimal} and set $\sigma = 0.25$.

Regarding \cref{alg:trm}, as in \cref{sec:application_elliptic_control},
we only solve instances of the convex subproblems 
\eqref{eq:tr_cvx}. We initialize $\varepsilon^0 = 1$, $\Delta^0 = 1.5$,
$\rho = 10^{-4}$, $\kappa^0 = 10^{-8}$, and
$\underline{\Delta}^0 = 1.14 \cdot 10^{-5}$. The initial control $w^0$
is the constant function with the value $0.5$.

We execute our algorithm on a laptop computer
with an Intel(R) Core(TM) i7-11850H CPU (2.50\,GHz) and 64\,GB RAM.
We solve the discretized subproblems with Gurobi 10.0.0 \cite{Gurobi}.
We stop the algorithm when no futher progress is made, that is $\varepsilon^n$
is reduced without any step having been accepted for the previous value of 
$\varepsilon^n$. 

\subsubsection{Results}\label{sec:results}

The execution of our implementation requires 174 iterations, of which 62 are accepted, which are broken down for the different values of 
$\varepsilon^n$ in \cref{tbl:accepted_iterations_per_varepsilon}. The number of accepted iterations is between 11 and 27 for the first
three values of $\varepsilon^n$ drops to $5$ for the fourth value. For the fifth value $\varepsilon^n = 1.6 \cdot 10^{-3}$,
the trust-region immediately contracts because it is already close to stationarity for $\varepsilon^n$ and no improvement can be made from last
accepted iterate for the previous value of $\varepsilon^n$. This is closely linked to the fact that we only use convex subproblems and the iterate
is already very close to being binary-valued. Hence the objective of the convex subproblem exhibits a high penalty that is proportional to
$\tfrac{1}{\varepsilon}$ on changes of the nodal values from $0$ to $1$ or vice versa. In other words, the objective gradient points to the
outside of the feasible set, implying the stationarity. A further reduction of $\varepsilon^n$ lead to the same result for this test case.
We also note that handling the coefficients in the trust-region subproblem becomes more difficult for smaller values of $\varepsilon^n$
due to the opposed scaling effect of $\varepsilon$ on the two terms in $E_\varepsilon$.
The running time of the execution of the algorithm was 2 hours and 18 minutes.

While it is difficult to measure instationarity with respect
to \eqref{eq:p} directly, we compute the predicted reduction of the solved 
trust-region subproblems \eqref{eq:tr_cvx} for the reset trust-region 
radius $\Delta^0$ for the initial and final iterate for each of the
values of $\varepsilon^n$ in order to have a surrogate and get some 
impression on the behavior of the algorithm. 
We observe a decreasing (but not monotonic) trend for the initial
predicted reduction over the different values of $\varepsilon^n$.
The predicted reduction drops the first to the second value of $\varepsilon$ by almost three orders of magnitude and
then remains relatively constant over the next two reductions of $\varepsilon$ and then drops almost by another order of magnitude.

As one expects from gradient-based optimization algorithms, the value
of the predicted reduction is much smaller for the final iterate for the 
four values of $\varepsilon$ for which iterates are accepted. We tabulate these values in the \cref{tbl:accepted_iterations_per_varepsilon} too.
\begin{table}[ht]
	\caption{Number of accepted iterations and initial predicted reduction
		for the reset trust-region radius $\Delta^0 = 1.5$ for different
		values of $\varepsilon^n$.}\label{tbl:accepted_iterations_per_varepsilon}
	\begin{center}
		\begin{tabular}{r|ccc}
			\toprule 
			&&
			\multicolumn{2}{c}{Predicted reduction for $\Delta = 1.5$} \\
			$\varepsilon^n$
			& Accepted iterates 
			& initial
			& final
			\\ \midrule 
			\num{1.0e+00}
			& \num{19} 
			& \num{1.162e-02}
			& \num{4.10e-08}
			\\
			\num{2.0e-01}
			& \num{11} 
			& \num{7.687e-06}
			& \num{4.99e-09}
			\\
			\num{4.0e-02}
			& \num{27}
			& \num{1.139e-05}
			& \num{2.62e-09}
			\\ 			
			\num{8.0e-03}
			& \num{5} 
			& \num{8.962e-06}			
			& \num{1.27e-08}
			\\ 		
			\num{1.6e-03}
			& -
			& \num{9.023e-07}			
			& \num{9.023e-07}
			\\ 						
			\bottomrule
		\end{tabular}
	\end{center}
\end{table}
Over the course of the iterations, the binarity of the (accepted) iterates
$w^n$ increases significantly and a sharp interface emerges.
We have visualized this in \cref{fig:sharp_interface_emerges}, where
we display the first and last accepted iterates for each of the values
$\varepsilon^n$.
\begin{figure}
	\centering
	\begin{subfigure}{.25\textwidth}
		\centering
		\includegraphics[width=\linewidth]{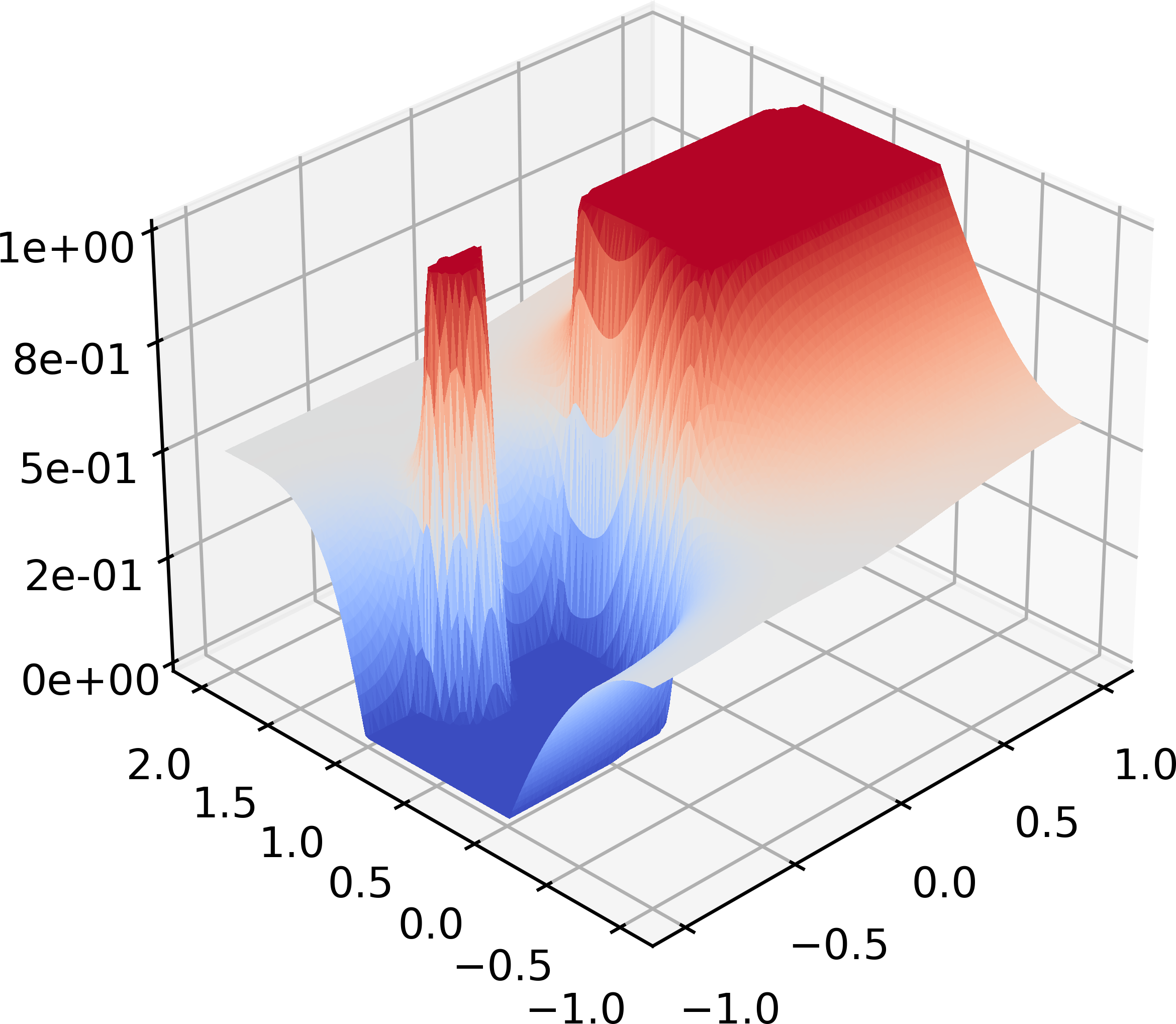}
		\caption{$\varepsilon^n = \num{1.0e+00}$}
	\end{subfigure}%
	\begin{subfigure}{.25\textwidth}
		\centering
		\includegraphics[width=\linewidth]{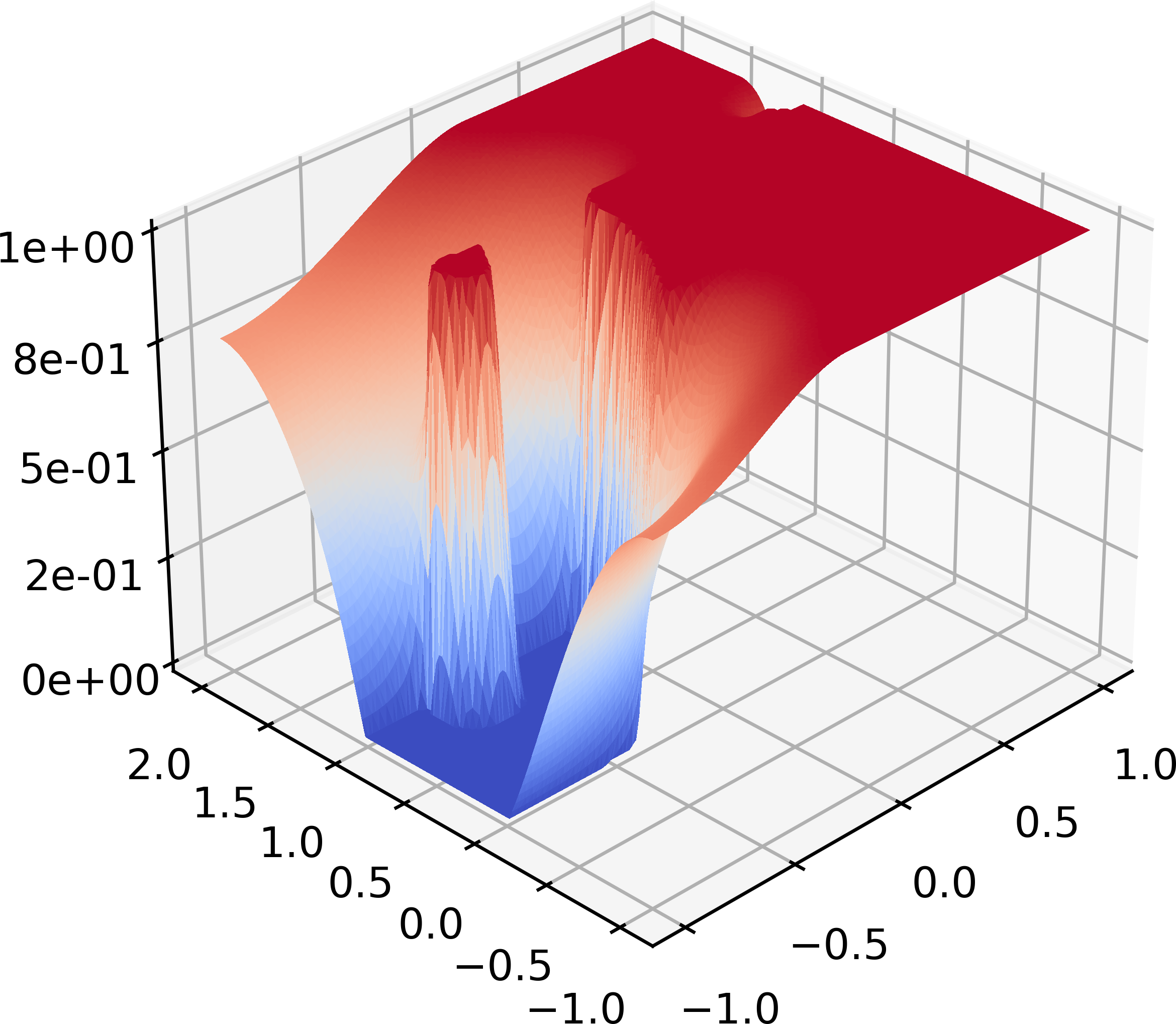}
		\caption{$\varepsilon^n = \num{1.0e+00}$}		
	\end{subfigure}%
	\begin{subfigure}{.25\textwidth}
		\centering
		\includegraphics[width=\linewidth]{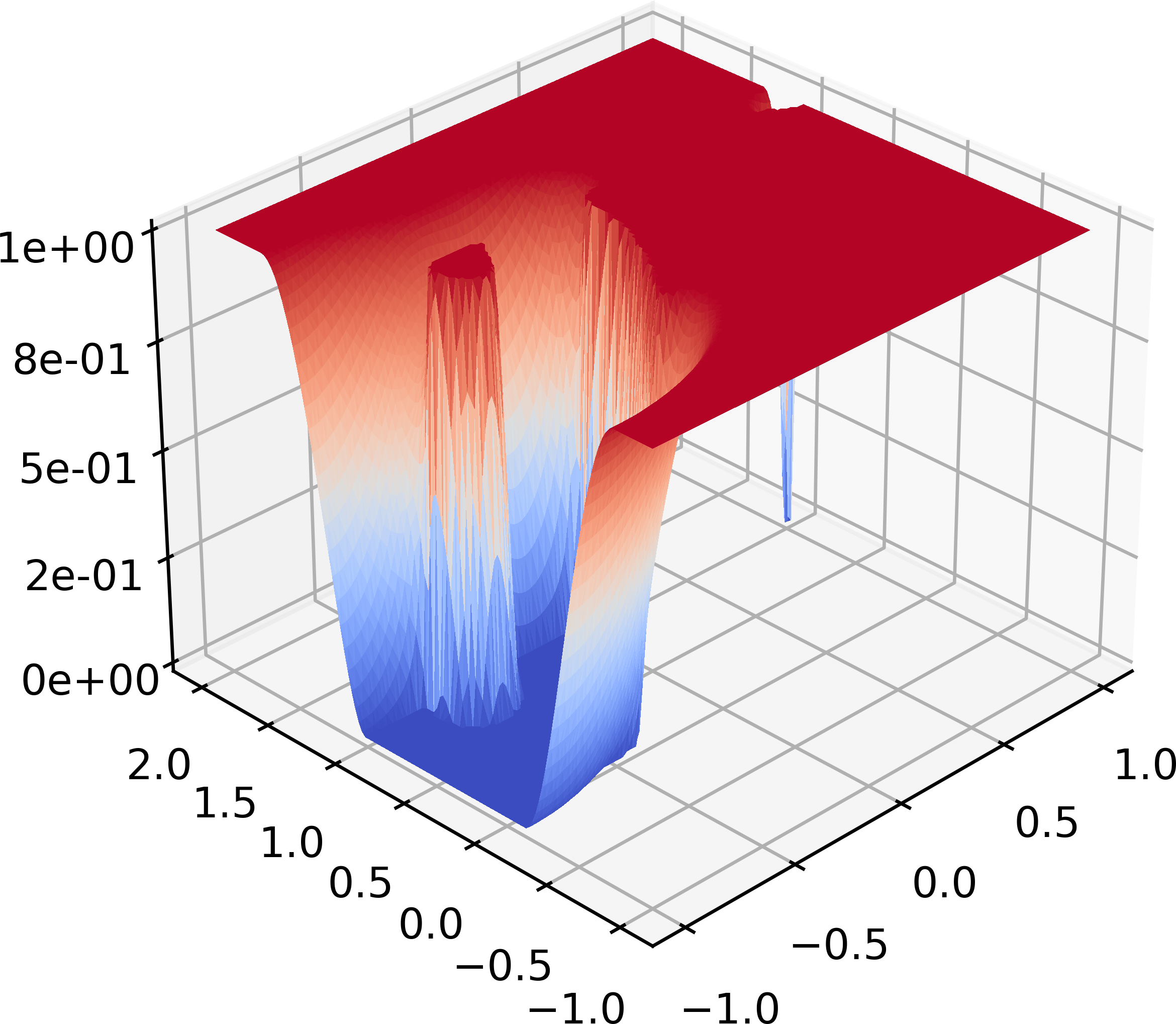}
		\caption{$\varepsilon^n = \num{2.0e-01}$}		
	\end{subfigure}%
	\begin{subfigure}{.25\textwidth}
		\centering
		\includegraphics[width=\linewidth]{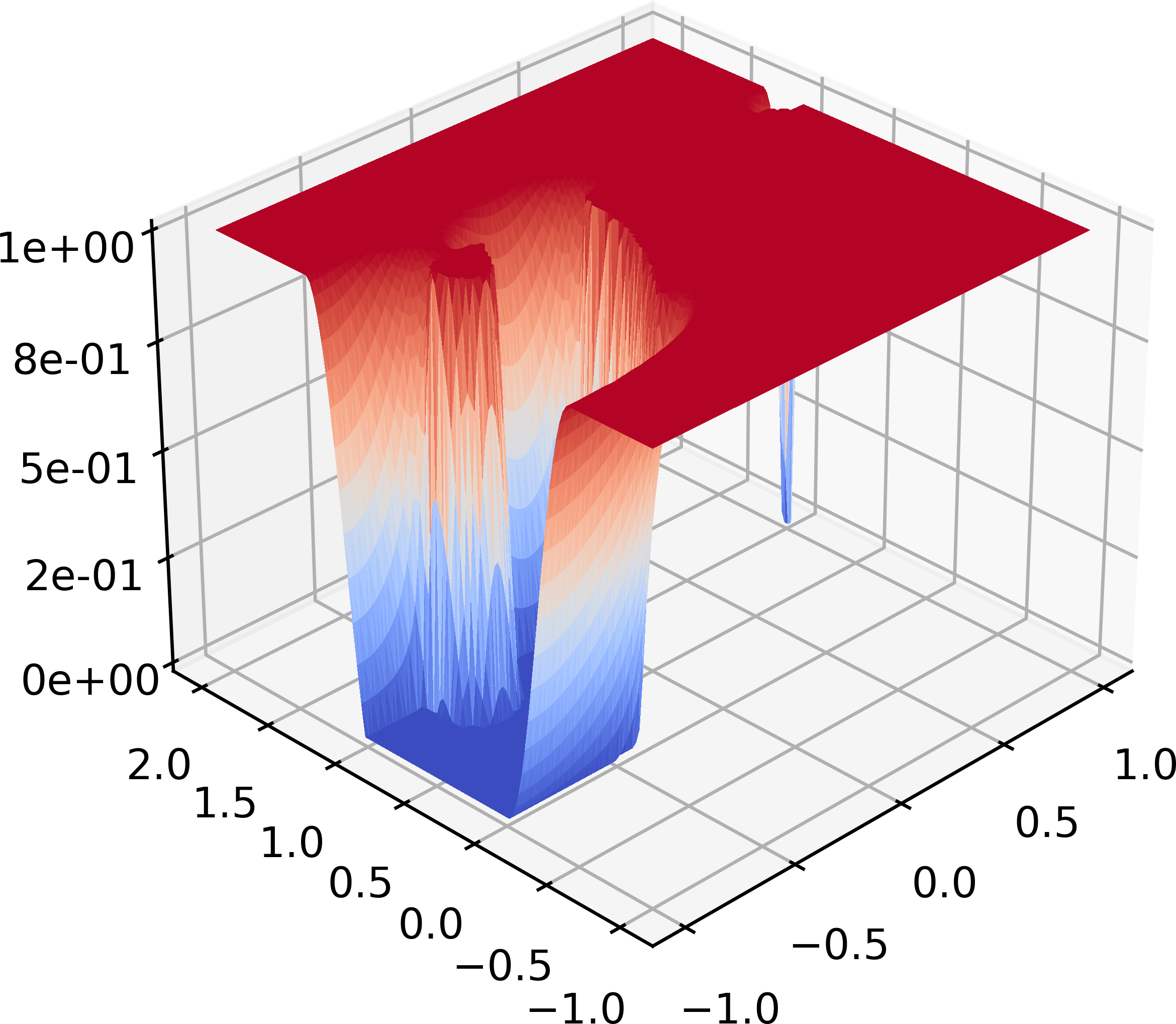}
		\caption{$\varepsilon^n = \num{2.0e-01}$}		
	\end{subfigure}\\
	\begin{subfigure}{.25\textwidth}
		\centering
		\includegraphics[width=\linewidth]{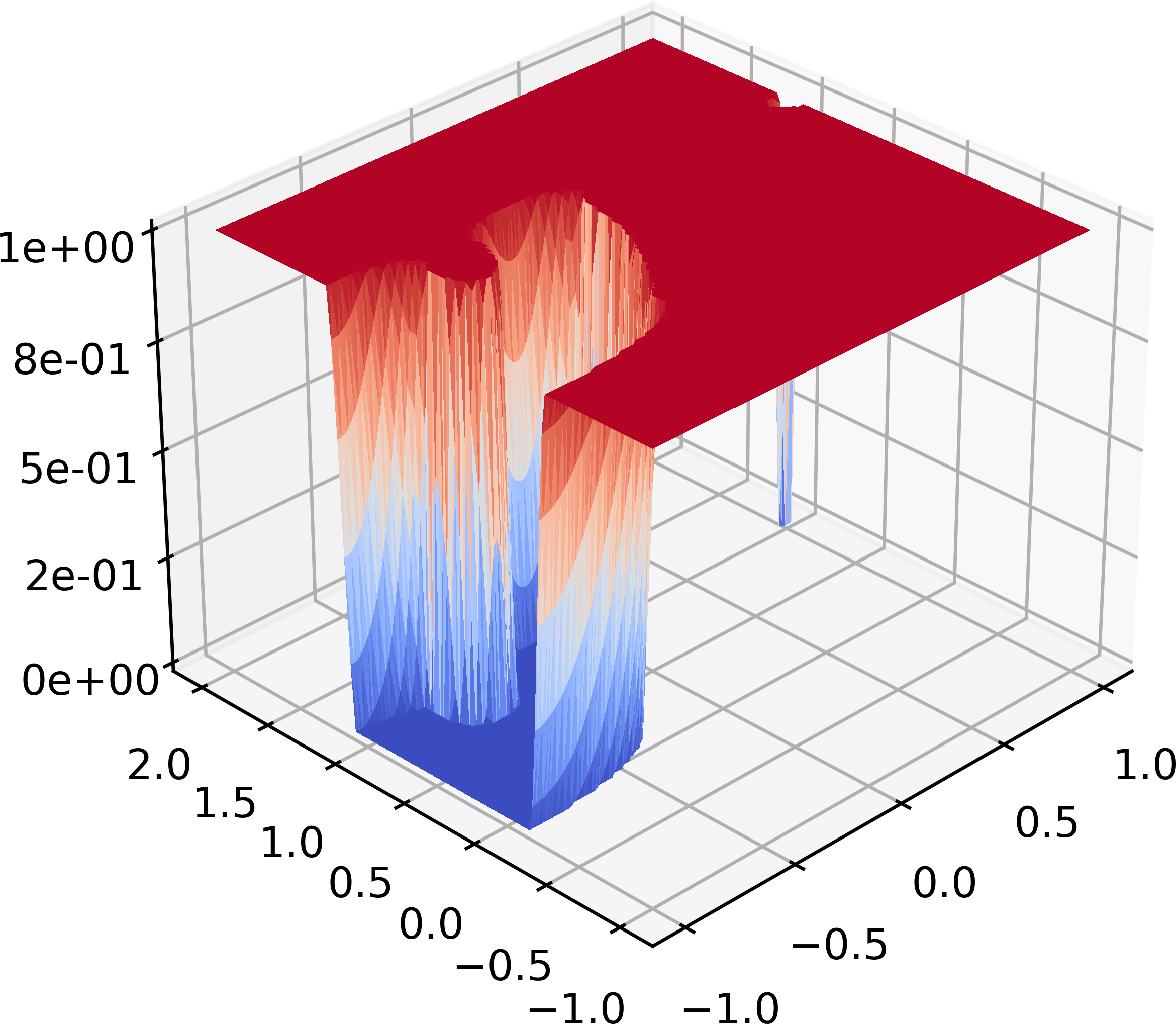}
		\caption{$\varepsilon^n = \num{4.0e-02}$}		
	\end{subfigure}%
	\begin{subfigure}{.25\textwidth}
		\centering
		\includegraphics[width=\linewidth]{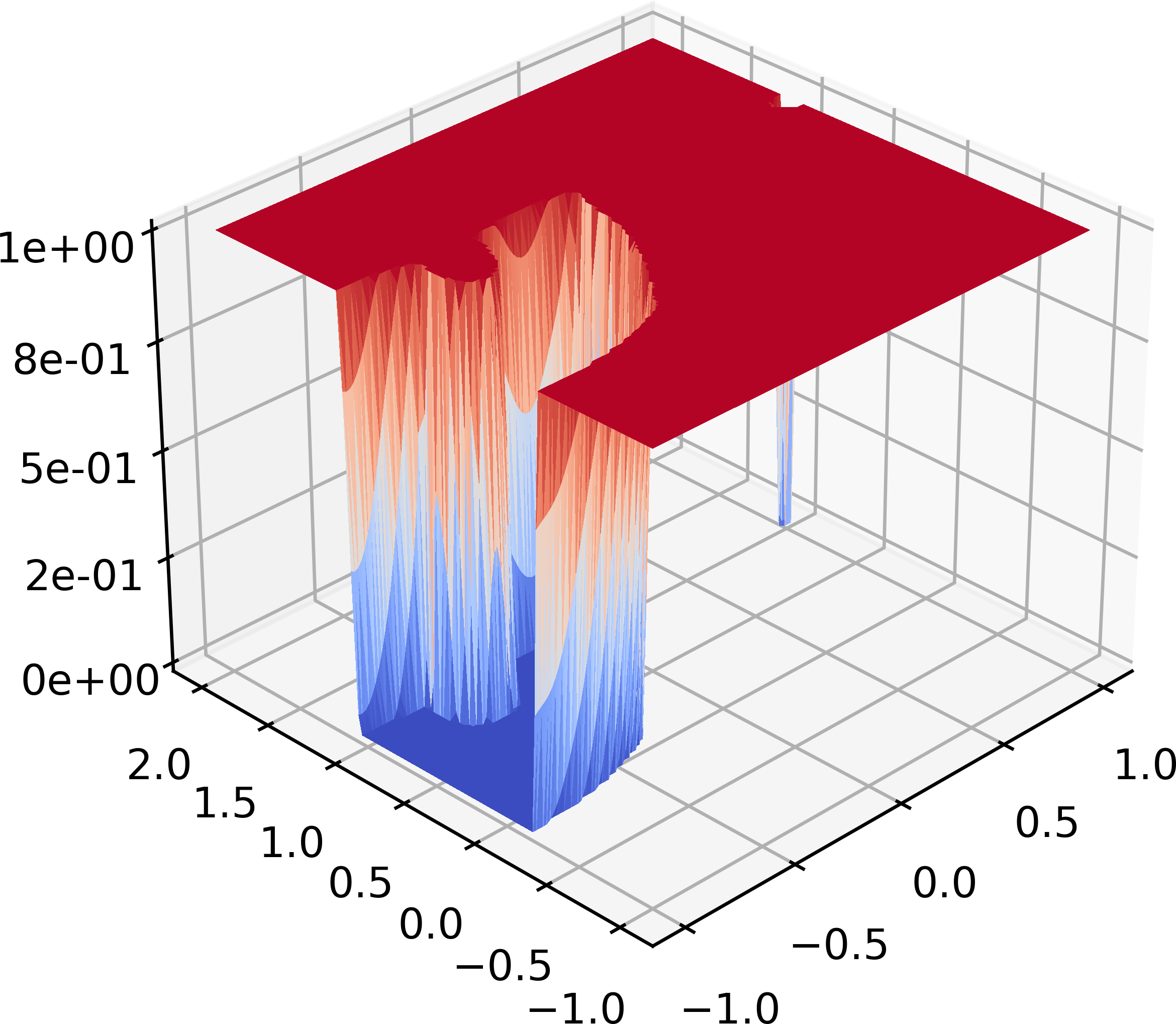}
		\caption{$\varepsilon^n = \num{4.0e-02}$}		
	\end{subfigure}%
	\begin{subfigure}{.25\textwidth}
		\centering
		\includegraphics[width=\linewidth]{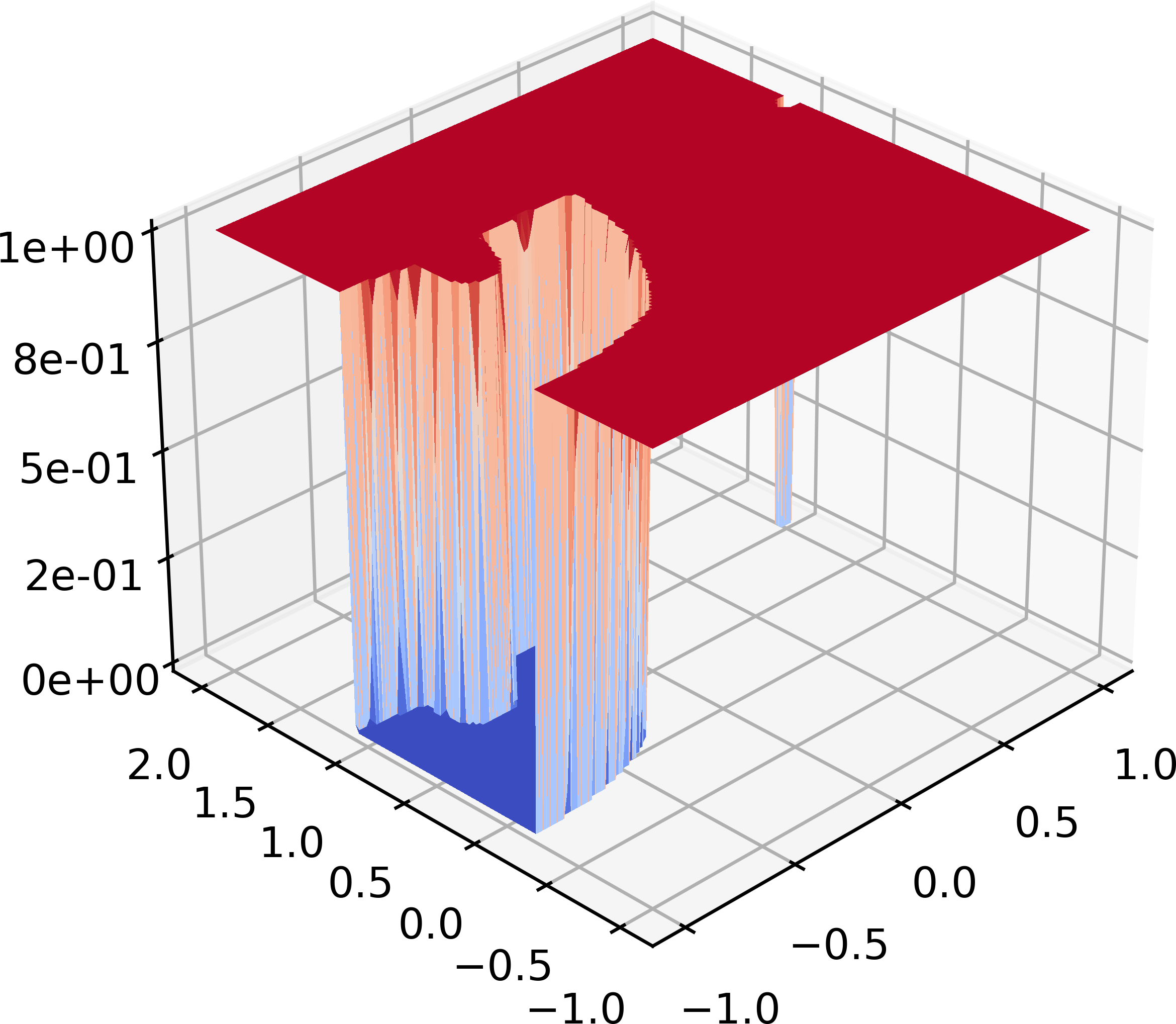}
		\caption{$\varepsilon^n = \num{8.0e-03}$}		
	\end{subfigure}%
	\begin{subfigure}{.25\textwidth}
		\centering
		\includegraphics[width=\linewidth]{iterations/initial_8.00e-03.png}
		\caption{$\varepsilon^n = \num{8.0e-03}$}		
	\end{subfigure}
	\caption{First (a, c, e, g) and last accepted iterate (b, d, f, h) 
		over the decreasing values
		of $\varepsilon^n$.}
	\label{fig:sharp_interface_emerges}
\end{figure}

\section{Conclusion}\label{sec:conclusion}

We have provided an algorithmic phase-field approach to a class of optimal control
problems with a restriction to binary control functions with perimeter regularization.
We have shown that the algorithmic idea of \cite{leyffer2022sequential,manns2022on}
can be  transferred to the phase-field setting and a similar convergence theory
can be obtained, where the L-stationarity of the limits requires an additional geometric 
assumption, see \cref{thm:gamma_to_stat}, which we believe cannot be dropped easily without
replacement and which we interpret as a constraint qualification in an 
analogy to nonlinear programming. Crucially, we have derived a relationship
between the reduction of the interface parameter and the reduction of the lower bound on the trust-region radius that in turn triggers reduction of 
the interface parameter.

As our applications and computational examples demonstrate, this work 
motivates research on several aspects of numerical optimization of
problems of type \eqref{eq:p} by means of the analyzed algorithmic 
framework. We highlight a few of them below.

\Cref{ass:general_var} requires that the reduced objective $j = J \circ S$
is Fr\'{e}chet differentiable as a function from $L^1(\Omega)$ to $\R$.
We cannot satisfy this assumption for our motivating application
and only get differentiability with respect to $L^\infty(\Omega)$ so far,
see \cref{sec:application_wave}.
This is not satisfactory because for any two binary-valued functions
that differ on a set of strictly positive Lebesgue measure,
their $L^\infty(\Omega)$-difference is one. This motivates research
on approximations of the PDE so that the assumption can be satisfied
as well as research on improving the \cref{ass:general_var}
so that differentiability is only required with respect to
some $L^p(\Omega)$, $p > 1$.

A crucial step in our convergence analysis is \cref{thm:tr_gamma_convergence}, in which
we prove $\Gamma$-convergence of the non-convex trust-region subproblems for a fixed
trust-region radius. \Cref{ex:no_gamma_convergence_for_cvx_tr} shows that this result
does not hold in the same way for the convex trust-region subproblems. Therefore, in order
to reduce the  computational burden that is required for convergence guarantees, we find it important to seek alternatives to the statement of \cref{thm:tr_gamma_convergence}
in the convergence analysis that would allow us to work with convex subproblems.
An alternative avenue is of course to develop efficient solution strategies
for the non-convex subproblems. 

\section*{Acknowledgments}
The authors are grateful to Matthias R\"{o}ger (TU Dortmund University) for helpful advice on the construction
in \cref{prp:boundedness_involved}. The authors are also grateful to Annika Schiemann (TU Dortmund University)
for suggesting and guiding the construction and implementation of the exact solution in \cref{sec:source_exact_solution}.
The authors thank two anonymous referees for helpful feedback on the manuscript.

\bibliography{references}{}
\bibliographystyle{siam} 
\end{document}